\documentclass[a4paper,12pt]{article}

\newcommand\blfootnote[1]{%
  \begingroup
  \renewcommand\thefootnote{}\footnote{#1}%
  \addtocounter{footnote}{-1}%
  \endgroup
}
\usepackage{color}
\usepackage[normalem]{ulem}

\usepackage{amsmath,amsfonts,amssymb,amsthm,amscd}

\font\cmssl=cmss10 at 12 pt  

\newtheorem{thm}{Theorem}
\newtheorem{lem}[thm]{Lemma}
\newtheorem{prop}[thm]{Proposition}
\newtheorem{defn}[thm]{Definition}
\newtheorem{cor}[thm]{Corollary}
\newtheorem{rem}[thm]{Remark}
\newtheorem{notation}[thm]{Notation}
\newtheorem{exa}[thm]{Example}

\begin{document}

\title{Twist, elementary deformation, and the KK correspondence in generalized complex geometry}
\author{Vicente Cort\'es and Liana David}
\maketitle
\begin{abstract}
\noindent
We define the operations of conformal change and elementary deformation
in the setting of generalized complex geometry. Then we apply Swann's twist construction
to generalized (almost) complex and Hermitian structures obtained by these operations
and establish necessary and sufficient conditions for the Courant integrability of the resulting 
twisted structures. In particular,  we associate to any appropriate generalized K\"{a}hler manifold $(M, G, \mathcal J)$  
with a  Hamiltonian Killing vector field a new generalized K\"{a}hler manifold, depending on the choice of a pair of non-vanishing functions and compatible twist data. 
We study this construction when $(M, G, \mathcal J )$ is (diagonal) toric, with emphasis on the four-dimensional case. 
In particular, we apply it to  deformations of the standard flat K\"{a}hler metric on $\mathbb{C}^{n}$, the Fubini-Study
K\"{a}hler metric on $\mathbb{C}P^{2}$ and the so called admissible K\"{a}hler metrics on Hirzebruch surfaces.  As a further application, we  recover the KK (K\"{a}hler-K\"{a}hler) correspondence, which is obtained by 
specializing to the case of an ordinary K\"ahler manifold.
\end{abstract}

\newpage

\tableofcontents

\section{Introduction}

Swann's twist construction is a powerful method to construct
new manifolds with geometrical structures from given 
ones \cite{swann}.
It was applied successfully 
to construct explicit examples of 
compact SKT, hyper-complex and HKT manifolds.   Combined with the so-called elementary deformation of hyper-K\"{a}hler structures,  it provides an elegant approach to the 
HK/QK (hyper-K\"{a}hler/quaternionic-K\"{a}hler) correspondence
\cite{swann-short}. The HK/QK correspondence basically  associates  to a hyper-K\"{a}hler manifold with a certain type of Killing vector field
a quaternionic-K\"{a}hler manifold of the same dimension.
It was introduced  by Haydys 
in \cite{haidis}  (see also \cite{hitchin-paper} for recent related work) and was
extended to allow for indefinite hyper-K\"{a}hler metrics  in \cite{ACM, ACDM},  
without losing control over the signature of the resulting quaternionic-K\"{a}hler metrics, for which a 
simple formula was given in \cite{ACDM}. This result was applied 
in \cite{ACDM} to prove that the one-loop quantum correction of the supergravity c-map metric is quaternionic-K\"ahler, leading to the first completeness results for the quaternionic-K\"{a}hler metric in the HK/QK correspondence (see \cite{CDS} for the state of the art). The approach from \cite{ACM,ACDM}
yields also 
the so-called KK (K\"{a}hler-K\"{a}hler) correspondence.
This is a method to associate to a K\"{a}hler manifold with a Hamiltonian Killing vector field a new K\"{a}hler manifold, of the same dimension.

Despite a rich literature on the twist construction in complex and Hermitian geometry, 
applications   in generalized complex geometry 
seem to be missing in the literature. In this paper we 
fill 
this gap, by studying how various notions from 
generalized complex geometry behave under the  twist construction. 
Our main goal is to extend, 
 in the spirit of \cite{swann-short}, 
the KK correspondence to the setting of generalized K\"{a}hler geometry.\

The plan of the paper is the following. Section \ref{preliminary} is  mainly  intended to fix notation. Here we recall Swann's  twist construction 
(see Section
\ref{twist-section})
and  the material  we need from generalized complex geometry (see 
Section
\ref{complex-geometry}).

In Section \ref{algebraic} we define 
two basic algebraic tools for this paper: 
the conformal change 
and elementary deformation in the setting of generalized complex geometry
(see Definitions  \ref{def-conf} and \ref{def-elem}). Our definition of elementary deformation is 
 inspired by \cite{swann-short}, where 
 elementary deformations of hyper-K\"{a}hler structures were introduced.
The integrability of a generalized K\"{a}hler structure is encoded in 
the Courant integrability of $L_{1}$ and of the intersections   $L_{1} \cap L_{2}$ and $L_{1} \cap \bar{L}_{2}$, where $L_{i}$ ($i=1,2$)  are the  $(1,0)$-bundles of its two generalized complex structures.  
We study how these intersections vary under elementary deformations of generalized almost Hermitian structures 
(see Lemma \ref{elem-bundles}).

In Section \ref{twist} we study how the Courant bracket behaves under the twist construction and we determine conditions on the twist data which 
ensure
that a
generalized almost complex structure $\mathcal J$ becomes  integrable under  this construction
(see Theorem \ref{cond-int-complex}).
In Section \ref{examples-twist} we present various particular cases:  when
$\mathcal J$  is a  symplectic 
structure (this  was already considered in \cite{swann});   when $\mathcal J$ 
is the deformation of a complex structure by a Poisson bivector $\Pi$ (when $\Pi=0$ this was also
considered in \cite{swann}); when $\mathcal J$   
is  the  interpolation between  a complex
and a symplectic structure; and finally, when $\mathcal J$ is a conformal change of a generalized almost complex structure  (we shall encounter this situation later in the paper).

In Section \ref{KK-section} we develop  what we call the
generalized KK correspondence. Its  statement in  highest generality is
Theorem \ref{gen-KK}.  Let $(M,G, \mathcal J )$ be a generalized K\"{a}hler manifold with a Hamiltonian Killing vector field $X_{0}.$ 
Let $f, h\in C^{\infty}(M)$ be two non-vanishing functions, $(G^{\prime}, \mathcal J )$ the elementary deformation of $(G, \mathcal J )$ by $X_{0}$ and $f$ and 
$\tau_{h} ( G^{\prime},  \mathcal J )$ the conformal change of $(G^{\prime}, \mathcal J )$ by $h$.  Theorem \ref{gen-KK} expresses the conditions which 
ensure 
that the twist 
$[\tau_{h} ( G^{\prime}, \mathcal J )]_{W}$ of $\tau_{h} ( G^{\prime},  \mathcal J )$  is generalized K\"{a}hler. 
The idea of the proof is the  following. As a first stage, 
we determine   in Section \ref{ah-section} conditions which ensure
that the twist of a generalized almost Hermitian manifold  is generalized
K\"{a}hler (see Theorem \ref{prop-twist}). Then 
 we compute various useful Courant brackets in Section \ref{brackets-section}, 
 using  in an essential way that $X_{0}$ is a Hamiltonian Killing
vector field.  In Sections \ref{invol-sect} and \ref{diff-section} we apply 
Theorem \ref{prop-twist}   to the generalized almost Hermitian manifold $\tau_{h} (G^{\prime}, \mathcal J )$ and we conclude the proof of Theorem \ref{gen-KK}.

In Section \ref{examples} we  develop  examples of the generalized KK correspondence. 
As a first particular case of Theorem \ref{gen-KK}  we recover in our  setting  the KK correspondence  of 
\cite{ACM}   (see  Proposition \ref{concise}). 
After a short discussion on various   conditions from Theorem \ref{gen-KK} and their
impact on the generalized K\"{a}hler manifold $(M, G, \mathcal J)$,   we apply
the generalized KK correspondence under the additional assumption that 
$\mathcal J_{3} X_{0}\in \Omega^{1}(M)$
(see Proposition \ref{gen-applied-1}).  
 In Section \ref{toric-section}, after a brief review of the theory of toric generalized K\"{a}hler manifolds
developed in \cite{boulanger}, we show that 
 examples of  (non-K\"{a}hler) generalized K\"{a}hler  manifolds with such Hamiltonian Killing vector fields   exist  in the toric setting (see Examples \ref{pr-t-ex}). In Section \ref{4dim-sect}
we treat in detail the generalized KK correspondence when $(M, G, \mathcal J )$ is toric and four-dimensional
(see Proposition \ref{dim2}, Corollaries \ref{corolar-1} and \ref{corolar-2}). 
In particular, we apply it to a suitable $1$-parameter family of deformations
of the Fubini-Study K\"{a}hler structure on $\mathbb{C}P^{2}$ and of  the class of admissible K\"{a}hler structures on Hirzebruch surfaces $\mathbb{F}_{k}$ (see Examples \ref{cp} and \ref{hz}).  

It would be interesting to study the properties of the generalized K\"{a}hler structures produced by the generalized KK correspondence. Owing to the length of this paper, this will be postponed for a future project. Some preliminary observations  in this direction are presented in Remark \ref{future}.

\section{Preliminaries}
\label{preliminary}

We begin by fixing various conventions we shall use in the paper. 
We work in the smooth setting. 
All our manifolds, vector bundles, functions, tensor fields, etc.\ are smooth.  
We denote by $\Gamma (E)$ the space of  (smooth) sections of a vector bundle $E\rightarrow M.$ 
For every manifold $M$, we denote by $\mathrm{pr}_{T}$ and $\mathrm{pr}_{T^{*}}$ the natural projections from 
the generalized tangent bundle $\mathbb{T}M= TM \oplus T^{*}M$ onto its components $TM$ and $T^{*}M$.
The Lie derivative in the direction of a vector field $X$ is denoted by $\mathcal L_{X}$. For a real form
$\alpha \in \Omega^{k}(M)$, we use the same notation $\alpha$ for its complex linear extension
to any (subbundle) of $(TM)^{\mathbb{C}}.$   For a non-degenerate $2$-form $\omega \in \Omega^{2}(M)$,  $\omega : TM \rightarrow T^{*}M$ is the map  $X\rightarrow i_{X}\omega $.   Similarly,
$g : TM \rightarrow T^{*}M$, $X \rightarrow g(X, \cdot )$ denotes  the Riemannian duality defined by a Riemannian metric $g$.
 In our conventions,  $J^{*}\alpha := \alpha \circ J$, for any $1$-form $\alpha$ and almost complex structure $J$.\\

\subsection{Twist construction}\label{twist-section}

Swann's twist construction associates to a manifold $M$ with a circle action a new manifold $W$. 
Following \cite{swann}, we now recall this construction. The starting point is a twist data: a tuple   $(X_{0},F, a)$ with the following properties:\\

$\bullet$ 
$X_{0}\in {\mathfrak X}(M)$ is a vector field which  generates a circle action on $M$.
Along the paper by an invariant tensor field, we mean a tensor field invariant under this action. 
We denote by $\Gamma (E)^{\mathrm{inv}}$ the space of  invariant sections of a tensor bundle $E\rightarrow M$;\\

$\bullet$   $F\in \Omega^{2} (M)$  is an invariant closed $2$-form. It is the curvature  of a connection $\mathcal H$
on a principal $S^{1}$-bundle $\pi: P \rightarrow  M$. We denote by $X^{P}\in {\mathfrak X}(P)$ 
the principal vector field of $\pi$ 
and by $\theta \in \Omega^{1}(M)$ the connection form of $\mathcal H$. 
We denote by $\tilde{X}\in {\mathfrak X} (P)$ the $\mathcal H$-horizontal lift
of any $X\in {\mathfrak X}(M)$;\\

$\bullet$ $a\in C^{\infty}(M)$ is non-vanishing and satisfies
\begin{equation}
 da = - i_{X_{0}} F. 
\end{equation}

In this setting, we assume that the vector field  $X^{\prime}_{0}:= \tilde{X}_{0} +\pi^{*}(a) X^{P}$
generates a free and proper group action 
of a $1$-dimensional Lie group. Let  
$W:= P /\langle X^{\prime}_{0}\rangle$ be the quotient manifold and 
$\pi_{W}: P \rightarrow W$ the natural projection.
Since $a$ is non-vanishing, $X^{\prime}_{0}$ is transversal to $\mathcal H$. We denote by $\widehat{X}\in \mathcal H$ the  horizontal lift of 
any $X\in {\mathfrak X}(W).$ 
Identifying (by means of the projections 
$\pi_{*}: \mathcal H_{p} \rightarrow  T_{\pi (p)}M$ and 
$(\pi_{W})_{*}: \mathcal H_{p} \rightarrow  T_{\pi_{W}(p)}W$),
$T_{\pi (p)}M$ and 
$T_{\pi_{W}(p)}W$ with $\mathcal H_{p}$, we can transfer any invariant tensor field $A$ on $M$ to a tensor field (of the same type) $A_{W}$ on 
$W$ (the invariance of $A$ 
ensures 
that  $A_{W}$ is well-defined). We say that $A$ and $A_{W}$ are $\mathcal H$-related
and we write $A_{W}\sim_{\mathcal H}A.$ The tensor field 
$A_{W}$ is called the twist of $A$.
In particular, $\tilde{X} = \widehat{X}_{W}$ and, for any invariant form $\alpha \in \Omega^{k}(M)$, $\pi^{*}(\alpha ) \vert_{\mathcal H}  
= \pi_{W}^{*} (\alpha_{W})\vert_{\mathcal H}.$ In fact, as proved in \cite{swann},
\begin{equation}\label{swann-pull}
\pi_{W}^{*} (\alpha_{W}) = \pi^{*} (\alpha ) -\frac{1}{a} \theta \wedge \pi^{*} ( i_{X_{0}} \alpha ).
\end{equation}

The exterior derivative behaves under twist as :
\begin{equation}\label{s1}
d \alpha_{W}\sim_{\mathcal H} d_{W}\alpha :=  d\alpha -\frac{1}{a} F\wedge i_{X_{0}} \alpha ,
\end{equation}
for any invariant form $\alpha \in \Omega^{k}(M)$. 
In particular, 
if $\alpha$ is closed, then its twist  $\alpha_{W}$ is also closed if and only if $F\wedge i_{X_{0}} \alpha =0.$ 

Similarly, the Lie bracket behaves under twist as:
\begin{equation}\label{s2}
[X_{W}, Y_{W}]\sim_{\mathcal H}  [X, Y] +\frac{1}{a} F(X, Y) X_{0},
\end{equation}
for any  invariant vector fields $X, Y\in {\mathfrak X}(M)$.
Using (\ref{s2}), one can show
that the twist  $J_{W}$ of an invariant 
complex structure $J$  on $M$ is integrable if and only if $F$ is of type $(1,1)$ with respect to $J$ (see \cite{swann}).

\subsection{Generalized complex geometry}\label{complex-geometry}

{\bf Generalized almost complex structures.}
A generalized almost complex structure on a manifold $M$ is a field of endomorphisms
$\mathcal J\in\Gamma   \mathrm{End}(\mathbb{T}M)$ of the generalized tangent bundle $\mathbb{T}M = TM\oplus T^{*}M$
which satisfies $\mathcal J^{2} = -\mathrm{Id}$ and is skew-symmetric with respect to the 
canonical metric $\langle \cdot , \cdot \rangle$ of $\mathbb{T}M$
of neutral signature, defined by
$$
\langle X+\xi , Y+\eta\rangle = \frac{1}{2} \left( \eta (X) + \xi (Y) \right) ,\ X+\xi , Y+ \eta\in \mathbb{T}M. 
$$
The $(1,0)$-bundle $L \subset (\mathbb{T}M)^{\mathbb{C}}$ (the $i$-eigenbundle) of a generalized almost 
complex structure $\mathcal J$  is maximal isotropic with respect to $\langle\cdot , \cdot \rangle$,  satisfies $L \oplus \bar{L} = (\mathbb{T}M)^{\mathbb{C}}$, and
(when $\mathcal J$ is of constant type),  has the form
$$
L = L(E, \epsilon )= \{ X+\xi \in E \oplus (T^{*}M)^{\mathbb{C}},\ \xi\vert_{E}= i_{X} \epsilon \} 
$$
where $E\subset (TM)^{\mathbb{C}}$ is a complex subbundle and $\epsilon\in\Gamma ( \Lambda^{2}E^{*})$ is a (complex) 2-form, 
such that 
$\mathrm{Im} (\epsilon\vert_{\Delta} )$ is non-degenerate, where
$\Delta \subset TM$ is defined by $\Delta^{\mathbb{C}} := E \cap \bar{E}$. 
The corank of $E\subset (TM)^{\mathbb{C}}$ is called the type of $\mathcal J $. 
For later use, we note that
\begin{equation}\label{epsilon}
\epsilon (X, Y) = 2 \langle X+\xi , Y\rangle = \xi (Y), 
\end{equation}
for $X+\xi \in L$ and $Y\in E.$

Usual almost complex and almost symplectic structures define generalized almost complex structures by
$$
\mathcal J_{J}:= \left( \begin{tabular}{cc}
                         $J$ & $0$\\
                         $0$ & $-J^{*}$
                        \end{tabular}\right) ,\quad \mathcal J _{\omega } := \left(\begin{tabular}{cc}
                        $0$ & $-\omega^{-1}$ \\
                        $ \omega$ & $0$
                        \end{tabular}\right) 
                        $$
with $(1,0)$-bundles $L( T^{1,0}M, 0)$ and $L((TM)^{\mathbb{C}}, - i\omega  )$ respectively.
For any generalized almost complex structure $\mathcal J$ and 
$2$-form $B\in \Omega^{2}(M)$, $\mathrm{exp}(B) \cdot \mathcal J := \mathrm{exp}(B)\circ \mathcal J \circ 
\mathrm{exp}(-B)$ is also a generalized almost complex structure, where $\mathrm{exp} (B) \in \mathrm{Aut} (\mathbb{T}M)$ is given by $\mathrm{exp} (B) (X+\xi )
:= X +\xi + i_{X}B$, for any $X+\xi \in \mathbb{T}M.$  
The generalized almost complex structure $\mathrm{exp}(B)\cdot \mathcal J$ is called the $B$-field transformation of
$\mathcal J .$ If $L= L(E, \epsilon )$ is the $(1,0)$-bundle of $\mathcal J$, then 
$L(E, B  + \epsilon  )$ is the $(1,0)$-bundle of $\mathrm{exp}(B)\cdot \mathcal J .$\\  

{\bf Generalized complex structures.}
A generalized almost complex structure $\mathcal J$ is integrable (i.e. is a generalized complex structure) if its 
$(1,0)$-bundle $L \subset (\mathbb{T}M)^{\mathbb{C}}$  is closed under the Courant bracket, given by
$$
[X+\xi , Y+\eta ] = [X, Y] + {\mathcal L}_{X}\eta  - {\mathcal L}_{Y} \xi  -\frac{1}{2}d \left (\eta (X) - \xi (Y)\right) .
$$
This reduces to the usual notions of integrability of almost complex and almost symplectic structures.
If $\mathcal J$ is integrable and $dB =0$, then $\mathrm{exp}(B)\cdot \mathcal J$ is also integrable. 
In terms of the $(1,0)$-bundle $L$, a generalized almost complex structure is integrable if and only if $E\subset (\mathbb{T}M)^{\mathbb{C}}$ is involutive and
$d\epsilon =0.$ As in the case of ordinary complex structures, a generalized almost complex structure $\mathcal J$
gives rise  to a tensor field $N_{\mathcal J} \in \Gamma ( \Lambda^{2} \mathbb{T}^{*}M\otimes \mathbb{T}M)$
(called the Nijenhuis tensor) defined by
\begin{equation}\label{def-nij}
N_{\mathcal J}(u,v):=  [ \mathcal J u, \mathcal J v] - [u,v] - \mathcal  J ( [\mathcal J u, v] + [u, \mathcal J v])
\end{equation}
and $\mathcal J$ is integrable if and only if $N_{\mathcal J }=0.$

The Courant bracket does not satisfy the Jacobi identity but has the following properties which we shall use in our computations (see e.g. \cite{hitchin}):
\begin{equation}\label{courant1}
[u, fv] = f [u,v] +\mathrm{pr}_{T} (u) (f) v - \langle u,v\rangle df
\end{equation}
and
\begin{equation}\label{uvw}
\mathrm{pr}_{T} (u) \langle v,w\rangle = \langle [u,v],w\rangle +  \langle v, [u,w]\rangle 
+ \langle d \langle u,v\rangle ,w\rangle + \langle v , d\langle u,w\rangle \rangle  ,
\end{equation}
for all $u,v,w\in \Gamma (\mathbb{T}M)$. Also,
the  Courant bracket  $L_{X}:= [ X,\cdot ]$  
is related to the Lie bracket ${\mathcal L}_{X}$ 
by
\begin{equation}\label{L-C}
L_{X} (Y+\eta ) = {\mathcal L}_{X} (Y+\eta )  - \frac{1}{2} d( \eta (X) ) =  {\mathcal L}_{X} (Y+\eta )  
- d \langle X , Y+\eta \rangle
\end{equation}
for any  $X\in {\mathfrak X}(M)$ and  $Y+\beta \in \Gamma (\mathbb{T}M)$.\\

{\bf Generalized almost Hermitian structures.}
A generalized almost Hermitian structure \cite{gualtieri} on a manifold $M$ is a Hermitian structure $(G, \mathcal J)$  on the bundle
$\mathbb{T}M$, such that
$\mathcal J$ is skew symmetric with respect to $\langle \cdot , \cdot \rangle$ (i.e. is a generalized almost complex structure on $M$)
and  $G^{\mathrm{end}}\in\Gamma  \mathrm{End} ( \mathbb{T}M)$, defined by
\begin{equation}\label{G}
G(u, v ) = \langle G^{\mathrm{end}} u , v \rangle ,
\end{equation}
satisfies $(G^{\mathrm{end}})^{2} =\mathrm{Id}.$  
The endomorphism $G^{\mathrm{end}}$ commutes with $\mathcal J$ and  
is of the form \cite{gualtieri}
\begin{equation}\label{gend}
G^{\mathrm{end}} = \left( \begin{tabular}{cc}
$A$ & $g^{-1}$\\
$\sigma$ & $ A^{*}$
\end{tabular}\right) ,
\end{equation}
where $g$ and $\sigma$ are Riemannian  metrics, 
$A\in \Gamma \mathrm{End} (TM)$ is skew-symmetric with respect to both $g$  and
$\sigma = g - bg^{-1} b$, where  $b:= - g  A\in \Omega^{2} (M)$. 
It has eigenvalues $\pm 1$ and its  associated eigenbundles $C_{\pm}$  are the graphs of 
$b\pm g : TM \rightarrow T^{*}M$.
By means of the isomorphism $\mathrm{pr}_{T}\vert_{C_{+}} : C_{+}\rightarrow TM$, 
the metric $g$ and the $2$-form $b$ correspond to 
$\langle \cdot , \cdot \rangle|_{C_+\times C_+}$ and, respectively, to  
$( \cdot , \cdot )|_{C_+\times C_+}$, where
$$
(X+\xi , Y+\eta ) :=\frac{1}{2} ( \xi (Y) - \eta (X)).
$$
Let $J_{\pm}\in\Gamma  \mathrm{End}(TM)$ be the $g$-orthogonal almost complex structures, which correspond to $\mathcal J\vert_{C_{\pm}}$ via the isomorphisms
$\mathrm{pr}_{T}\vert_{C_{\pm}} :C_{\pm}\rightarrow TM$. The generalized almost Hermitian structure
$(M, G, \mathcal J )$ is uniquely determined by  the data $(J_{+}, J_{-}, g, b)$.

On a generalized almost Hermitian manifold $(M, G, \mathcal J)$ there is a second generalized almost complex structure 
$\mathcal J_{2} := G^{\mathrm{end}} \mathcal J $, which commutes with $\mathcal J_{1}:= \mathcal J .$ 
The generalized almost complex structures $\mathcal J_{1}$ and $\mathcal J_{2}$  are on  equal footing: 
one can alternatively define a generalized almost Hermitian structure as a
pair of commuting generalized almost complex structures $(\mathcal J_{1}, \mathcal J_{2})$, such that
the metric $G$ defined by  (\ref{G}) with 
$G^{\mathrm{end}} := - \mathcal J_{1} \mathcal J_{2}$ is positive definite.
In analogy with the quaternionic case, we shall use the notation 
$\mathcal J_{3}:= \mathcal J_{1} \mathcal J_{2} = - G^{\mathrm{end}}.$
Note that the structures $\mathcal J_{1}, \mathcal J_{2}, \mathcal J_{3}$ commute and that 
\begin{equation} \mathcal{J}_{1}^2=\mathcal{J}_{2}^2=-\mathcal{J}_{3}^2=-\mathrm{Id}.
\end{equation}

Any almost Hermitian structure $(g, J)$ determines a generalized almost Hermitian structure with generalized 
almost complex structures $\mathcal J_{1} := \mathcal J_{J}$, $\mathcal J_{2} := \mathcal J_{\omega}$
(where $\omega := g\circ J$ is the K\"{a}hler form), metric $G$ and endomorphism $G^{\mathrm{end}}$ given by
\begin{equation}
G(X+\xi , Y+\eta ) = \frac{1}{2} \left( g(X,Y) + g^{-1}(\xi , \eta ) \right) ,\quad 
G^{\mathrm{end}} = \left( \begin{tabular}{cc}
$0$ & $g^{-1}$\\
$g$ & $ 0$
\end{tabular}\right) .
\end{equation}

{\bf Generalized K\"{a}hler structures.}
A generalized K\"{a}hler structure  \cite{gualtieri} is a generalized almost Hermitian structure $(G, \mathcal J)$ 
on a manifold $M$
for which $\mathcal J_{1}$ and 
$\mathcal J_{2}$ (defined as above) 
are generalized complex structures. 
As proved by Gualtieri \cite{gualtieri}, a generalized almost Hermitian structure 
$(G, \mathcal J)$ is generalized K\"{a}hler if and only if  
$\mathcal J_{1}$  (or $L_{1}$) is Courant integrable and the bundles 
$L_{1} \cap L_{2}$ and $L_{1} \cap \bar{L}_{2}$ are  also Courant integrable, 
where $L_{i}$ are the $(1,0)$-bundles  of $\mathcal J_{i}$ ($i=1,2$).   
In terms of the data $(J_{+}, J_{-}, g, b)$ associated to $(G, \mathcal J )$, the generalized K\"{a}hler condition is equivalent to the integrability of the almost complex structures $J_{\pm}$,
together with the relation
$$
db ( X, Y, Z) = d\omega_{+} (J_{+}X, J_{+}Y, J_{+}Z) = - d\omega_{-}(J_{-} X, J_{-}Y, J_{-} Z),
$$
for all $X, Y, Z\in \mathfrak{X}(M)$, where $\omega_{\pm}:= g (J_{\pm}\cdot , \cdot )$ are the K\"{a}hler forms.

 The  generalized almost Hermitian structure 
$(\mathcal J_{J}, \mathcal J_{\omega })$  
defined by an
almost Hermitian structure $(g, J)$, 
with K\"{a}hler form $\omega$,   is generalized K\"{a}hler, if and only if $(g, J)$ is K\"{a}hler.
The $B$-field transformation $(\mathrm{exp}(B)\cdot\mathcal J_{1}, \mathrm{exp}(B)\cdot \mathcal J_{2})$ 
(with $dB=0$) of a generalized K\"{a}hler structure $(\mathcal J_{1}, \mathcal J_{2})$ is  generalized 
K\"{a}hler.\\

{\bf Generalized K\"{a}hler structures of symplectic type.} 
Let  $(M, \omega )$ be a 
symplectic manifold. 
The map $\mathcal J \rightarrow J_{+}$ is  a one to one correspondence  from the space of generalized K\"{a}hler structures 
$(G, \mathcal J )$ with second generalized complex structure $\mathcal J_{2}= \mathcal J_{\omega }$, and the space of 
(integrable) 
complex structures $J_{+}$ which 
tame $\omega$,   which means that $\omega (X, J_{+}X) >0$, for all 
$X\neq 0$, and 
the $\omega$-adjoint of $J_{+}$, given by $J_{+}^{*\omega} = \omega^{-1}\circ J_{+}^{*}\circ \omega$, 
is  integrable \cite{fino}. The complex structure $J_{-}$, Hermitian metric $g$ and $2$-form $b$  associated to $(G,\mathcal J)$ 
are 
given by:
\begin{equation}\label{hermitian-p}
J_{-} = - J_{+}^{*\omega},\ g = -\frac{1}{2} \omega \circ ( J_{+} + J_{-}),\  b= -\frac{1}{2} \omega \circ (J_{+} - J_{-}).
\end{equation}
A generalized K\"{a}hler structure with $\mathcal J_{2} = \mathcal J_{\omega}$ is called of  symplectic type.\\

{\bf Hamiltonian vector fields on generalized K\"{a}hler manifolds.}
This is a particular case of the more general notion of  Hamiltonian (real) actions of groups on generalized K\"{a}hler manifolds \cite{tolman}.

\begin{defn}\label{def-HK} Let $(M, G, \mathcal J )$ be a generalized K\"{a}hler manifold, with generalized complex structures $\mathcal J_{1}:= \mathcal J$
and $\mathcal J_{2}:= G^{\mathrm{end}} \mathcal J .$
 A vector field $X_{0}$ is called {\cmssl Hamiltonian Killing} if
${\mathcal L}_{X_{0}} ({\mathcal J })=0$, ${\mathcal L}_{X_{0}} (G) =0$, and there is a function $f^{H} : M \rightarrow \mathbb{R}$ (called the {\cmssl Hamiltonian function}) 
such that
${\mathcal J}_{2} X_{0} = d f^{H}$.
\end{defn}

Let $(M, \mathcal J , G)$ be a generalized K\"{a}hler manifold with $\mathcal J_{2} = \mathcal J_{\omega}$
determined by a symplectic form $\omega$. Any Hamiltonian Killing vector field $X_{0}$ on $(M, \omega )$, 
with ${\mathcal L}_{X_{0}} ( \mathcal J ) =0$,  is Hamiltonian Killing on the generalized K\"{a}hler manifold
$(M, \mathcal J , G ).$

\section{Algebraic operations in generalized complex geometry}\label{algebraic}

Let $M$ be a manifold and   $\tau \in \mathrm{Isom}( \mathbb{T}M, \langle \cdot , \cdot \rangle )$. If  $\mathcal J$ is a generalized almost complex structure on $M$, then so is $\tau (\mathcal J) := \tau\circ \mathcal J \circ \tau^{-1}$. Similarly,  if  $(G, {\mathcal J })$ is a generalized almost Hermitian structure on $M$, then  
$\tau (G, \mathcal J ):= (G^{\prime} := (\tau^{-1})^{*} (G), \mathcal J^{\prime} = \tau\circ \mathcal J \circ \tau^{-1})$ is also  a  generalized almost Hermitian structure, with endomorphism 
$(G^{\prime})^{\mathrm{end}}
=\tau \circ G^{\mathrm{end}} \circ \tau^{-1}$ and  second generalized almost complex structure $\mathcal J_{2}^{\prime} =  \tau\circ \mathcal J_{2} \circ \tau^{-1}$.
We apply these remarks to define the conformal change and elementary deformation in generalized complex geometry.

\subsection{Conformal change}

Any non-vanishing  function 
$h\in C^{\infty} (M)$ defines 
$\tau_{h}\in \mathrm{Isom} (\mathbb{T}M,\langle \cdot , \cdot \rangle )$, by
$\tau_{h}(X)  =hX$, $\tau_{h} (\xi ) =\frac{1}{h} \xi$, for any  $X+\xi \in \mathbb{T}M.$

\begin{defn}\label{def-conf} The generalized almost complex structure $\tau_{h}(\mathcal J )$ 
is called the {\cmssl conformal change} of  $\mathcal J$ by  $h\in C^{\infty}(M)$.
The generalized almost Hermitian structure $\tau_{h} (G, \mathcal J )$ is called the 
{\cmssl conformal change} of $(G, \mathcal J )$ by $h$. 
\end{defn}

\begin{rem}{\rm i) If $\mathcal J$ is a generalized almost complex structure with $(1,0)$-bundle $L = L(E, \epsilon )$, then the $(1,0)$-bundle
of $\tau_{h} (\mathcal J)$ is $L^{h} = L (E,\frac{1}{h^{2}} \epsilon ).$ 
In particular, the conformal change preserves  the type of a generalized almost complex structure.
Notice that  $\tau_{h} (\mathcal J_{J} ) =\mathcal J_{J}$ for every almost complex structure $J$ and 
$\tau_{h} (\mathcal J_{\omega}) = \mathcal J_{\frac{1}{h^{2}} \omega }$ for every almost symplectic form $\omega .$\

ii) If $(G, \mathcal J )$ is the generalized almost Hermitian structure defined by a usual 
almost Hermitian structure $(g, J)$,  then $\tau_{h}( G, \mathcal J )$ is defined by $ (\frac{1}{h^{2}} g, J).$}
\end{rem}

\subsection{Elementary deformation}\label{elementary}

Let $( G, \mathcal J)$ be a generalized almost Hermitian structure on a manifold $M$,  $X_{0}\in {\mathfrak X }(M)$ a non-vanishing vector field and
$f\in C^{\infty}(M)$ a non-vanishing function.  In this section we associate to this data a new generalized almost Hermitian structure on $M$ and we study some of 
its properties. Define
$$
\mathcal S := \mathrm{span} \{ X_{0} , {\mathcal J}_{1}X_{0}, {\mathcal J}_{2}X_{0}, 
\mathcal J_{3} X_{0}\} ,
$$
where, as usual, $\mathcal J_{1}= \mathcal J$, $\mathcal J_{2}= G^{\mathrm{end}}  \mathcal J$ and $\mathcal J_{3} =- G^{\mathrm{end}} = \mathcal J_{1} \mathcal J_{2}.$
Since $G$ is positive definite,  $X_{0}$, ${\mathcal J}_{1}X_{0}$, ${\mathcal J}_{2}X_{0}$, $\mathcal J_{3} X_{0}$
are linearly independent.  Moreover, 
$$
\mathcal S  = \mathrm{span}_{\mathbb{R}} \{ X_{0}, {\mathcal J}X_{0}\} +\mathrm{span}_{\mathbb{R}} \{ \mathcal J_{2} X_{0},
\mathcal J_{3} X_{0} \}
$$
is a direct sum decomposition into two isotropic planes, which are interchanged by $\mathcal J_{2}$ and 
$\mathcal J_{3} .$
The restriction of $\langle \cdot , \cdot \rangle$ to $\mathcal S$ is non-degenerate and 
the orthogonal complements 
of $\mathcal S$ with respect to $G$ and $\langle\cdot , \cdot \rangle$ coincide and   will be denoted by
$\mathcal S^{\perp}.$ 
For $u\in \mathbb{T}M$ we shall  denote by $u^{\mathcal S}$ and $u^{\perp}$ its components on 
$\mathcal S$ and $\mathcal S^{\perp}$, with respect to the decomposition 
$\mathbb{T}M = \mathcal S \oplus \mathcal S^{\perp}.$

Let
$\tau_{f}^{\mathcal S}\in \mathrm{Aut}(\mathcal S )$ be given by 
$$
X_{0}\mapsto  f X_{0},\ {\mathcal J} X_{0} \mapsto  f \mathcal J X_{0},\ 
{\mathcal J}_{2} X_{0} \mapsto \frac{1}{ f} \mathcal 
J_{2} X_{0},\ {\mathcal J}_{3} X_{0} \mapsto \frac{1}{ f} \mathcal J_{3} X_{0}.
$$

\begin{lem}\label{clear} The automorphism $\tau^{\mathcal S}_{f}\in \mathrm{Aut}(\mathcal S )$ is an isometry 
with respect to $\langle \cdot , \cdot \rangle$.
 \end{lem}

Decompose $\mathbb{T}M:=  \mathcal S^{\perp}\oplus \mathcal S$ and let
$\tau := \mathrm{Id}_{\mathcal S ^{\perp}} + \tau_{f}^{\mathcal S }\in \mathrm{Aut} (\mathbb{T}M)$.

\begin{prop} Define a metric $G^{\prime}:=(\tau^{-1})^{*}(G)$ on $\mathbb{T}M$. 
Then $(G^{\prime}, \mathcal J )$ is a generalized almost Hermitian structure. \end{prop}

\begin{proof}
The claim  follows from the fact that  $\tau $ is an isometry for 
$\langle \cdot , \cdot \rangle $ 
(using Lemma \ref{clear}) and commutes with $\mathcal J$. 
\end{proof}

\begin{defn}\label{def-elem} The generalized almost Hermitian structure $(G^{\prime}, {\mathcal J} )$ is called the {\cmssl elementary deformation} of $(G, {\mathcal J })$ by
$X_{0}$ and $f$.
\end{defn}

\begin{rem}\label{rem-elem}{\rm  
i)   If a generalized almost Hermitian structure is defined by  a usual almost Hermitian structure with metric $g$ and almost complex structure $J$,
then its elementary deformation by $X_{0}$ and $f$  is also defined by a usual almost Hermitian structure, with 
the same almost complex structure  $J$ and 
metric $g^{\prime}$,
such that  $\mathrm{span}_{\mathbb{R}} \{ X_{0},  J X_{0} \}$ and  
$\mathrm{span}_{\mathbb{R}}\{ X_{0}, J X_{0} \}^{\mathrm{perp}}$ are $g^{\prime}$-orthogonal and
$$
g^{\prime} = \frac{1}{f^{2}} g\vert_{\mathrm{span}_{\mathbb{R}} \{ X_{0}, JX_{0}\} } + g\vert_{ \mathrm{span} _{\mathbb{R}}\{ X_{0}, JX_{0}\}^{\mathrm{perp}} }. 
 $$
Here ``$\mathrm{perp}$'' refers to the $g$-orthogonal complement.\

ii)  More generally, elementary deformations preserve the class of generalized almost Hermitian structures
which are $B$-field transformations  of usual almost Hermitian structures.
This follows from the fact that elementary deformations leave the first generalized almost complex structure
unchanged, together with
the fact  that  any generalized almost 
Hermitian structure $(G, \mathcal J )$ for which $\mathcal J$ is the $B$-field transformation of a usual almost
complex structure is the $B$-field transformation of a usual almost Hermitian structure (this  is a
consequence of relation (6.3), page 76, of \cite{gualtieri}).}
\end{rem}

The next lemma can be checked directly, from the definition of elementary deformation and Lemma \ref{clear}.
 
\begin{lem}
The second generalized
almost complex structure $\mathcal J_{2}^{\prime}$ 
and the endomorphism $(G^{\mathrm{end}})^{\prime}$ of $(G^{\prime}, \mathcal J )$ coincide, respectively,  with
 $\mathcal J_{2}$ and $G^{\mathrm{end}}$  on $\mathcal S^{\perp}$.
 On $\mathcal S$, $\mathcal J_{2}^{\prime}$ is given by 
\begin{align}
\nonumber&X_{0}\mapsto \frac{1}{f^{2}} \mathcal J_{2} X_{0},\  \mathcal J X_{0}\mapsto 
\frac{1}{f^{2}} \mathcal J_{3} X_{0},\\
\label{j2-prime}&\mathcal J_{2} X_{0} \mapsto -  f^{2} X_{0},\  \mathcal J_{3} X_{0}\mapsto - 
f^{2} \mathcal J X_{0}
\end{align}
and $(G^{\mathrm{end}})^{\prime}$ by
\begin{align*}
\nonumber&X_{0}\mapsto - \frac{1}{f^{2}} \mathcal J_{3} X_{0},\  \mathcal J X_{0}\mapsto 
\frac{1}{f^{2}} \mathcal J_{2} X_{0},\\
&\mathcal J_{2} X_{0} \mapsto   f^{2}\mathcal J  X_{0},\  \mathcal J_{3} X_{0}\mapsto - 
f^{2}  X_{0}.
\end{align*}
The vector fields  $\{ X_{0} , {\mathcal J } X_{0},  {\mathcal J }_{2} X_{0}, {\mathcal J }_{3} X_{0}\}$ are $G^{\prime}$-orthogonal and
\begin{align}
\nonumber & G^{\prime}(X_{0}, X_{0}) = G^{\prime} (\mathcal J X_{0}, \mathcal J X_{0}  )= \frac{1}{f^{2}} G(X_{0}, X_{0})\\
\nonumber & G^{\prime}(\mathcal J_{2} X_{0},  \mathcal J_{2}X_{0}) = G^{\prime} (\mathcal J_{3} X_{0}, \mathcal J_{3} X_{0} ) = f^{2} G(X_{0}, X_{0}).
\end{align}
\end{lem}

An important role for the integrability of a generalized almost Hermitian structure  $(G , \mathcal J)$
 is played by the intersections $L_{1} \cap L_{2}$ and $L_{1} \cap \bar{L}_{2}$,
where $L_{i}$ are the $(1,0)$-bundles of $\mathcal J_{i}$ (recall Section \ref{complex-geometry}).
The next lemma describes these intersections for elementary deformations. 
Let
\begin{align}
\nonumber {v}_{f}&:= X_{0} - i {\mathcal J}X_{0} -\frac{1}{f^{2}}  
({\mathcal J}_{3} X_{0} + i{\mathcal J}_{2} X_{0}) \\
\label{v-f} {v}_{if}&:= X_{0} - i {\mathcal J}X_{0} + \frac{1}{f^{2}}  
({\mathcal J}_{3} X_{0} + i{\mathcal J}_{2} X_{0}).
\end{align}

\begin{lem}\label{elem-bundles} Let $(M,G, \mathcal J)$ be a generalized almost Hermitian manifold, $X_{0}\in {\mathfrak X}(M)$ a non-vanishing vector field and 
$f\in C^{\infty}(M)$ a non-vanishing function. Denote by $L_{i}$ the $(1,0)$-bundles of $\mathcal J_{i}$ $(i=1,2$).
Let $(G^{\prime}, \mathcal J )$ be the elementary deformation of $(G, \mathcal J)$ by
$X_{0}$ and $f$. Let
$L_{2}^{\prime}= L (E_{2}^{\prime}, \epsilon_{2}^{\prime} )$ be the $(1,0)$-bundle of 
the second generalized almost complex structure $\mathcal J^{\prime}_{2} = (G^{\mathrm{end}})^{\prime}  \mathcal J$
of $(G^{\prime}, \mathcal J ).$
 Then
\begin{align}
\nonumber &  L_{1} \cap L_{2}^{\prime} = \mathrm{span}_{\mathbb{C}} \{ v_{f}\} +
L_{1}\cap L_{2}\cap \mathcal S_{\mathbb{C}}^{\perp},\\
\label{hol-sp}& L_{1} \cap\bar{L}_{2}^{\prime} =
 \mathrm{span}_{\mathbb{C}} \{ v_{if}\} + L_{1}\cap \bar{L}_{2}\cap \mathcal S_{\mathbb{C}}^{\perp}
\end{align}
and
\begin{align}
\nonumber &\mathrm{pr}_{T} (L_{1} \cap L_{2}^{\prime}) 
=\mathrm{span}_{\mathbb{C}} \{ \mathrm{pr}_{T}( {v}_{f} ) \} +  \mathrm{pr}_{T} (L_{1} \cap L_{2} \cap 
\mathcal S_{\mathbb{C}}^{\perp})\\
\label{dec-direct} &\mathrm{pr}_{T} (L_{1} \cap \bar{L}_{2}^{\prime}) 
=\mathrm{span}_{\mathbb{C}} \{ \mathrm{pr}_{T} ( v_{if})  \} 
+ \mathrm{pr}_{T} (L_{1} \cap\bar{L}_{2} \cap \mathcal S_{\mathbb{C}}^{\perp})
\end{align}
are all direct sum decompositions, where $\mathcal S_{\mathbb{C}}$ is the complexification of $\mathcal S$.
\end{lem}

\begin{proof}
We only prove the statements about $L_{1} \cap L_{2}^{\prime}$, i.e. the first
relation (\ref{hol-sp}) and the first relation (\ref{dec-direct})  (the statements about  
$L_{1} \cap \bar{L}_{2}^{\prime}$ can be proved in a similar way).
Since 
$\mathcal J^{\prime}_{2}$ preserves $\mathcal S$ and $\mathcal S^{\perp}$,   
$L_{2}^{\prime} = L_{2}^{\prime} \cap \mathcal S_{\mathbb{C}} + 
L_{2}^{\prime} \cap \mathcal S_{\mathbb{C}}^{\perp}.$ 
Since $\mathcal J$ commutes with $\mathcal J_{2}^{\prime}$, it preserves $L_{2}^{\prime}.$ 
It also preserves $\mathcal S$ and its orthogonal complement
$\mathcal S^{\perp}$. Therefore, 
\begin{equation}\label{mare}
L_{1}\cap L_{2}^{\prime}  = L_{1}\cap L_{2}^{\prime} \cap \mathcal S_{\mathbb{C}} +
L_{1}\cap L^{\prime}_{2} \cap \mathcal S_{\mathbb{C}}^{\perp} =  L_{1}\cap L_{2}^{\prime} \cap \mathcal S_{\mathbb{C}}  
+ L_{1}\cap L_{2} \cap 
\mathcal S_{\mathbb{C}}^{\perp} 
\end{equation}
(direct sums)   since  $ L_{1}\cap L_{2}^{\prime} \cap \mathcal S_{\mathbb{C}}^{\perp} = 
L_{1}\cap L_{2} \cap \mathcal S_{\mathbb{C}}^{\perp}  $ (because $\mathcal J_{2}^{\prime} = \mathcal J_{2}$
on $\mathcal S^{\perp}$).
On $\mathcal S$, $\mathcal J_{2}^{\prime}$ is given by  (\ref{j2-prime}), from where we deduce that
\begin{align}
 \nonumber& L_{2}^{\prime} \cap \mathcal S_{\mathbb{C}} = \mathrm{span}_{\mathbb{C}} 
\{ X_{0} - \frac{i}{f^{2}} \mathcal J_{2} X_{0}, \mathcal J X_{0} - \frac{i}{f^{2}} \mathcal J_{3} X_{0} \},\\
\label{l2-w}&  L_{1} \cap L_{2}^{\prime} \cap \mathcal S_{\mathbb{C}} = \mathrm{span}_{\mathbb{C}} \{ v_{f} \}.
\end{align}
Combining 
the second relation (\ref{l2-w})  with (\ref{mare}) we obtain
the first relation (\ref{hol-sp}).

We now check  the first relation  (\ref{dec-direct}). 
By projecting the first relation
(\ref{hol-sp}) onto the tangent bundle, we obtain that $\mathrm{pr}_{T} ({v}_{f} )$ and 
$\mathrm{pr}_{T} (L_{1} \cap L_{2} \cap \mathcal S_{\mathbb{C}}^{\perp})$ generate
$\mathrm{pr}_{T} (L_{1} \cap L_{2})$. 
We will show that they are also transverse. 
Suppose, by contradiction, that this is not true. Then there is $\xi \in (T^{*}M)^{\mathbb{C}}$ such that 
\begin{equation}\label{cond-psi}
{v}_{f} +\xi \in L_{1} \cap L_{2} \cap \mathcal S_{\mathbb{C}}^{\perp}. 
\end{equation}
The condition $v_{f} +\xi \in L_{1}$ together with $v_{f}\in L_{1}$ implies that $\xi \in L_{1}.$
On the other hand, from the definition of $v_{f}$, 
\begin{equation}\label{this}
{\mathcal J}_{2} {v}_{f}  = {\mathcal J}_{2} X_{0} - i {\mathcal J}_{3} X_{0} +\frac{1}{f^{2}} ( i X_{0} + {\mathcal J}X_{0}).
\end{equation}
From (\ref{this})  and 
$\mathcal J_{2} ( v_{f} + \xi ) = i (v_{f}+\xi )$ 
(recall that ${v}_{f} +\xi \in L_{2}$), we obtain 
\begin{equation}\label{j2psi}
\mathcal J_{2}  \xi  = i \xi +( \frac{1}{f^{2}} -1) ({\mathcal J}_{2} X_{0} - i {\mathcal J}_{3} X_{0}) +( 1 -\frac{1}{f^{2}}) ( iX_{0} + \mathcal J X_{0}).
\end{equation}
But $\mathcal J \xi  = i \xi $ implies  $\mathcal J_{2} \xi  = i G^{\mathrm{end}} (\xi )$ and relation (\ref{j2psi})  becomes
$$
G^{\mathrm{end}}\left( \xi + ( 1-\frac{1}{f^{2}}) ( X_{0} - i {\mathcal J} X_{0})\right) =  \xi + ( 1-\frac{1}{f^{2}}) ( X_{0} - i {\mathcal J} X_{0}),
$$
i.e. $ \xi + ( 1-\frac{1}{f^{2}} )( X_{0} - i {\mathcal J}X_{0})\in C_{+}$ (the $+1$-eigenbundle $C_{+}$ of $G^{\mathrm{end}}$). But
$C_{+}$ is the graph of $ (b+g)$, where $b$ and $g$ are the $2$-form, respectively the metric of the bi-Hermitian structure
associated to $(G, \mathcal J )$  (see Section \ref{complex-geometry}).
We obtain that $\xi$ is given by
\begin{equation}\label{def-xi}
\xi = i (1-\frac{1}{f^{2}}) \mathrm{pr}_{T^{*}} ({\mathcal J} X_{0}) +
(1-\frac{1}{f^{2}}) (b+ g) (X_{0} - i\mathrm{pr}_{T}{\mathcal J} X_{0}).
\end{equation}
To summarize:  we  proved that ${v}_{f} +\xi \in L_{1} \cap L_{2} $ implies that $\xi$ is 
given by (\ref{def-xi}). We now show that the additional condition 
${v}_{f} +\xi \in \mathcal S^{\perp}_{\mathbb{C}}$ from (\ref{cond-psi}),  
with $\xi$ given by (\ref{def-xi}), leads to a contradiction. 
Indeed, suppose, by  contradiction, that  
${v}_{f} +\xi \in \mathcal S^{\perp}_{\mathbb{C}}$. In particular, $\langle {v}_{f} +\xi , X_{0}\rangle =0$.
Taking the real part of this equality  
and using (\ref{def-xi}) 
we obtain
\begin{equation}\label{special}
\frac{1}{f^{2}} G(X_{0}, X_{0}) + \frac{1}{2} ( 1-\frac{1}{f^{2}}) g(X_{0}, X_{0})=0.
\end{equation}
But, using (\ref{gend}),
$$
G(X_{0}, X_{0}) = \frac{1}{2} \left( g(X_{0}, X_{0}) + g^{-1}  ( b(X_{0}), b( X_{0}))\right) 
$$
and the left hand side of (\ref{special}) is equal to 
$$
\frac{1}{2f^{2}} g^{-1} (b(X_{0}), b (X_{0})) +\frac{1}{2} g(X_{0}, X_{0}),
$$ 
which is non-zero (because $g$ is positive definite and $X_{0}$ is non-vanishing). We obtain a contradiction.
\end{proof}

\begin{rem}\label{prel-deco}{\rm 
i) From  (\ref{hol-sp}), the bundle $E_{2}^{\prime}$ decomposes as
\begin{align}
\label{eprime-2} &  E_{2}^{\prime}  = \mathrm{pr}_{T} ( \bar{L}_{1}\cap {L}_{2}^{\prime}) + \mathrm{pr}_{T} ({L}_{1} \cap L_{2}^{\prime})\\ 
\nonumber&= \mathrm{span}_{\mathbb{C}} \{ \mathrm{pr}_{T}(v_{f})
, \mathrm{pr}_{T}(\bar{v}_{if})\} + \mathrm{pr}_{T} ( \bar{L}_{1} \cap L_{2} \cap 
\mathcal S_{\mathbb{C}}^{\perp}) + \mathrm{pr}_{T} ({L}_{1}\cap L_{2} \cap \mathcal S_{\mathbb{C}}^{\perp}).
\end{align}
The $2$-forms $\epsilon_{2}^{\prime}$ and $\epsilon_{2}$ coincide, when both arguments belong to
$\mathrm{pr}_{T}(L_{2} \cap \mathcal S_{\mathbb{C}}^{\perp}) $ (the sum of the last two components in the second
decomposition (\ref{eprime-2})). From (\ref{epsilon}), for any $X\in E_{2}^{\prime}$, 
$$
\epsilon_{2}^{\prime}( \mathrm{pr}_{T}(v_{f}), X) = 2\langle v_{f}, X\rangle ,\  \epsilon_{2}^{\prime}( \mathrm{pr}_{T} (\bar{v}_{if}), X) = 
2\langle \bar{v}_{if}, X\rangle 
$$
and
$$
\epsilon_{2}^{\prime} ( \mathrm{pr}_{T} (v_{f}), \mathrm{pr}_{T}(\bar{v}_{if})) = 
2 \langle v_{f}, \mathrm{pr}_{T} ( \bar{v}_{if})\rangle  = - 2  \langle\bar{v}_{if},\mathrm{pr}_{T} ({v}_{f})\rangle . 
$$
ii)  The intersections $L_{1} \cap L_{2}$ and $L_{1} \cap \bar{L}_{2}$ have similar decompositions
as in Lemma \ref{elem-bundles}, with ${v}_{f}$ replaced by
\begin{equation}\label{v-1}
{v}_{1}:=  X_{0} - i {\mathcal J}X_{0} -
({\mathcal J}_{3} X_{0} + i{\mathcal J}_{2} X_{0})
\end{equation}
and $v_{if}$ replaced by
\begin{equation}\label{v-i}
{v}_{i}:=  X_{0} - i {\mathcal J}X_{0} +
{\mathcal J}_{3} X_{0} + i{\mathcal J}_{2} X_{0}.
\end{equation}
That is, 
\begin{align}
\nonumber & L_{1} \cap L_{2} = \mathrm{span}_{\mathbb{C}} \{ v_{1} \} +
L_{1}\cap L_{2}\cap \mathcal S_{\mathbb{C}}^{\perp}\\
\label{hol-sp-1} &L_{1} \cap\bar{L}_{2} = \mathrm{span}_{\mathbb{C}} \{ v_{i} \} + 
L_{1}\cap \bar{L}_{2}\cap \mathcal S_{\mathbb{C}}^{\perp}.
\end{align}
}\end{rem}

From Lemma \ref{elem-bundles} and Remark \ref{prel-deco}, we obtain:

\begin{cor} The elementary deformations preserve the rank of the bundles
$L_{1} \cap L_{2}$, $L_{1} \cap \bar{L}_{2}$ and of their projections to $(TM)^{\mathbb{C}}.$ 
\end{cor}

\section{Twist of generalized almost complex structures}\label{twist}

From now on, we fix twist data $(X_{0}, F, a)$ on a manifold $M$.

\begin{lem}\label{lie-forms-lemma} For any invariant vector field $X\in {\mathfrak X}(M)$ and invariant form $\alpha \in \Omega^{k}(M)$,
\begin{equation}\label{lie-forms}
{\mathcal L}_{X_{W}} \alpha_{W}\sim_{\mathcal H} {\mathcal L}_{X}\alpha  -\frac{1}{a} (i_{X}F) \wedge i_{X_{0}}\alpha .
\end{equation}
\end{lem}

\begin{proof}
Using relation (\ref{swann-pull}), $\tilde{X} = \widehat{X}_{W}$ and $\theta (\tilde{X})=0$, we obtain  
\begin{align}
\nonumber& \pi_{W}^{*} ( i_{X_{W}} (\alpha_{W} ) )= \pi_{W}^{*} (\alpha_{W}) (\widehat{X}_{W},\cdot ) 
= \pi_{W}^{*} (\alpha_{W}) (\tilde{X},\cdot ) \\
\nonumber & = \left( \pi^{*}(\alpha ) -\frac{1}{a} \theta \wedge \pi^{*} (i_{X_{0}} \alpha )\right) (\tilde{X},\cdot )\\
\label{contraction} &= \pi^{*} ( i_{X}\alpha ) +\frac{1}{a} \theta \wedge \pi^{*} (i_{X} i_{X_{0}} \alpha ).
\end{align}
Thus, 
\begin{equation}\label{1}
\pi_{W}^{*} ( i_{X_{W}} \alpha_{W} ) = \pi^{*} ( i_{X}\alpha ) +\frac{1}{a} \theta \wedge \pi^{*} ( i_{X} i_{X_{0}} \alpha ).
\end{equation}
Pulling back by $\pi_{W}$ the Cartan formula 
$$
{\mathcal L}_{X_{W}} \alpha_{W}= i_{X_{W}} ( d\alpha_{W}) + d ( i_{X_{W}}\alpha_{W})
$$ 
and using (\ref{1}), 
we can write 
\begin{align}
\nonumber \pi_{W}^{*} ( {\mathcal L}_{X_{W}} (\alpha_{W}))  & =\pi_{W}^{*} (  i_{X_{W}} ( d\alpha_{W}) ) +
d \pi_{W}^{*}  ( i_{X_{W}} \alpha_{W})\\
\label{2}& = i_{\widehat{X}_{W} } \pi_{W}^{*} ( d \alpha_{W} ) + d \left( \pi^{*} (i_{X}\alpha ) +\frac{1}{a} \theta \wedge \pi^{*} ( i_{X} i_{X_{0}}\alpha ) \right) . 
\end{align}
Recall from relation (\ref{s1})  that  $ d\alpha_{W} \sim_{\mathcal H} d_{W}\alpha$. From this fact, together with (\ref{swann-pull}), we obtain
$$
\pi_{W}^{*} ( d\alpha_{W}) = \pi^{*} (d_{W}\alpha ) -\frac{1}{a} \theta \wedge \pi^{*} ( i_{X_{0}} d_{W}\alpha ).
$$
Replacing this relation into (\ref{2}) and using $da = - i_{X_{0}} F$ and $d\theta = \pi^{*}F$,  we obtain
\begin{align}
\nonumber & \pi_{W}^{*} ( {\mathcal L}_{X_{W}} \alpha_{W})  = \pi^{*} ( {\mathcal L}_{X}\alpha ) -\frac{1}{a} \pi^{*} (( i_{X}F) \wedge i_{X_{0}} \alpha )\\
\nonumber & + \frac{1}{a} \theta \wedge \pi^{*} ( i_{X}i_{X_{0}} d\alpha - d( i_{X}i_{X_{0}} \alpha ))-\frac{F(X_{0}, X)}{a^{2}} \theta\wedge \pi^{*} ( i_{X_{0}}\alpha ).
\end{align}
This relation, restricted to $\mathcal H$, gives  
\begin{equation}\label{c}
\pi_{W}^{*} ( {\mathcal L}_{X_{W}} \alpha_{W}) \vert_{\mathcal H} = \pi^{*} ( {\mathcal L}_{X}\alpha ) \vert_{\mathcal H}-\frac{1}{a} \pi^{*} (( i_{X}F) \wedge i_{X_{0}}\alpha )\vert_{\mathcal H}
\end{equation}
which is (\ref{lie-forms}).
\end{proof}

The next lemma describes the behaviour of the Courant bracket under twists. 

\begin{lem} \label{l-courant-twist}For any invariant sections $X+\xi , Y+\eta \in \Gamma (\mathbb{T}M)$, the Courant bracket 
$[ (X+\xi )_{W} , (Y+\eta )_{W}]$  is $\mathcal H$-related to
$$
 [X+\xi , Y+\eta ] + \frac{ F(X, Y) }{a}X_{0}  -   \frac{\eta (X_{0})}{a} i_{X}F + \frac{\xi (X_{0})}{a} i_{Y}F.
$$
\end{lem}

\begin{proof} Recall the expression of the Courant bracket: 
\begin{align}
\nonumber [(X+ \xi )_{W}, (Y+\eta )_{W}]  &= [X_{W}, Y_{W}] + {\mathcal L}_{X_{W}} \eta_{W} - {\mathcal L}_{Y_{W}} \xi_{W} \\
\label{def-c}&-\frac{1}{2}d \left( \eta_{W}(X_{W}) - \xi_{W} (Y_{W})\right).
\end{align}
From  (\ref{s2}),  $[X_{W}, Y_{W}]\sim_{\mathcal H} [X, Y] + \frac{ F(X, Y)}{a} X_{0}$. From (\ref{lie-forms}),  
${\mathcal L}_{X_{W}} \eta_{W}\sim_{\mathcal H}{\mathcal L}_{X}\eta  -\frac{\eta (X_{0})}{a}i_{X}F$ and 
${\mathcal L}_{Y_{W}} \xi_{W}\sim_{\mathcal H}  {\mathcal L}_{Y}\xi  - \frac{\xi (X_{0})}{a} i_{Y}F.$ 
From (\ref{contraction}), $\pi_{W}^{*} (\eta_{W} ( X_{W})) = \pi^{*} (\eta (X))$.
Taking the exterior derivative, we obtain 
$d( \eta_{W}(X_{W}) )\sim_{\mathcal H} d( \eta (X))$. A similar argument shows that 
$d(\xi_{W}(Y_{W}))\sim_{\mathcal H} d(\xi (Y))$. 
Combining these facts with (\ref{def-c}) we obtain the claim.
\end{proof}

To simplify terminology, we introduce the following definition.

\begin{defn} A subbundle $E\subset (TM)^{\mathbb{C}}$ is called {\cmssl $(F, a)$-involutive} if, for any sections $X, Y\in \Gamma (E)$, the complex vector field
$$
[X,Y]^{(F,a)}:= [X, Y] +\frac{ F(X, Y)}{a} X_{0}
$$ is also a section of $E$.
\end{defn}

If $E$ is  $(F, a)$-involutive and $\alpha \in \Gamma (\Lambda^{2}E^{*})$ then its twisted exterior 
differential $d^{(F, a)} \alpha\in\Gamma ( \Lambda^{3} E^{*})$  is defined by
\begin{align}
\nonumber& (d^{(F,a)} \alpha) (X, Y, Z):= X(\alpha (Y, Z)) + Z ( \alpha (X, Y)) + Y (\alpha (Z,X))\\  
\label{cond2} &+ \alpha ( Z, [X,Y]^{(F,a)} )
+ \alpha ( X, [Y,Z]^{(F,a)} )  + \alpha ( Y, [Z,X]^{(F,a)}),
\end{align}
for any $X, Y, Z\in \Gamma (E)$. Remark that, if  $E=(TM)^{\mathbb{C}}$,  
then $d^{(F,a)}\alpha = d_{W}\alpha$ (the latter defined in (\ref{s1})).

\begin{defn} A form $\alpha \in \Gamma (\Lambda^{2} E^{*})$ defined on an $(F,a)$-involutive bundle $E$ is called {\cmssl $(F,a)$-closed} if $d^{(F,a)} \alpha =0.$
\end{defn}

Using this preliminary material, we now find conditions which ensure that the twist $\mathcal J_{W}$ of an invariant
generalized almost complex structure $\mathcal J$ on $M$ is integrable. Recall that $\mathcal J_{W}$ is defined by 
\begin{equation}\label{def-gcs}
\mathcal  J_{W}(u_{W}) =( \mathcal J u)_{W},\ u\in \Gamma (\mathbb{T}M)^{\mathrm{inv}}.
\end{equation}
Let $L= L(E, \epsilon )$ be the $(1,0)$-bundle of $\mathcal J .$

\begin{thm}\label{cond-int-complex} In this setting,  
${\mathcal J}_{W}$ is integrable if and only if one of the following equivalent conditions holds:\

i) for any invariant sections $X+\xi , Y+\eta$ of $L$, the section
\begin{equation}\label{exp-r}
F(X, Y) X_{0} -  \eta (X_{0}) i_{X}F +  \xi (X_{0}) i_{Y}F +  a[ X+\xi , Y+\eta ]
\end{equation}
of $(\mathbb{T}M)^{\mathbb{C}}$  is  a section of $L$;\

ii) the bundle $E$ is $(F, a)$-involutive and 
$\epsilon \in \Gamma (\Lambda ^{2} E^{*})$ is $(F, a)$-closed.
\end{thm}

\begin{proof}  A straightforward computation which uses Lemma \ref{l-courant-twist} and relation (\ref{def-gcs})   shows that 
$N_{{\mathcal J}_{W}} (u_{W},v_{W})$, with $u, v\in \Gamma ( \mathbb{T}M)^{\mathrm{inv}}$, is $\mathcal H$-related to 
\begin{align}
\nonumber&{ \mathcal E} (u ,v ):= N_{\mathcal J} (u, v) + \frac{1}{a} F( \mathrm{pr}_{T} (\mathcal J u ),
\mathrm{pr}_{T}( \mathcal J v) )X_{0}-
\frac{\mathrm{pr}_{T^{*}}( \mathcal J v) (X_{0})}{a} { i_{\mathrm{pr}_{T}(\mathcal J u)} F}\\
\nonumber&  + \frac{\mathrm{pr}_{T^{*}} (\mathcal J u) (X_{0})}{a} { i_{\mathrm{pr}_{T}(\mathcal J v )} F}  -\frac{1}{a} F(\mathrm{pr}_{T}(u), \mathrm{pr}_{T}(v)) { X_{0}} 
+\frac{(\mathrm{pr}_{T^{*}}v) (X_{0})}{a} { i_{\mathrm{pr}_{T}(u)}F} \\
\nonumber&- \frac{\mathrm{pr}_{T^{*}}(u)(X_{0})}{a} i_{\mathrm{pr}_{T}(v)}F -\frac{1}{a} F(\mathrm{pr}_{T}( \mathcal J u), \mathrm{pr}_{T}(v)) { \mathcal J X_{0}}
+\frac{\mathrm{pr}_{T^{*}}(v) (X_{0})}{a}  \mathcal J i_{\mathrm{pr}_{T}( \mathcal J u)}F \\
\nonumber&-\frac{\mathrm{pr}_{T^{*}} (\mathcal J u)(X_{0})}{ a} { \mathcal J  i_{\mathrm{pr}_{T}(v)}F }-\frac{1}{a} F(\mathrm{pr}_{T}(u), \mathrm{pr}_{T}( \mathcal J v))  {\mathcal J X_{0}}\\ 
\nonumber& + \frac{\mathrm{pr}_{T^{*}} (\mathcal J v)(X_{0})}{a} {\mathcal J i_{\mathrm{pr}_{T} (u)} F } 
-\frac{\mathrm{pr}_{T^{*}}(u) (X_{0})}{a}
{\mathcal J i_{\mathrm{pr}_{T}( \mathcal J v) } F}.
\end{align}
We obtain that $\mathcal J_{W}$ is integrable if and only if 
\begin{equation}\label{courant-int}
{\mathcal E} ( u, v ) =0,\  \forall u,v\in \Gamma (\mathbb{T}M)^{\mathrm{inv}}.  
\end{equation}
We now consider in (\ref{courant-int}) various cases: a) $u, v \in \Gamma (L)^{\mathrm{inv}}$; b) $u\in \Gamma ({L})^{\mathrm{inv}}$, $v \in 
\Gamma (\bar{L})^{\mathrm{inv}}$; c) $u, v \in \Gamma (\bar{L})^{\mathrm{inv}}.$ 
In case b) relation (\ref{courant-int}) is automatically satisfied. In case 
a) relation (\ref{courant-int})  with $u= X+\xi$ and $v= Y+\eta$ becomes
\begin{align}
\nonumber  & - F(X, Y) ( X_{0}+ i \mathcal J X_{0}) +\eta (X_{0}) ( i_{X} F + i\mathcal J i_{X}F ) 
-\xi (X_{0}) (i_{Y}F + i\mathcal J i_{Y}F )\\
\label{l-l}&  = a( [X+\xi , Y+\eta ]+ i\mathcal J [X+\xi , Y+\eta ] ),
\end{align}
for any $X+\xi , Y+\eta \in \Gamma (L)^{\mathrm{inv}}$,
and in case c) it is equivalent (by conjugation) to (\ref{l-l}).
But relation  (\ref{l-l}) is equivalent to
claim i).

We now prove claim ii). The $(1,0)$-bundle of $\mathcal J_{W}$ is 
$L_{W} := L(E_{W}, \epsilon_{W})$, where $E_{W}\subset (TW)^{\mathbb{C}}$ is generated (locally) by sections
$X_{W}$, where $X\in \Gamma (E)^{\mathrm{inv}}$,  and $\epsilon_{W}\in \Gamma (\Lambda^{2}E^{*}_{W})$ 
is given by 
$$
\epsilon_{W}(X_{W}, Y_{W}) = \epsilon (X,Y)_{W},\ \forall X,Y\in\Gamma  (E)^{\mathrm{inv}}.
$$
The generalized almost complex structure  $\mathcal J_{W}$ is integrable if and only if 
$E_{W}$ is involutive, i.e.\ $E$ is $(F,a)$-involutive  
(from (\ref{s2}))  and $\epsilon_{W}$ is closed.
Using (\ref{s2}) and $X_{W}(f_{W}) = (Xf)_{W}$, for any invariant vector field $X\in {\mathfrak X}(M)$ and
 invariant function $f\in C^{\infty}(M)$ 
(see Lemma \ref{lie-forms-lemma}),
we obtain:
 for any $X, Y, Z\in \Gamma (E)^{\mathrm{inv}} $,
\begin{align}
\nonumber &X_{W} (\alpha_{W} ( Y_{W}, Z_{W})) = X_{W} ( \alpha (Y, Z)_{W}) = (X\alpha (Y, Z))_{W}\\
\nonumber & \alpha_{W} ( X_{W}, [Y_{W}, Z_{W} ]) = \alpha_{W} ( X_{W}, ([Y, Z]^{(F,a)})_{W} ) = \alpha (X, [Y,Z]^{(F,a)})_{W}
\end{align}
and similarly for cyclic permutations of $X$, $Y$ and $Z$.
We deduce that 
$$
d\epsilon_{W}(X_{W}, Y_{W}, Z_{W})= (d^{(F,a)}\epsilon (X,Y, Z))_{W},\ \forall X,Y,Z\in \Gamma (E)^{\mathrm{inv}} .
$$
In particular, $\epsilon_{W}$ is closed if and only if $\epsilon$ is $d^{(F,a)}$-closed.
\end{proof}

\begin{rem}{\rm The invariance of $X+\xi$ and $Y+\eta$ is not essential in 
Theorem \ref{cond-int-complex} i), which can be formulated  also without this condition.
This is true as $\mathbb{T}M$ admits local frames formed by invariant sections,
from relation (\ref{courant1}), and from the fact that $L$ is isotropic with respect to $\langle \cdot , \cdot \rangle.$  }
\end{rem}

\subsection{Examples}\label{examples-twist}

In this subsection we apply  Theorem \ref{cond-int-complex} to various particular cases.

\subsubsection{Twist of  symplectic and (deformations) of complex structures}

The conditions which ensure 
that the integrability of a complex or symplectic structure is preserved under twist 
were determined 
in \cite{swann} (see Section \ref{twist-section}).  In this section we rediscover these conditions using Theorem \ref{cond-int-complex}.   
We begin with the symplectic case.

\begin{cor}
Let $\omega$ be an invariant  symplectic form on $M$ and $\mathcal J_{\omega}$ the associated
generalized complex structure. Then  
$(\mathcal J_{\omega})_{W}$ is integrable (i.e. $\omega_{W}$ is symplectic) if and only if
$F\wedge i_{X_{0}} \omega  =0.$
\end{cor}

\begin{proof}  For any $X\in TM$ and $\xi \in (TM)^{*}$, $\mathcal J_{\omega }X = i_{X}\omega$ , 
$\mathcal J_{\omega}\xi  =  - \omega^{-1}(\xi )$ and 
the $(1,0)$-bundle of $\mathcal J_{\omega}$ is 
$L= L( (TM)^{\mathbb{C}}, -i\omega )$. The sections of $L$  are of the form 
$X- i (i_{X}\omega )$, for any $X\in (TM)^{\mathbb{C}}.$ From Theorem \ref{cond-int-complex}, 
$({\mathcal J}_{\omega})_{W}$ is integrable if and only if, for any  complex vectors $X$, $Y$,  
\begin{align}
\nonumber &F(Y,X) (X_{0}+ i (i_{X_{0}}\omega )) +i\omega (X_{0},Y)(i_{X}F - i \omega^{-1} (i_{X}F))\\
\label{verif-sympl}& + i\omega (X, X_{0}) ( i_{Y}F - i\omega^{-1}( i_{Y}F) )=0.
\end{align}
Identifying the $TM$ and $T^{*}M$ components  in (\ref{verif-sympl}) we obtain $F\wedge i_{X_{0}} \omega =0$, as required.
\end{proof}

We now turn to the complex case. We  consider a  more general setting, namely 
a complex structure $J$ on $M$  and  $\Pi\in \Gamma ( \Lambda^{2} TM)$, viewed as a homomorphism  
from $T^{*}M$ to $TM$.
Assume that  
\begin{equation}\label{types}
J\circ \Pi = \Pi \circ J^{*}.
\end{equation}
From (\ref{types}), (the complexification)  of  $\Pi$ maps $\Lambda^{1,0}M$ 
to $T^{1,0}M$ and $\Lambda^{0,1}M$ to $T^{0,1}M$
(the type $(1,0)$ and $(0,1)$ of forms/vectors are with respect to $J$).  Let 
$\sigma := \frac{i}{2} \Pi^{0,2}$, i.e. 
$$
\sigma : \Lambda^{1,0} M \rightarrow T^{1,0}M,\  \sigma (\xi ) :=\frac{i}{2} \Pi  (\xi),\ \forall  \xi \in \Lambda^{1,0}M.
$$  
From (\ref{types}),   
\begin{equation}\label{def-sigma}
{\mathcal J}_{J, \sigma} := \left( 
\begin{tabular}{cc}
$J$ & $ \Pi$\\
$0$ & $ - J^{*}$
\end{tabular}\right) 
\end{equation}
is a generalized almost complex structure, with $(0,1)$-bundle given by
$$
\bar{L} = \{ X +\sigma (\xi ) + \xi,\ X\in T^{0,1}M,\ \xi \in \Lambda^{1,0}M\} .
$$ 
It is well-known  that ${\mathcal J}_{J, \sigma}$ is integrable if and only if 
$\sigma$ is a holomorphic Poisson structure on $(M, J)$ (see  \cite{paul-1,hitchin}).   
We assume that this holds. Moreover, we assume that $J$ and $\Pi$ are invariant.

\begin{cor}i)  In the above setting, $({\mathcal J}_{J, \sigma})_{W}$ is integrable if and only if
\begin{equation}\label{required-p}
i_{X}F \in \sigma^{-1} (\mathrm{span}_{\mathbb{C}} \{ X_{0} - i JX_{0}\}  ),\ \forall X\in T^{0,1}M.
\end{equation}
In particular, if $(\mathcal J_{J, \sigma })_{W}$ is integrable then   $F$ is of type $(1,1)$.\

ii)  If  $\mathcal J_{J}$ is the generalized complex structure associated to an invariant  complex structure $J$, then  
$(\mathcal J_{J})_{W}$ is integrable if and only if $F$ is of type $(1,1)$. 
\end{cor}

\begin{proof} i)  From Theorem \ref{cond-int-complex} i), 
${\mathcal J}_{J, \sigma}$ is integrable if and only if, 
for any  $X + \sigma (\xi ) + \xi , Y+\sigma (\eta ) + \eta\in \bar{L}$ (where $X, Y\in T^{0,1}M$ and $\xi , \eta \in \Lambda^{1,0}M$), 
\begin{align}
\nonumber & F(X+ \sigma ( \xi ) , Y+\sigma (\eta ) ) (X_{0}- i \mathcal J_{J, \sigma}X_{0})\\ 
\nonumber & - \eta ( X_{0})
(i_{X+\sigma (\xi ) }F- i \mathcal J_{J, \sigma} i_{X+\sigma (\xi ) } F)\\ 
\label{f-x-y}&  + \xi (X_{0} )( i_{Y+\sigma (\eta )} F - i \mathcal J_{J, \sigma} i_{Y+\sigma (\eta ) }F ) =0.
\end{align}    
Using the definition (\ref{def-sigma})  of ${\mathcal J}_{J, \sigma}$ and identifying the vector and covector parts in
the above relation, we obtain that 
(\ref{f-x-y}) is equivalent to
\begin{equation}\label{red-1}
\eta (X_{0}) i_{X+\sigma (\xi ) } (F) - \xi (X_{0}) i_{Y +\sigma ( \eta )} F \in \Lambda^{1,0}M
\end{equation}
together with
\begin{equation}\label{red-2}
F(X+\sigma (\xi ), Y+\sigma (\eta )) ( X_{0} - i JX_{0}) = i \Pi\left(   \xi (X_{0})  i_{Y+\sigma (\eta )}F 
-  \eta  (X_{0})   i_{ X+\sigma (\xi  )}F \right) .
\end{equation}
Taking in (\ref{red-1}) $\eta=0$ we obtain $i_{Y} F \in \Lambda^{1,0} M$, i.e. $F$ is of type $(1,1).$ 
This, together with $\sigma (\xi ), \sigma (\eta ) \in T^{1,0}M$, imply that $i_{\sigma (\xi )} F$ and  
$i_{\sigma (\eta )} F$ are of type $(0,1).$ Relation (\ref{red-1}) is equivalent to
\begin{equation}\label{red-1-prime}
F\vert_{\Lambda^{2} T^{1,0}M}=0,\ \eta (X_{0}) i_{\sigma (\xi )} F =  \xi (X_{0})i_{\sigma (\eta )}F. 
\end{equation}  
From (\ref{red-1-prime}), relation (\ref{red-2}) becomes 
\begin{equation}\label{red-2-prime}
\left( F(\sigma (\xi ) , Y) - F(\sigma (\eta ), X)\right) (X_{0} - i JX_{0}) =
i \Pi  (   \xi (X_{0})    i_{ Y}F - \eta (X_{0}) i_{X} F).
\end{equation}   
We proved that $(\mathcal J_{J, \sigma })_{W}$ is integrable if and only if
(\ref{red-1-prime}) and (\ref{red-2-prime}) hold. We now study these relations. 
Relation  (\ref{red-2-prime}) with $\xi =0$ and $\eta (X_{0}) = 0$ gives 
\begin{equation}\label{red-2-prime-1}
F= 0\ \mathrm{on}\ T^{0,1}M \wedge \sigma (\mathrm{Ann} (X_{0})\cap \Lambda^{1,0}M).
\end{equation}
If $F$ is of type $(1,1)$, then relation (\ref{red-2-prime-1}) implies the second relation (\ref{red-1-prime}), because $\eta (X_{0}) \xi - \xi (X_{0})\eta \in \mathrm{Ann} (X_{0}) \cap \Lambda^{1,0}M.$ 
Thus, $(\mathcal J_{J})_{W}$ is integrable if and only if $F$ is of type $(1,1)$ and relation
(\ref{red-2-prime}) holds. We assume that $F$ is of type $(1,1)$ and we consider in more detail relation (\ref{red-2-prime}):

a) for  $\xi =0$ and $\eta (X_{0})=0$, we saw that it gives (\ref{red-2-prime-1});

b) for $\xi =0$ and $\eta (X_{0})\neq 0$ it gives  
\begin{equation}\label{simpl-def}
\Pi ( i_{X}F) = i F(X, \sigma (\frac{\eta}{\eta (X_{0})})) (X_{0} - i JX_{0}),\ \forall X\in T^{0,1}M.
\end{equation}
Remark that  relation  (\ref{simpl-def})  implies that
\begin{equation}\label{red-2-2}
i_{X} F \in \Pi^{-1} ( \mathrm{span}_{\mathbb{C}} \{ X_{0} - i JX_{0}\} ),\ \forall X\in T^{0,1}M 
\end{equation}
and an easy argument shows that the converse is also true, i.e. (\ref{simpl-def}) and (\ref{red-2-2}) are equivalent. 

The remaining cases in (\ref{red-2-prime}), namely: 
c)  $\xi ,\eta \in \Lambda^{1,0}M$  both non-zero with $\xi (X_{0} ) =0$ and $\eta (X_{0})\neq 0$; 
d)  $\xi , \eta \in \mathrm{Ann}(X_{0})\cap \Lambda^{1,0}M$  both non-zero; e) $\xi ,\eta \in \Lambda^{1,0}M$ with $\xi (X_{0} ) \neq 0$ and $\eta (X_{0})\neq 0$,  all follow from (\ref{red-2-prime-1}) and (\ref{simpl-def}). 

To summarize:  we  proved that $(\mathcal J_{J, \sigma})_{W}$ is integrable if and only if $F$ is of type $(1,1)$ 
and relations
(\ref{red-2-prime-1}),  (\ref{simpl-def}) hold.  
It is easy to show that these conditions are equivalent to 
(\ref{required-p}).
This concludes claim i). Claim ii) follows from claim i), by taking $\Pi =0.$ 
\end{proof}

\subsubsection{Twist of interpolation between complex and symplectic structures}
We now apply  Theorem \ref{cond-int-complex} to a family of generalized complex structures, which interpolate between complex and symplectic
structures. Let $(g, I, J, K)$ be an invariant  hyper-K\"{a}hler structure on $M$, with K\"{a}hler forms $\omega_{I}$, $\omega_{J}$ and
$\omega_{K}.$  For any $t\in [ 0, \frac{\pi}{2})$, let
$\mathcal J_{t}:=\mathrm{sin}(t) \mathcal J_{I} + \mathrm{cos}(t) \mathcal J_{\omega_{J}}.$ 
Then $\mathcal J_{t}$ is an (integrable)
generalized complex structure (see \cite{gualtieri}, page 55).

\begin{cor} The generalized almost complex structure $(\mathcal J_{t})_{W}$ is integrable if and only if 
$F= f ( i_{X_{0}} \omega_{K}) \wedge ( i_{X_{0}} \omega_{J})$, where  $f\in C^{\infty}(M)$ 
is invariant and 
$df \wedge (i_{X_{0}}\omega_{K}) \wedge ( i_{X_{0}} \omega_{J}) =0.$ 
\end{cor}

\begin{proof}
Let $B_{t}:=  \mathrm{tan}(t) \omega_{K}.$ As proved in \cite{gualtieri}, 
  $e^{B_{t}} \mathcal J_{t} e^{-B_{t}} = \mathcal J_{\mathrm{sec}(t)\omega_{J}}$.
We deduce that the $(1,0)$-bundle of $\mathcal J_{t}$ is
$L ((TM)^{\mathbb{C}}, - B_{t} - i \mathrm{sec}(t) \omega_{J}).$ From Theorem \ref{cond-int-complex}, we deduce that 
$(\mathcal J_{t})_{W}$ is integrable if and only if $d^{(F,a)} ( B + i\mathrm{sec}(t) \omega_{J}) =0.$ Using that 
$B$ and $\omega_{J}$ are closed, and the formula
\begin{equation}\label{alpha-p}
d^{(F,a)}\alpha  = d\alpha -\frac{1}{a} F\wedge  i_{X_{0}} \alpha ,\quad \forall \alpha \in \Omega^{k}(M),
\end{equation}
we obtain that $\mathcal J_{t}$ is integrable if and only if
$i_{X_{0}} ( B_{t} + i \mathrm{sec}(t) \omega_{J})\wedge F =0$, i.e. 
\begin{equation}\label{i-j-k}
(i_{X_{0}} \omega_{K}) \wedge F =0,\ ( i_{X_{0}}\omega_{J}) \wedge F =0.
\end{equation}
Relations  (\ref{i-j-k}) together with $dF =0$  imply our claim. 
\end{proof}

\subsubsection{Twist and  conformal change} For the KK correspondence developed later in the paper, 
we need to understand when the twist of the conformal change of a generalized almost complex structure is integrable. This is done in the next proposition.

\begin{prop}\label{int-conf} Let $\mathcal J$ be an invariant  generalized almost complex structure on $M$, with 
$(1,0)$-bundle $L = L(E, \epsilon )$,  and $h\in C^{\infty}(M)$ an invariant non-vanishing function.Then the 
twist $[{\tau_{h} ({\mathcal J})}]_{W}$ of the conformal change $\tau_{h}(\mathcal J )$ of  $\mathcal J$ by $h$ is
integrable if and only if one of the following conditions hold:\

i) for any invariant sections $X+ \xi , Y+\eta$ of $L$, the expression
\begin{align}
\nonumber & - F(X, Y) X_{0} + \eta (X_{0}) i_{X}F - \xi (X_{0}) i_{Y}F + \frac{2a}{h} \left( X(h) \eta - Y(h) \xi \right)\\
\label{conf-change}&  - \frac{a}{h}( \eta (X) - \xi (Y)) dh - a[ X+\xi , Y+\eta ]
\end{align}
is  a section of $L$;\

ii) the bundle $E$ is $(F,a)$-involutive and
\begin{equation}\label{cond2}
d^{(F,a)} \epsilon = \frac{2}{h} \epsilon \wedge dh\vert_{E}.
\end{equation}

\end{prop}

\begin{proof}  Let $L^{h}$ be the $(1,0)$-bundle of $\tau_{h} (\mathcal J)$ and $\tilde{h}:= \frac{1}{h^{2}}.$
Then  $X+\xi \in L$ if and only if $X+ \tilde{h} \xi \in L^{h}.$ 
From Theorem \ref{cond-int-complex}  i), $[{\tau_{h} ({\mathcal J})}]_{W}$ is integrable if and only if,
for any $X+\xi , Y+\eta\in \Gamma  (L)^{\mathrm{inv}}$,
$$
- F(X, Y) X_{0} + (\tilde{h}\eta )(X_{0}) i_{X}F - (\tilde{h} \xi )(X_{0}) i_{Y}F - a[ X+\tilde{h}\xi ,
Y+\tilde{h} \eta ] \in \Gamma (L^{h}),
$$
or 
\begin{align}
\nonumber&- F(X, Y) X_{0} + \eta (X_{0}) i_{X}F - \xi (X_{0}) i_{Y}F - \frac{a}{\tilde{h}} \mathrm{pr}_{T^{*}}[ X+\tilde{h}\xi ,
Y+\tilde{h} \eta ]\\
\label{rel-adaugat}&- a\mathrm{pr}_{T}  [X+\tilde{h} \xi  , Y+\tilde{h}\eta ] \in \Gamma (L).
\end{align}
On the other hand, 
\begin{align}
\nonumber&[X+\tilde{h}\xi  , Y+\tilde{h} \eta ] = [X, Y] + X(\tilde{h}) \eta - Y(\tilde{h}) \xi + \tilde{h}
( {\mathcal L}_{X} \eta  -{\mathcal  L}_{Y} \xi  )\\
\label{conformal-courant}&-\frac{1}{2} \tilde{h} d (\eta (X) - \xi (Y)) -\frac{1}{2} (\eta (X) - \xi (Y)) d\tilde{h}.
\end{align}
Replacing this relation in (\ref{rel-adaugat})  we obtain  (\ref{conf-change}), as needed.  
Relation (\ref{cond2}) follows  from Theorem \ref{cond-int-complex} ii) and 
$L^{h} = L(E, \tilde{h}\epsilon )$, using that 
$d^{(F,a)} (\tilde{h}\epsilon ) = (d\tilde{h})\vert_{E}\wedge \epsilon + \tilde{h} d^{(F,a)} \epsilon .$
\end{proof}

\section{KK correspondence in generalized complex geometry}\label{KK-section}

\subsection{Twist of generalized almost Hermitian structures}\label{ah-section}

Let $(G, \mathcal J)$ be an invariant  generalized almost Hermitian structure on $M$. As usual, we denote by $L_{i} = L(E_{i}, \epsilon_{i})$
the $(1,0)$-bundles of its generalized almost complex structures $\mathcal J_{1}= \mathcal J$ and $\mathcal J_{2}=
G^{\mathrm{end}}\mathcal J$. We define a $2$-form   
$$
\epsilon \in\Gamma (  \mathrm{pr}_{T} (L_{1} \cap {L}_{2})^{*}\wedge (E_{1}+ {E}_{2})^{*})
$$ by
\begin{equation}
\epsilon  (X, \cdot )\vert_{E_{1}} := \epsilon_{1} (X, \cdot),\quad \epsilon  (X, \cdot )\vert_{E_{2}}:= \epsilon_{2} (X, \cdot),
\end{equation}
for any $X\in \mathrm{pr}_{T} (L_{1} \cap L_{2})$. 
The form $\epsilon$ is well defined:  from
(\ref{epsilon}),   $\epsilon_{1} (X,\cdot ) := \xi\vert_{E_{1}}$, where $\xi\in 
(T^{*}M)^{\mathbb{C}}$ is arbitrary, such
that $X+\xi \in L_{1}$; similarly, $\epsilon_{2}(X,\cdot ) = \eta\vert_{E_{2}}$, where 
$\eta\in (T^{*}M)^{\mathbb{C}}$ is arbitrary, such that  $X+\eta \in L_{2}$; since
$X\in \mathrm{pr}_{T}(L_{1} \cap L_{2})$, we can take $\xi = \eta$ and we obtain that 
$\epsilon_{1}(X, \cdot ) \vert_{E_{1} \cap E_{2}} = \epsilon_{2} (X.\cdot )\vert_{ E_{1} \cap E_{2}}$ 
as claimed. Similarly, the $2$-form 
$$
\tilde{\epsilon} \in\Gamma ( \mathrm{pr}_{T} (L_{1} \cap \bar{L}_{2} )^{*}\wedge (E_{1}+ \bar{E}_{2})^{*})
$$ 
given by
\begin{equation}
\tilde{\epsilon}  (X, \cdot )\vert_{E_{1}} := \epsilon_{1} (X, \cdot),\quad  
\tilde{\epsilon} (X, \cdot )\vert_{\bar{E}_{2}} := \bar{\epsilon}_{2} (X, \cdot),
\end{equation}
for any $X\in \mathrm{pr}_{T} (L_{1} \cap \bar{L}_{2}) $, is well defined. Above $\bar{\epsilon}_{2}\in
\Gamma ( \Lambda^{2} \bar{E}_{2}^{*})$ is defined by 
$$
\bar{\epsilon}_{2} (\bar{X}, \bar{Y}) := \overline{ \epsilon_{2} (X, Y)},\ X, Y\in E_{2}.
$$

Let $h\in C^{\infty}(M)$ be an invariant,   non-vanishing function.

\begin{thm}\label{prop-twist}
In the above setting,  the twist  $[\tau_{h} (G, {\mathcal J})]_{W}$ of the conformal change
$\tau_{h} (G, \mathcal J )$ of
$(G, {\mathcal J})$ by $h$ is generalized K\"{a}hler  if and only if the following conditions hold:\

i) the bundles $E_{1}$, $\mathrm{pr}_{T} (L_{1}\cap {L}_{2})$ and
$\mathrm{pr}_{T} (L_{1}\cap \bar{L}_{2})$ are $(F,a)$-involutive. Moreover, 
for any  $X\in \Gamma  \mathrm{pr}_{T} (L_{1} \cap L_{2})$ and $Y\in \Gamma 
\mathrm{pr}_{T} (\bar{L}_{1} \cap L_{2})$,
\begin{equation}\label{f-a-inv}
[X, Y]^{(F,a)} \in \Gamma ( E_{1} + E_{2})
\end{equation}
and for any  $X\in \Gamma  \mathrm{pr}_{T} (L_{1} \cap \bar{L}_{2})$ and $ Y\in \Gamma \mathrm{pr}_{T}(\bar{L}_{1}\cap 
\bar{L}_{2})$, 
\begin{equation}\label{f-a-inv-1}
[X, Y]^{(F,a)}\in \Gamma ( E_{1} + \bar{E}_{2}). 
\end{equation}

ii) the forms $\epsilon_{1}$, $\epsilon$ and $\tilde{\epsilon}$ satisfy  the relations
\begin{align}
\nonumber & d^{(F,a)} \epsilon_{1} = \frac{2}{h} \epsilon_{1}\wedge dh\vert_{ E_{1}}\\ 
\nonumber & d^{(F,a)} \epsilon = \frac{2}{h} \epsilon\wedge dh\ \mathrm{on}\  \Lambda^{2}\mathrm{pr}_{T} (L_{1} \cap {L}_{2})\wedge \mathrm{pr}_{T}
(\bar{L}_{1} \cap L_{2})\\
\label{f-a-herm-1}& d^{(F,a)} \tilde{\epsilon} = \frac{2}{h}\tilde{\epsilon}\wedge dh\ \mathrm{on}\  
\Lambda^{2} \mathrm{pr}_{T} (L_{1} 
\cap \bar{L}_{2})\wedge  \mathrm{pr}_{T} (\bar{L}_{1}\cap \bar{L}_{2})
\end{align}
(which, owing to i), are well-defined).

\end{thm}

\begin{proof} 
Let 
${\mathcal J}_{i}^{h} =\tau_{h}\circ \mathcal J_{i}\circ \tau_{h}^{-1}$  ($i=1,2$) be the generalized almost complex
structures of the conformal change $\tau_{h}(G, \mathcal J )$.
The $(1,0)$-bundle of $\mathcal J_{i}^{h}$ is $L_{i}^{h} = L(E_{i}, \tilde{h} \epsilon_{i})$, where $\tilde{h}:= \frac{1}{h^{2}}$.
Let $(L_{i}^{h})_{W}$ be the $(1,0)$- bundle of $(\mathcal J_{i}^{h})_{W}$.
It is generated by sections of the form $u_{W}$, where $u$ is an invariant section of $L_{i}^{h}.$
From Gualtieri's characterization of generalized K\"{a}hler 
structures (see Section \ref{complex-geometry}), $[\tau_{h} (G, \mathcal J )]_{W}$ is generalized K\"{a}hler if and only if 
the following conditions hold:\

a) $(\mathcal J_{1}^{h})_{W}$ is a generalized complex structure;\

 b) $(L_{1}^{h})_{W} \cap (L_{2}^{h})_{W}$ is Courant integrable;\

 c) $(L_{1}^{h})_{W} \cap (\bar{L}_{2}^{h})_{W}$ is Courant integrable.\

From Proposition \ref{int-conf}, 
condition a) is equivalent to the $(F,a)$-involutivity of $E_{1}$ and to the 
first relation (\ref{f-a-herm-1}). From now on  we assume that a) holds. Since   $X+\xi \in \Gamma (L_{i})$ if and only if $X+\tilde{h} \xi \in \Gamma (L_{i}^{h})$, we obtain
that condition b) is equivalent to
\begin{equation}\label{cond-gen-kahler}
[ (X+\tilde{h}\xi )_{W}, (Y+\tilde{h}\eta )_{W}] \in \Gamma ( (L_{1}^{h})_{W}\cap (L_{2}^{h})_{W}),\  \forall X+\xi , Y+\eta \in \Gamma  (L_{1} \cap L_{2})^{\mathrm{inv}} .
\end{equation}
Using  Lemma  \ref{l-courant-twist}, relation 
(\ref{conformal-courant}) and
$(L_{1}^{h})_{W}\cap (L_{2}^{h})_{W} = (L_{1}^{h}\cap L_{2}^{h})_{W}$,
we see that (\ref{cond-gen-kahler}) becomes
\begin{align}
\nonumber& [X, Y] +\frac{F(X, Y)}{a} X_{0} +
\frac{1}{\tilde{h}} ( X(\tilde{h}) \eta - Y(\tilde{h}) \xi ) + {\mathcal L}_{X}\eta  - {\mathcal L}_{Y}\xi   -\frac{1}{2}  d (\eta (X) - \xi (Y))\\
\label{expression-ajut}& -\frac{1}{2\tilde{h}} (\eta (X) - \xi (Y) ) d\tilde{h} - \frac{\eta (X_{0})}{a} i_{X} F
+  \frac{\xi (X_{0})}{a} i_{Y} F \in \Gamma (L_{1} \cap L_{2}),
\end{align}
for any  $X+\xi , Y+\eta \in \Gamma   (L_{1} \cap L_{2})^{\mathrm{inv}}$, which is  equivalent to the following two conditions:\

$\bullet$   $[X, Y]^{(F,a)}$ is a section of $\mathrm{pr}_{T}(L_{1} \cap L_{2})$, for any $X, Y\in \Gamma \mathrm{pr}_{T} (L_{1} \cap L_{2})$, i.e.
$\mathrm{pr}_{T} (L_{1} \cap L_{2})$  is $(F, a)$-involutive;\

$\bullet$ the left hand side of (\ref{expression-ajut}) is a section of $L_{2}$, i.e.  
for any $V\in \Gamma (E_{2})$, 
\begin{align}
\nonumber& \frac{1}{\tilde{h}} ( X(\tilde{h}) \eta (V) - Y(\tilde{h}) \xi (V)) + ({\mathcal L}_{X}\eta ) (V)- 
({\mathcal L}_{Y}\xi ) (V) -\frac{1}{2}  V (\eta (X) - \xi (Y))\\
\nonumber & -\frac{1}{2\tilde{h}} (\eta (X) - \xi (Y) ) V(\tilde{h}) - \frac{\eta (X_{0})}{a} F(X, V)
+  \frac{\xi (X_{0})}{a} F(Y, V)  \\
\label{ii}&  =\epsilon_{2} ( [X, Y]^{(F,a)}, V).
\end{align}
(From a),   the 
left-hand side  
of (\ref{expression-ajut})   belongs to $\Gamma (L_{1})$ as well).
Assume now that  $\mathrm{pr}_{T}(L_{1} \cap L_{2})$ is $(F,a)$-involutive. Under this assumption we prove  that  (\ref{ii}) is equivalent to
(\ref{f-a-inv}) together with  the second relation 
(\ref{f-a-herm-1}).  For this, we remark,  from  $Y+\eta   \in \Gamma (L_{1} \cap L_{2})$ (in particular, $Y+\eta$ is a section of $L_{2}$) and $V\in \Gamma (E_{2})$, 
$$
({\mathcal L}_{X} \eta )(V) =  X(\eta (V)) - \eta ([X, V]) = X ( \epsilon_{2} (Y, V)) - \eta  ([X, V]) . 
$$
Similarly, 
$$
({\mathcal L}_{Y} \xi )(V) =  Y ( \epsilon_{2} (X, V)) - \xi  ([Y, V]) .
$$
With these relations, (\ref{ii}) becomes 
\begin{align}
\nonumber&-\frac{2}{h} \left( X(h) \epsilon_{2} (Y, V) + Y(h) \epsilon_{2}(V, X) + V(h) \epsilon_{2} (X, Y)\right)\\
 \nonumber& + X ( \epsilon_{2} (Y, V)) + Y ( \epsilon_{2} (V,X)) +V ( \epsilon_{2} (X,Y))\\
 \label{ii-rescris}& + \eta ( [V, X]^{(F,a)}) + \xi ( [Y, V]^{(F,a)})+ \epsilon_{2} (V, [ X, Y]^{(F,a)}) =0.
 \end{align}
 Since $X+\xi , Y+\eta \in \Gamma (L_{1} \cap L_{2})$, the $1$-forms $\xi$ and $\eta$ are determined by $X$ and, respectively, by  $Y$, on $E_{1} + E_{2}$,
 but outside this bundle they take arbitrary values. 
 Therefore, relation (\ref{ii-rescris}) implies  that 
 \begin{equation}\label{large}
 [V, X]^{(F,a)}\in \Gamma (E_{1} + E_{2}),\quad 
 \forall X\in \Gamma \mathrm{pr}_{T} (L_{1} \cap L_{2}),\ V\in \Gamma (E_{2}),
 \end{equation}
 which  is equivalent to (\ref{f-a-inv}), when  $\mathrm{pr}_{T}(L_{1} \cap L_{2})$ is $(F,a)$-involutive
 (by decomposing $E_{2} = \mathrm{pr}_{T} (L_{1} \cap L_{2}) + \mathrm{pr}_{T} (\bar{L}_{1} \cap L_{2})$).
 Moreover, if $\mathrm{pr}_{T} (L_{1} \cap L_{2})$ is $(F,a)$-involutive and 
(\ref{f-a-inv}) holds, then 
$d^{(F,a)}\epsilon $ is defined  on $\Lambda^{2} \mathrm{pr}_{T} (L_{1} \cap {L}_{2})\wedge E_{2}$ and 
(\ref{ii-rescris}) is equivalent to
 \begin{equation}\label{rescris-large}
 d^{(F,a)} \epsilon = \frac{2}{h} \epsilon\wedge dh\ \mathrm{on}\  \Lambda^{2} \mathrm{pr}_{T} (L_{1} \cap {L}_{2})\wedge E_{2}.
 \end{equation}
 Decomposing $E_{2} = \mathrm{pr}_{T} ( L_{1} \cap L_{2}) + \mathrm{pr}_{T} (\bar{L}_{1} \cap L_{2})$ again and using
  the first relation (\ref{f-a-herm-1}) (which holds because condition a) holds)
 we obtain that (\ref{rescris-large}) is equivalent to the second relation  (\ref{f-a-herm-1}).

 We proved that if condition a) holds, then condition b) is equivalent to the $(F,a)$-involutivity of $\mathrm{pr}_{T} (L_{1} \cap L_{2})$,  together with relation
 (\ref{f-a-inv}) and   the second  relation (\ref{f-a-herm-1}). 
 A similar argument shows that  if condition a) holds,
 then condition c) is equivalent to the $(F,a)$-involutivity of $\mathrm{pr}_{T} (L_{1} \cap \bar{L}_{2})$,  together with relation
 (\ref{f-a-inv-1}) and   the third  relation (\ref{f-a-herm-1}). 
  \end{proof}

\subsection{Statement of the KK correspondence}

Let $(G, \mathcal J )$ be an invariant generalized K\"{a}hler structure on  $M$,
$\mathcal J_{1}=\mathcal J$, $\mathcal J_{2} = G^{\mathrm{end}}\mathcal J$  its generalized complex structures 
with  $(1,0)$-bundles
$L_{i} = L  (E_{i}, \epsilon_{i})$ ($i=1,2$).
Assume that the vector field $X_{0}$ is Hamiltonian Killing  
on $(M,G, \mathcal J )$, with Hamiltonian function $f^{H}$, 
and let $f, h\in C^{\infty}(M)$ be invariant, non-vanishing functions.  Let $(G^{\prime}, {\mathcal J})$ be the elementary deformation 
of $(G, {\mathcal J })$ by  $X_{0}$ and  $f$.  
We assume that $f^2-1$ is non-vanishing and denote by $\alpha$ the $1$-form
\begin{equation}\label{alpha}
\alpha := - d \left( \mathrm{ln}\frac{| f^{2} -1|}{f^{2} G(X_{0}, X_{0})}\right) . 
\end{equation}

\begin{thm} \label{gen-KK} The twist $[\tau_{h} (G^{\prime}, \mathcal J )]_{W}$ 
of the conformal change $\tau_{h} (G^{\prime}, {\mathcal J })$ 
of $(G^{\prime}, {\mathcal J })$ by $h$ is generalized K\"{a}hler if and only 
if the following conditions i) - v) hold:

\begin{enumerate}
\item[i)]
The curvature $F$ vanishes on 
$$ \Lambda^{2}\mathrm{pr}_{T} ( L_{1} \cap L_{2} \cap \mathcal S_{\mathbb{C}}^{\perp})\oplus 
\Lambda^{2} \mathrm{pr}_{T} ( L_{1} \cap \bar{L}_{2} \cap \mathcal S_{\mathbb{C}}^{\perp})
$$ 
and
\begin{equation}\label{F}
F(X, Y) X_{0} \in \Gamma (E_{1}),\   \forall X\wedge Y\in \Lambda^{2}E_{1}.
\end{equation}
\item[ii)]  For any  $X\in \Gamma \mathrm{pr}_{T} (L_{1} \cap L_{2} \cap \mathcal S^{\perp}_{\mathbb{C}} )$, 
\begin{align}
\nonumber& [\mathrm{pr}_{T} ( \mathcal J_{3} X_{0} ) , X] +\alpha (X)  \mathrm{pr}_{T} (\mathcal J_{3} X_{0} ) +
\frac{f^{2}F(\mathrm{pr}_{T}(v_{f}), X)}{a( f^{2}-1)} X_{0}\\
\label{cond-long-m}&\in \Gamma \mathrm{pr}_{T} (L_{1} \cap L_{2} \cap \mathcal S^{\perp}_{\mathbb{C}})
\end{align}
and for any 
$X\in \Gamma \mathrm{pr}_{T} ( L_{1} \cap \bar{L}_{2} \cap \mathcal S^{\perp}_{\mathbb{C}})$, 
\begin{align}
\nonumber& [\mathrm{pr}_{T} ( \mathcal J_{3} X_{0} ) , X] +\alpha (X)  \mathrm{pr}_{T} (\mathcal J_{3} X_{0} ) 
- \frac{f^{2}F(\mathrm{pr}_{T}(v_{if}), X)}{a(f^{2}-1)} X_{0}\\
\label{cond-long-1-m}& \in \Gamma \mathrm{pr}_{T} (L_{1} \cap \bar{L}_{2} \cap \mathcal S^{\perp}_{\mathbb{C}}).
\end{align}
\item[iii)]  For any $X\in \mathrm{pr}_{T}(\mathcal S^{\perp})$, 
\begin{equation}\label{log}
X (af^{2}h^{2}) =0.
\end{equation}
\item[iv)]  The following algebraic conditions on $\epsilon_{1}$ and $\epsilon_{2}$ hold:
 \begin{align}
\label{cond-e} & (i_{X_{0}} \epsilon_{1})  \wedge F= -\frac{2a}{h}\epsilon_{1}\wedge dh\ \mathrm{on}\ \Lambda^{3}E_{1}\\
\nonumber& \epsilon_{2} \wedge dh =0\ \mathrm{on}\ \mathrm{pr}_{T} (L_{1} \cap L_{2} \cap \mathcal S_{\mathbb{C}}^{\perp}) \wedge \mathrm{pr}_{T}(\bar{L}_{1} \cap L_{2} \cap \mathcal S_{\mathbb{C}}^{\perp})
\wedge \mathrm{pr}_{T} (L_{2} \cap \mathcal S_{\mathbb{C}}^{\perp}) .
\end{align}

\item[v)]  On $\Lambda^{2}\mathrm{pr}_{T} (L_{2} \cap \mathcal S_{\mathbb{C}}^{\perp})$, 
\begin{align}
\nonumber &  d\left( \frac{1- f^{2}}{f^{2} G(X_{0},X_{0})} \mathrm{pr}_{T^{*}} (\mathcal J_{3} X_{0}) \right)
+\frac{f^{2} -1}{ f^{2} G(X_{0}, X_{0}) } {\mathcal D}_{\mathrm{pr}_{T} (\mathcal J_{3} X_{0})} \epsilon_{2} 
+\frac{2}{af^{2}} F\\
\label{cond-de}& = \frac{2}{hG(X_{0}, X_{0})}  \left( \mathrm{pr}_{T} (v_{f})(h) \epsilon_{2} +\mathrm{pr}_{T^{*}}(v_{f}) \wedge dh \right) , 
\end{align}
where, for any  $X\wedge Y\in \Lambda^{2} \mathrm{pr}_{T} (L_{2} \cap \mathcal S_{\mathbb{C}}^{\perp})$,  
\begin{align}
\nonumber (\mathcal D_{\mathcal J_{3} X_{0}} \epsilon_{2})(X, Y)&:= (\mathrm{pr}_{T} \mathcal J_{3} X_{0})(\epsilon_{2} (X, Y))\\ 
\nonumber& -  \epsilon_{2} ( [\mathrm{pr}_{T} \mathcal J_{3}X_{0}, X] + \alpha (X) \mathrm{pr}_{T} \mathcal J_{3} X_{0}, Y) \\
\label{dd-e2} & - \epsilon_{2} (X,  [\mathrm{pr}_{T} \mathcal J_{3}X_{0}, Y] + \alpha (Y) \mathrm{pr}_{T} \mathcal J_{3} X_{0}) .
\end{align}
\end{enumerate}
\end{thm}

Now we give more detailed explanations for some of the relations from the above theorem.

\begin{rem}{\rm  i)  Relation (\ref{F}) does not imply, a priori, that $X_{0}\in \Gamma (E_{1})$
(the form $F$ could vanish on $\Lambda^{2}E_{1}$), but it does imply that the $3$-form
$(i_{X_{0}}\epsilon_{1} )\wedge F$, which appears in the first relation
(\ref{cond-e}),  is well defined (in the usual way).\

ii)  The form  $\mathcal D_{\mathrm{pr}_{T}(\mathcal J_{3} X_{0})} \epsilon_{2}$, as given in (\ref{dd-e2}), is well-defined, owing to condition ii) from Theorem \ref{gen-KK}:  this condition implies that 
$$ 
[\mathrm{pr}_{T}( \mathcal J_{3}X_{0}), X] + 
\alpha (X) \mathrm{pr}_{T} (\mathcal J_{3} X_{0})\in \Gamma (E_{2}),\ \forall  X\in \Gamma \mathrm{pr}_{T} (L_{2} \cap \mathcal S^{\perp}_{\mathbb{C}})
$$ 
(decompose $\mathrm{pr}_{T}( L_{2} \cap \mathcal S_{\mathbb{C}}^{\perp})$ into the sum of 
$\mathrm{pr}_{T}( L_{1} \cap L_{2} \cap \mathcal S_{\mathbb{C}}^{\perp})$ and  
$\mathrm{pr}_{T}(\bar{L}_{1} \cap L_{2} \cap \mathcal S_{\mathbb{C}}^{\perp})$
and use that $X_{0} \in \Gamma (E_{2})$, which holds since $X_{0} - i \mathcal J_{2} X_{0}\in \Gamma (L_{2})$ and  
$\mathcal J_{2} X_{0} \in \Omega^{1}(M)$).} 
\end{rem}

The next sections are devoted to the proof of Theorem \ref{gen-KK}. The plan is the following.
Section  \ref{brackets-section} is a preliminary part of the proof. Here we compute  various Courant brackets, which will be useful  in our argument.
In order to prove Theorem \ref{gen-KK},  we apply Theorem \ref{prop-twist} to the
generalized almost Hermitian structure $(G^{\prime}, \mathcal J )$ and the conformal function $h$.  Namely, in
Section \ref{invol-sect}  we 
prove that  the generalized almost Hermitian structure  
$(G^{\prime}, \mathcal J)$ satisfies condition i) from Theorem \ref{prop-twist}
if and only if 
conditions i) and  ii) from Theorem \ref{gen-KK} are  satisfied. 
In Section \ref{diff-section} we assume that conditions i) and  ii) from Theorem \ref{gen-KK} 
are  satisfied,  and we prove  that
condition ii) from
Theorem \ref{prop-twist} (applied to the generalized almost Hermitian structure $(G^{\prime}, \mathcal J)$ and function $h$) is equivalent to the   
remaining conditions iii) - v) from Theorem \ref{gen-KK}.

\subsubsection{Various Courant brackets}\label{brackets-section}

We begin by computing the Courant brackets of the canonical basis of $\mathcal S$.

\begin{lem}\label{Courant-brackets} i) The Courant bracket of $X_{0}$ with  $\mathcal J_{i} X_{0}$ ($i=1,3$) is given by 
\begin{equation}\label{p1-1}
L_{X_{0}} (\mathcal J X_{0}) =L_{X_{0}} (\mathcal J_{2} X_{0}) = 0,\ L_{X_{0}} ( {\mathcal J}_{3} X_{0}) = d G(X_{0}, X_{0}).
\end{equation}

ii) The Courant bracket of $\mathcal J X_{0}$ with $\mathcal J_{2}X_{0}$ and $\mathcal J_{3} X_{0}$ is given by
\begin{equation}\label{der-j}
[{\mathcal J X_{0}} ,\mathcal J_{2} X_{0}]=  d G(X_{0}, X_{0}),\  [{\mathcal J X_{0}} ,\mathcal J_{3} X_{0}] = 2\mathcal J  d G(X_{0}, X_{0}) .
\end{equation}

iii) The Courant bracket of $\mathcal J_{2} X_{0}$ with $\mathcal J_{3} X_{0}$ is trivial:
\begin{equation}\label{der-jj}
[ \mathcal J_{2} X_{0}, \mathcal J_{3} X_{0} ]=0.
\end{equation}
\end{lem}

\begin{proof}  We use relation (\ref{L-C}) and 
${\mathcal L}_{X_{0}} (\mathcal J ) =0$, ${\mathcal L}_{X_{0}} (G^{\mathrm{end}} ) =0$ 
(see Definition \ref{def-HK}).
The first relation (\ref{p1-1}) follows from
$$
L_{X_{0}} (\mathcal J X_{0} ) = {\mathcal L}_{X_{0}} (\mathcal J X_{0}) - d \langle X_{0}, \mathcal J X_{0} \rangle =0.
$$
In a similar way we obtain the other relations (\ref{p1-1}).
Let us prove (\ref{der-j}): from the definition of the Courant bracket and $\mathcal J_{2} X_{0} = df^{H}$,   
\begin{align}
\nonumber& [\mathcal J X_{0}, \mathcal J_{2} X_{0}] =  [\mathcal J X_{0},  df^{H}] =
 {\mathcal L }_{\mathrm{pr}_{T} \mathcal J X_{0}} (df^{H}) -\frac{1}{2} 
d \left( (df^{H})( \mathrm{pr}_{T} \mathcal J X_{0})\right)\\
\label{rel-added}& =\frac{1}{2} d \left( df^{H} (\mathrm{pr}_{T} \mathcal J X_{0})\right) =
d\langle df^{H}, \mathcal J X_{0}\rangle = dG(X_{0}, X_{0}),
\end{align} 
which is the first relation (\ref{der-j}). For the second relation (\ref{der-j}) we use that
$N_{\mathcal J }(X_{0}, \mathcal J_{2} X_{0})=0$, i.e.
\begin{equation}\label{nj}
[ \mathcal J X_{0}, \mathcal J_{3} X_{0} ] 
- [ X_{0} , \mathcal J_{2}X_{0}] = \mathcal J ( [\mathcal J X_{0},\mathcal J_{2} X_{0}] + 
[X_{0},\mathcal J_{3} X_{0}]).   
\end{equation}
The Courant brackets $[X_{0},\mathcal J_{2} X_{0} ]$ and $[X_{0} , \mathcal J_{3} X_{0}]$ were computed
in   (\ref{p1-1}) and the 
Courant bracket $[\mathcal J X_{0},\mathcal J_{2} X_{0}]$ in  (\ref{rel-added}). 
Using (\ref{nj}) we obtain  the  second relation (\ref{der-j}). Relation (\ref{der-jj}) can be proved equally  easy. 
\end{proof}

\begin{cor}\label{vfif} i) The Courant bracket $[ {v}_{f}, \bar{v}_{if}]$ is given by
\begin{equation}\label{courant-v}
[{v}_{f} , \bar{v}_{if} ] = - 2i  \mathrm{pr}_{T} (v_{f}) (\frac{1}{f^{2}}) {\mathcal J}_{2} X_{0} + 4G(X_{0}, X_{0}) d(\frac{1}{f^{2}}). 
 \end{equation}
ii) The following relation holds:
\begin{equation}\label{need-comp}
[\mathrm{pr}_{T} ({v}_{f}) ,\bar{v}_{if} ] - [\mathrm{pr}_{T} (\bar{v}_{if}) , {v}_{f}] = 
 - 2i  \mathrm{pr}_{T} (v_{f}) (\frac{1}{f^{2}}) {\mathcal J}_{2} X_{0} + 4G(X_{0}, X_{0}) d(\frac{1}{f^{2}}). 
\end{equation}
 \end{cor}

\begin{proof} From the definition of $v_{if}$ and the property  (\ref{courant1}) of the Courant bracket, we obtain
\begin{align}
\nonumber & [v_{f}, \bar{v}_{if} ]= - [X_{0}, v_{f}] - i [\mathcal J X_{0}, v_{f}]-  \frac{1}{f^{2}} [
\mathcal J_{3} X_{0}, v_{f}] + \pi (v_{f}) (\frac{1}{f^{2}} ) \mathcal J_{3} X_{0} \\
\label{v-f-if-1} & + 2 G(X_{0}, X_{0})d(\frac{1}{f^{2}})+ \frac{i}{f^{2}} [\mathcal J_{2} X_{0}, v_{f} ] 
-i \pi (v_{f})(\frac{1}{f^{2}}) \mathcal J_{2} X_{0}. 
\end{align}
From Lemma \ref{Courant-brackets} we obtain
\begin{align}
\nonumber  [X_{0}, v_{f} ] &= -d \left(\frac{G(X_{0},X_{0})}{f^{2}}\right) ;\\
\nonumber  [\mathcal J X_{0}, v_{f} ] &= -\frac{1}{f^{2}} ( 2 \mathcal J d G(X_{0},X_{0}) 
+ i dG(X_{0}, X_{0}) ) - \pi (\mathcal J X_{0}) (\frac{1}{f^{2}}) (\mathcal J_{3} X_{0} + i \mathcal J_{2} X_{0}) \\
\nonumber& + i G(X_{0}, X_{0}) d(\frac{1}{f^{2}});\\
\nonumber   [\mathcal J_{2}X_{0}, v_{f} ] &= i d G(X_{0}, X_{0});\\
\nonumber   [\mathcal J_{3} X_{0}, v_{f} ]&= - d G(X_{0}, X_{0}) + 2i \mathcal J dG(X_{0}, X_{0}) - 
\mathrm{pr}_{T}(\mathcal J_{3} X_{0}) (\frac{1}{f^{2}}) (\mathcal J_{3} X_{0} + i \mathcal J_{2} X_{0}).
\end{align}
Replacing these relations in (\ref{v-f-if-1}) we obtain (\ref{courant-v}).
In order to prove (\ref{need-comp}) we remark, from 
(\ref{courant-v}), that   $[\mathrm{pr}_{T} (v_{f}), \mathrm{pr}_{T} (\bar{v}_{if}) ] =0$.
Since   the Courant bracket of any $2$-forms is trivial, we  obtain 
that the left hand side of (\ref{need-comp}) is equal to $[ v_{f}, \bar{v}_{if} ]$.
From (\ref{courant-v}) again, we obtain (\ref{need-comp}).
\end{proof}

\begin{lem}\label{dec-courant}  The Courant bracket $L_{X_{0}}$ preserves $\Gamma (\mathcal S^{\perp})$, $\Gamma (L_{i}  \cap \mathcal S_{\mathbb{C}}^{\perp})$ and $\Gamma ( \bar{L}_{i} \cap \mathcal S_{\mathbb{C}}^{\perp})$
($i=1,2$).
\end{lem}

\begin{proof} We   prove the statements which involve 
$\mathcal S^{\perp}$ and $L_{i}$ (the statements which involve $\bar{L}_{i}$
can be obtained similarly).
Let $w\in \Gamma (\mathcal S^{\perp})$. Then 
$\langle X_{0}, w\rangle =0$ and, from (\ref{L-C}),   
$L_{X_{0}} (w) = {\mathcal L}_{X_{0}}(w)$. From relation 
 (\ref{uvw}) applied to $u:= X_{0}$, $v:= w$ and $w:=  x\in \Gamma (\mathcal S )$, 
together with $L_{X_{0}} (x) = {\mathcal L}_{X_{0}} (x) - d \langle X_{0}, x\rangle$, 
we obtain 
\begin{equation}\label{left-x}
\langle L_{X_{0}} (w), x\rangle + \langle w, {\mathcal L}_{X_{0}}(x)\rangle =0.
\end{equation}
For $x\in \{ X_{0}, {\mathcal J}X_{0}, \mathcal J_{2}X_{0}, \mathcal J_{3} X_{0} \}$, 
${\mathcal L} _{X_{0}}(x) =0$ and from 
(\ref{left-x})  we deduce that  
$\langle L_{X_{0}} (w), x\rangle =0$. We proved that 
$L_{X_{0}}$ preserves $\Gamma (\mathcal S^{\perp})$.  
We now prove that $L_{X_{0}}$ preserves $\Gamma ( L_{i} \cap \mathcal S_{\mathbb{C}}^{\perp})$.
Without loss of generality, we take $i=1$ (the argument for $i=2$ is similar).  
Let $w\in \Gamma (L_{1} \cap \mathcal S_{\mathbb{C}}^{\perp}).$ 
As $L_{X_{0}} (w) \in \Gamma  (\mathcal S^{\perp}_{\mathbb{C}})$ (from claim i)), we need to show that 
$L_{X_{0}}(w) \in \Gamma (L_{1}).$ Since $\mathcal L_{X_{0}} (\mathcal J_{1}) = 0$, $\mathcal L_{X_{0}}  (\mathcal J_{1} w) = \mathcal J_{1}
\mathcal L_{X_{0}}(w).$ But $w,\mathcal J_{1}w\in \Gamma (\mathcal S_{\mathbb{C}}^{\perp})$ which implies that 
$L_{X_{0}} (w) = {\mathcal L}_{X_{0}} (w)$, $L_{X_{0}} (\mathcal J_{1}w) = {\mathcal L}_{X_{0}} (\mathcal J_{1}w)$.
We obtain that $L_{X_{0}} (\mathcal J_{1}w)= \mathcal J_{1} L_{X_{0}} (w).$ 
Since $w\in \Gamma (L_{1})$, 
$\mathcal J_{1} w = i w$ and $i L_{X_{0}} (w) =\mathcal J_{1} L_{X}(w)$, i.e. $L_{X}(w) \in \Gamma (L_{1})$, 
as required. 
\end{proof}

\begin{lem}\label{v1-vi}  i) For any $X\in\Gamma  \mathrm{pr}_{T} (L_{1} \cap L_{2} \cap \mathcal S^{\perp}_{\mathbb{C}})$, 
and  
\begin{equation}\label{v1}
[\mathrm{pr}_{T} (v_{1}) , X]  +\frac {XG(X_{0}, X_{0}) }{G(X_{0}, X_{0})} \mathrm{pr}_{T} (v_{1})
\in\Gamma \mathrm{pr}_{T} (L_{1} \cap L_{2} \cap \mathcal S^{\perp}_{\mathbb{C}}).
\end{equation}
ii) For any $X\in \Gamma  \mathrm{pr}_{T} (L_{1} \cap \bar{L}_{2} \cap \mathcal S^{\perp}_{\mathbb{C}})$,
\begin{equation}\label{vi}
 [\mathrm{pr}_{T} (v_{i}) , X]  +\frac {XG(X_{0}, X_{0}) }{G(X_{0}, X_{0})} \mathrm{pr}_{T} (v_{i})
\in \Gamma \mathrm{pr}_{T} (L_{1} \cap \bar{L}_{2} \cap \mathcal S^{\perp}_{\mathbb{C}}).
\end{equation}
\end{lem}

\begin{proof} Let $\mathcal E (X)$ be the left hand side of (\ref{v1}). 
Since $L_{1}$ and $L_{2}$ are Courant integrable and $v_{1} \in \Gamma (L_{1} \cap L_{2})$, we deduce that 
$\mathcal E (X)$  is a section of  $ \mathrm{pr}_{T} (L_{1} \cap L_{2})$. Since  $\langle \mathcal E  (X), df^{H} \rangle =0$ (easy check), we obtain   (\ref{v1}).   Relation (\ref{vi}) can be proved similarly.
\end{proof}

\subsubsection{Involutivity of the bundles}\label{invol-sect}

Consider the setting from Theorem \ref{gen-KK} and let  $L_{2}^{\prime} = L(E_{2}^{\prime}, \epsilon_{2}^{\prime})$ be the 
$(1,0)$-bundle of 
the second generalized almost complex structure 
$\mathcal J_{2}^{\prime}$ of $(G^{\prime}, \mathcal J ).$

\begin{prop} \label{first-aim} Condition i) from Theorem \ref{prop-twist}, applied to the generalized almost Hermitian structure 
$(G^{\prime}, \mathcal J )$, holds, if and only if
conditions i) and ii) from Theorem \ref{gen-KK} hold.
\end{prop}

Part of the statement of Proposition \ref{first-aim}  is obvious: since 
$\mathcal J$ is integrable, the bundle $E_{1}$ is involutive. We deduce that $E_{1}$ is $(F,a)$-involutive 
if and only if $F(X,Y) X_{0}\in \Gamma (E_{1})$, for any $X,Y\in \Gamma (E_{1})$, i.e.\ relation
(\ref{F}) holds. The remaining part of the proof of Proposition \ref{first-aim} is divided into several lemmas,
as follows.

\begin{lem}\label{lema1}
i) The bundle $\mathrm{pr}_{T} (L_{1} \cap L_{2}^{\prime})$ is $(F,a)$-involutive if and only if
$F=0$ on $\Lambda^{2}  \mathrm{pr}_{T} (L_{1} \cap L_{2} \cap 
\mathcal S_{\mathbb{C}}^{\perp})$  
and relation (\ref{cond-long-m}) holds.\

ii) The bundle $\mathrm{pr}_{T} (L_{1} \cap \bar{L}_{2}^{\prime})$ is  $(F,a)$-involutive if and only if
$F=0$ on $\Lambda^{2}  \mathrm{pr}_{T} (L_{1} \cap \bar{L}_{2} \cap \mathcal S_{\mathbb{C}}^{\perp})$
and relation (\ref{cond-long-1-m}) holds.
\end{lem}

\begin{proof} 
Owing to the decomposition 
of $\mathrm{pr}_{T}(L_{1} \cap L_{2}^{\prime})$ given by the first relation 
 (\ref{dec-direct}), the $(F,a)$-involutivity 
of $\mathrm{pr}_{T}(L_{1} \cap L_{2}^{\prime})$ involves two cases: when both arguments 
$X$ and $Y$ of $[X, Y]^{(F,a)}$ 
are sections of
$\mathrm{pr}_{T} (L_{1} \cap L_{2} \cap \mathcal S_{\mathbb{C}}^{\perp})$ and, respectively, when one is a section of
this bundle and the other is $\mathrm{pr}_{T} (v_{f}).$ 
We begin with the first case.  From Lemma 2.1 of \cite{tolman},  $\{df^{H}\}^{\perp}$  is Courant integrable.
Since $L_{i}$ are Courant integrable, we obtain that also 
$L_{1} \cap L_{2} \cap \mathcal S_{\mathbb{C}}^{\perp} =  \{ df^{H}\}^{\perp} \cap L_{1}\cap L_{2}$ is Courant integrable,
and,  in particular, $\mathrm{pr}_{T} (L_{1} \cap L_{2} \cap \mathcal S^{\perp}_{\mathbb{C}})$ is involutive. 
If $\mathrm{pr}_{T} (L_{1} \cap L_{2}^{\prime})$ is $(F,a)$-involutive,  then 
\begin{equation}
F(X,Y) X_{0}\in \mathrm{pr}_{T} (L_{1} \cap L_{2}^{\prime}),\   \forall X, Y\in  \mathrm{pr}_{T} (L_{1} \cap L_{2} \cap \mathcal S_{\mathbb{C}}^{\perp}).
\end{equation}
We will show that this relation implies 
that $F=0$ on 
$\Lambda^{2} \mathrm{pr}_{T} (L_{1} \cap L_{2} \cap \mathcal S_{\mathbb{C}}^{\perp})$. 
Suppose that this is not true. 
Then $X_{0} \in \mathrm{pr}_{T} (L_{1} \cap L_{2}^{\prime})$
(at least at one point of $M$; for simplicity, this point will be omitted in our notation)  and
there is $\lambda \in \mathbb{C}$, 
$w\in L_{1} \cap L_{2} \cap \mathcal S_{\mathbb{C}}^{\perp}$
and $\xi \in (T^{*}M)^{\mathbb{C}}$,  such that
$X_{0} = \lambda v_{f} +  w +\xi$, or (from the definition of $v_{f}$), 
\begin{equation}\label{left}
(1-\lambda ) X_{0} +\lambda i \mathcal J X_{0}+ \frac{\lambda}{f^{2}} (\mathcal J_{3} X_{0} + i \mathcal J_{2} X_{0}) -  \xi\in L_{1} \cap L_{2} \cap \mathcal S_{\mathbb{C}}^{\perp}.
\end{equation}
In particular, the inner product $\langle \cdot , \cdot \rangle$ of the
left hand side of (\ref{left})  with 
$\mathcal J_{2} X_{0} = df^{H}$ vanishes and we obtain 
$\lambda G(X_{0}, X_{0})=0$, i.e. $\lambda =0$  (because $G$ is positive definite).  
We deduce that $X_{0} - \xi \in L_{1} \cap L_{2} \cap \mathcal S_{\mathbb{C}}^{\perp}.$
We will now show that this leads to a contradiction. Indeed, since
$X_{0} - \xi \in L_{1} \cap L_{2}$, we obtain that  $X_{0}- \xi \in C_{+}$ 
(the $1$-eigenbundle of $G^{\mathrm{end}}$). Recall now that $C_{+}$ is the graph of $b+g$
(where $b$ and $g$ are the $2$-forms, respectively the metric of the bi-Hermitian structure
associated to $(G, \mathcal J )$).
This implies $\xi = - (b+g) (X_{0})$. But then $\langle X_{0} - \xi , X_{0} \rangle = -\frac{1}{2} g(X_{0}, X_{0}) \neq 0$ (because $g$ is positive definite) and 
$X_{0} - \xi \notin \mathcal S_{\mathbb{C}}^{\perp}.$ 
We obtain a contradiction and we conclude that  
the first case of the $(F,a)$-involutivity of $\mathrm{pr}_{T}(L_{1}\cap L_{2}^{\prime})$ is equivalent to
$F=0$ on 
$\Lambda^{2} \mathrm{pr}_{T} (L_{1} \cap L_{2} \cap \mathcal S_{\mathbb{C}}^{\perp}).$

We now prove that  the second case  mentioned above, of the $(F,a)$-involutivity of 
$\mathrm{pr}_{T} (L_{1} \cap L^{\prime}_{2})$,   is equivalent to relation (\ref{cond-long-m}). 
For any $X\in \Gamma \mathrm{pr}_{T} (L_{1} \cap L_{2} \cap \mathcal S^{\perp}_{\mathbb{C}})$, let
\begin{equation}\label{F-f}
\mathcal F_{f}(X) := [\mathrm{pr}_{T} (v_{f}) , X]^{(F,a)}  +\frac {XG(X_{0}, X_{0}) }{G(X_{0}, X_{0})} \mathrm{pr}_{T} (v_{f})
\end{equation}
By a standard computation, which uses relation (\ref{uvw}), we obtain that
$\mathcal F_{f}(X) $ is related to the left hand side $\mathcal E (X)$ of (\ref{v1}) by
\begin{align}
\nonumber \mathcal F_{f}(X) & = \mathcal E (X) + (1-\frac{1}{f^{2}}) \{ 
[\mathrm{pr}_{T} (\mathcal J_{3} X_{0}), X] +  \alpha (X) \mathrm{pr}_{T} ( \mathcal J_{3} X_{0})\} \\
\label{rond-e-f}& +\frac{ F(\mathrm{pr}_{T}( v_{f}), X)}{a} X_{0}.
\end{align}
We obtain that   
$[\mathrm{pr}_{T} (v_{f}), X ]^{(F,a)}\in \Gamma ( L_{1} \cap {L}_{2}^{\prime})$ if and only if
$\mathcal F_{f} (X) \in \Gamma (L_{1} \cap L_{2}^{\prime})$,  if and only if
(from  Lemma \ref{v1-vi})
\begin{equation}\label{cln}
 [\mathrm{pr}_{T} (\mathcal J_{3} X_{0}), X] +  \alpha (X) \mathrm{pr}_{T} ( \mathcal J_{3} X_{0})
+\frac{ f^{2}F(\mathrm{pr}_{T}( v_{f}), X)}{a(f^{2} -1)} X_{0}\in \Gamma \mathrm{pr}_{T} (L_{1} \cap L_{2}^{\prime}). 
\end{equation}
In order to conclude the proof  of claim i), it remains to show that 
 (\ref{cln}) is equivalent to (\ref{cond-long-m}).  In order to prove this, 
let $u\in \Gamma (L_{1} \cap L_{2}
\cap \mathcal S^{\perp}_{\mathbb{C}})$ which projects to $X.$  From (\ref{cln}), there is a $1$-form $\xi$ and $\lambda \in \mathbb{C}$ such that 
\begin{align*}
& [\mathcal J_{3} X_{0}, u] +  \alpha (X) \mathcal J_{3} X_{0}
&  +\frac{ f^{2}F(\mathrm{pr}_{T}( v_{f}), X)}{a(f^{2} -1)} X_{0} +\lambda v_{f} + \xi \in \Gamma (L_{1} \cap L_{2} 
\cap \mathcal S^{\perp}_{\mathbb{C}}).
\end{align*}
In particular,  the $\langle\cdot , \cdot\rangle$-inner product of the above expression with $df^{H}$ vanishes. 
From  $\langle [ \mathcal J_{3} X_{0}, u], df^{H} \rangle =0$ (which follows from
relation (\ref{uvw}),  with $u$ replaced by $\mathcal J_{3} X_{0}$, $v$ replaced by $u$ and $w$ replaced by $df^{H}$,
by using $[df^{H}, \mathcal J_{3} X_{0} ] =0$, from Lemma \ref{Courant-brackets}), 
we obtain  that $\lambda =0$, which implies (\ref{cond-long-m}). 
Claim i) follows.  Claim ii) can be proved in a similar way. More precisely, one shows that 
the condition $[X, Y]^{(F,a)}\in \Gamma \mathrm{pr}_{T} (L_{1} \cap \bar{L}^{\prime}_{2})$, for any
$X, Y\in \Gamma \mathrm{pr}_{T}(L_{1} \cap \bar{L}_{2}\cap \mathcal S^{\perp}_{\mathbb{C}})$ is equivalent to
$F=0$ on $\Lambda^{2}  \mathrm{pr}_{T} (L_{1} \cap \bar{L}_{2} \cap \mathcal S_{\mathbb{C}}^{\perp})$.
Then, for any $X\in \Gamma \mathrm{pr}_{T}(L_{1} \cap\bar{ L}_{2}\cap \mathcal S^{\perp}_{\mathbb{C}})$, one defines
\begin{equation}\label{f-F-prime}
\mathcal F_{f}^{\prime}(X) := [\mathrm{pr}_{T} (v_{if}) , X]^{(F,a)}  +\frac {XG(X_{0}, X_{0}) }{G(X_{0}, X_{0})} \mathrm{pr}_{T} (v_{if})
\end{equation}
and shows that it is related to the left hand side $\mathcal E^{\prime}(X)$ of (\ref{vi})  
by
\begin{align}
\nonumber \mathcal F^{\prime}_{f}(X) & = \mathcal E^{\prime} (X) + (\frac{1}{f^{2}}-1) \{ 
[\mathrm{pr}_{T} (\mathcal J_{3} X_{0}), X] +  \alpha (X) \mathrm{pr}_{T} ( \mathcal J_{3} X_{0})\} \\
\label{rond-e-f}& +\frac{ F(\mathrm{pr}_{T}( v_{if}), X)}{a} X_{0}.
\end{align}
From Lemma \ref{v1-vi}, $\mathcal E^{\prime}(X)\in \Gamma \mathrm{pr}_{T} (L_{1} \cap \bar{L}_{2} \cap \mathcal S^{\perp}_{\mathbb{C}})$. Thus, 
$[ \mathrm{pr}_{T} (v_{if}), X ]^{(F,a)} \in \Gamma \mathrm{pr}_{T} (L_{1} \cap\bar{ L}_{2}^{\prime})$
if and only if 
\begin{equation}\label{cond-long-n}
 [\mathrm{pr}_{T} (\mathcal J_{3} X_{0}), X] +  \alpha (X) \mathrm{pr}_{T} ( \mathcal J_{3} X_{0})
- \frac{ f^{2}F(\mathrm{pr}_{T}( v_{if}), X)}{a( f^{2}-1 )} X_{0}\in \Gamma \mathrm{pr}_{T} (L_{1} \cap\bar{ L}_{2}^{\prime})
\end{equation}
and as before one can show that this is equivalent to (\ref{cond-long-1-m}). 
\end{proof}

\begin{rem}\label{conditii-echivalente} {\rm 
The proof of the above lemma shows that 
if $\mathrm{pr}_{T} (L_{1} \cap L_{2}^{\prime})$ and $\mathrm{pr}_{T} (L_{1} \cap \bar{L}_{2}^{\prime })$
are $(F, a)$-involutive,  then, for every
 $X\in \Gamma \mathrm{pr}_{T} (L_{1}\cap {L}_{2} \cap \mathcal S_{\mathbb{C}}^{\perp})$,
\begin{equation}\label{cond-long}
 [\mathrm{pr}_{T} (v_{f}), X]^{(F,a)} +\frac{ XG(X_{0}, X_{0})}{G(X_{0}, X_{0})} \mathrm{pr}_{T} (v_{f}) \in \Gamma 
\mathrm{pr}_{T}(L_{1} \cap L_{2} \cap \mathcal S_{\mathbb{C}}^{\perp})
\end{equation}
and for every  $X\in \Gamma \mathrm{pr}_{T} (L_{1}\cap \bar{L}_{2} \cap \mathcal S_{\mathbb{C}}^{\perp})$, 
\begin{equation}\label{cond-long-1}
 [\mathrm{pr}_{T} (v_{if}), X] ^{(F,a)}+\frac{ XG(X_{0}, X_{0})}{G(X_{0}, X_{0})} \mathrm{pr}_{T} (v_{if}) \in \Gamma 
\mathrm{pr}_{T}(L_{1} \cap \bar{L}_{2} \cap \mathcal S_{\mathbb{C}}^{\perp}).
\end{equation}}
\end{rem}

Next, we assume  that  the bundles $\mathrm{pr}_{T} (L_{1} \cap L_{2}^{\prime})$ and 
$\mathrm{pr}_{T} (L_{1} \cap \bar{L}^{\prime}_{2})$ are $(F,a)$-involutive.
Under these assumptions, we will prove that the conditions
(\ref{f-a-inv}) and (\ref{f-a-inv-1}) from Theorem \ref{prop-twist}, applied to 
$L_{1}$ and $L_{2}^{\prime}$,  are  satisfied. That is, we aim to show that
for any  $X\in \Gamma  \mathrm{pr}_{T} (L_{1} \cap L^{\prime}_{2})$ and $Y\in \Gamma 
\mathrm{pr}_{T} (\bar{L}_{1} \cap L_{2}^{\prime})$,
\begin{equation}\label{f-a-inv-new}
[X, Y]^{(F,a)} \in \Gamma ( E_{1} + E_{2}^{\prime})
\end{equation}
and for any  $X\in \Gamma  \mathrm{pr}_{T} (L_{1} \cap \bar{L}_{2}^{\prime})$ and $ Y\in \Gamma \mathrm{pr}_{T}(\bar{L}_{1}\cap 
\bar{L}_{2}^{\prime})$, 
\begin{equation}\label{f-a-inv-1-new}
[X, Y]^{(F,a)}\in \Gamma ( E_{1} + \bar{E}_{2}^{\prime}). 
\end{equation}

This will be done in Lemma \ref{inv-cons} and in Corollary 
\ref{inv-cons-cor} below, by  analysing  
how the map $(X, Y) \rightarrow [X,Y]^{(F,a)}$  behaves with respect to the decompositions of  
$\mathrm{pr}_{T}(L_{1}\cap L_{2}^{\prime})$, $\mathrm{pr}_{T}(\bar{L}_{1}\cap L_{2}^{\prime})$
and $E_{1} + E_{2}^{\prime}$ 
(for condition (\ref{f-a-inv-new})) 
and  how it behaves with respect to the decompositions of 
$\mathrm{pr}_{T}(L_{1}\cap \bar{L}_{2}^{\prime})$, $\mathrm{pr}_{T}(\bar{L}_{1}\cap \bar{L}_{2}^{\prime})$
and $E_{1} + \bar{E}_{2}^{\prime}$ (for condition (\ref{f-a-inv-1-new})).
Recall the decompositions (\ref{dec-direct}) of the bundles $\mathrm{pr}_{T} (L_{1} \cap L_{2}^{\prime})$ and $\mathrm{pr}_{T}(L_{1} \cap \bar{L}_{2}^{\prime})$.
By conjugation,  $\mathrm{pr}_{T} (\bar{L}_{1} \cap \bar{L}_{2}^{\prime})$ and 
$\mathrm{pr}_{T}(\bar{L}_{1} \cap {L}_{2}^{\prime})$ decompose  similarly:
\begin{align}
\nonumber &\mathrm{pr}_{T} (\bar{L}_{1} \cap \bar{L}_{2}^{\prime}) 
=\mathrm{span}_{\mathbb{C}} \{ \mathrm{pr}_{T}(\bar{v}_{f} ) \} +  \mathrm{pr}_{T} (\bar{L}_{1} \cap \bar{L}_{2} \cap 
\mathcal S_{\mathbb{C}}^{\perp})\\
\label{dec-direct-conj} &\mathrm{pr}_{T} (\bar{L}_{1} \cap {L}_{2}^{\prime}) 
=\mathrm{span}_{\mathbb{C}} \{ \mathrm{pr}_{T} ( \bar{v}_{if})  \} 
+ \mathrm{pr}_{T} (\bar{L}_{1} \cap {L}_{2} \cap \mathcal S_{\mathbb{C}}^{\perp})
\end{align}
(direct sum decompositions). 
From Remark \ref{prel-deco} i) and $v_{f}, v_{if} \in L_{1}$, we obtain
\begin{align}
\nonumber  E_{1}+ E_{2}^{\prime} &= E_{1}+ \mathrm{pr}_{T} ( L_{2} \cap 
\mathcal S_{\mathbb{C}}^{\perp}) + \mathrm{span}_{\mathbb{C}} \{ \mathrm{pr}_{T} (\bar{v}_{if}) \} \\
\label{desc-new}E_{1}+ \bar{E}_{2}^{\prime} &=  E_{1}+ \mathrm{pr}_{T} ( \bar{L}_{2} \cap \mathcal S_{\mathbb{C}}^{\perp}) +
 \mathrm{span}_{\mathbb{C}}
 \{ \mathrm{pr}_{T} (\bar{v}_{f}) \} . 
\end{align}
Also, 
\begin{equation}\label{X-vf}
X_{0}=\frac{1}{2}  \mathrm{pr}_{T} (v_{f} + \bar{v}_{if}) \in E_{2}^{\prime} \cap \bar{E}_{2}^{\prime}
\end{equation}
and, from  Corollary \ref{vfif} i),
\begin{equation}\label{v-f-if}
[ \mathrm{pr}_{T} (v_{f}), \mathrm{pr}_{T} (\bar{v}_{if}) ] ^{(F,a)}= \frac{F(\mathrm{pr}_{T} (v_{f}), \mathrm{pr}_{T}( \bar{v}_{if}))}{a} X_{0}.
\end{equation}

\begin{lem}\label{inv-cons} Suppose that $\mathrm{pr}_{T} (L_{1} \cap L_{2}^{\prime})$ and $\mathrm{pr}_{T} (L_{1} 
\cap \bar{L}_{2}^{\prime})$ are $(F,a)$-involutive. The following statements hold:\

i)   For any $X\in \Gamma \mathrm{pr}_{T}( {L}_{1} \cap L_{2} \cap \mathcal S_{\mathbb{C}}^{\perp})$, 
\begin{align}
\nonumber & [\mathrm{pr}_{T} (\bar{v}_{if}), X]^{(F,a)} -\frac{ XG(X_{0}, X_{0})}{G(X_{0}, X_{0})}
\mathrm{pr}_{T} ( {v}_{f}) -\frac { 2F(X_{0}, X)}{a} X_{0}\\
\label{aut2} & \in \Gamma \mathrm{pr}_{T}(L_{1} \cap L_{2} \cap
\mathcal S_{\mathbb{C}}^{\perp}).
\end{align}

ii) For any $X\in \Gamma \mathrm{pr}_{T}( \bar{L}_{1} \cap L_{2} \cap \mathcal S_{\mathbb{C}}^{\perp})$, 
\begin{align}
\nonumber& [\mathrm{pr}_{T} (v_{f}), X]^{(F,a)} -
\frac{ XG(X_{0}, X_{0})}{G(X_{0}, X_{0})} \mathrm{pr}_{T} ( \bar{v}_{if}) -\frac { 2F(X_{0}, X)}{a} X_{0}\\
\label{aut1}&   \in 
\Gamma \mathrm{pr}_{T} (\bar{L}_{1} \cap L_{2} \cap \mathcal S_{\mathbb{C}}^{\perp}).
\end{align}
\end{lem}

\begin{proof} 
Let $X\in \Gamma \mathrm{pr}_{T} (L_{1} \cap L_{2}
\cap \mathcal S^{\perp}_{\mathbb{C}})$ and $\mathcal R_{f}(X)$ be the expression from the left hand side  of (\ref{aut2}). 
Recall  the expression $\mathcal F_{f}(X)$ defined in (\ref{F-f}). 
A straightforward computation shows that 
$$
\mathcal R_{f}(X) = - \mathcal F_{f}(X) + 2 [X_{0}, X].
$$
From Lemmas \ref{dec-courant}
and (the first relation of) Remark \ref{conditii-echivalente},  $[X_{0}, X]$ and   $\mathcal F_{f}(X)$ are sections of   
$ \mathrm{pr}_{T} (L_{1} \cap L_{2} \cap \mathcal S^{\perp}_{\mathbb{C}})$. Relation 
 (\ref{aut2}) follows. Relation (\ref{aut1}) is obtained similarly, by 
 comparering the
 the left hand side of (\ref{aut1})  with the
conjugate of $\mathcal F^{\prime}_{f}(X)$, defined in (\ref{f-F-prime}),
and using (the second relation of)  Remark \ref{conditii-echivalente}. 
\end{proof}

\begin{cor}\label{inv-cons-cor}
Suppose that $\mathrm{pr}_{T} (L_{1} \cap L_{2}^{\prime})$ and $\mathrm{pr}_{T} (L_{1} 
\cap \bar{L}_{2}^{\prime})$ are $(F,a)$-involutive. 
Then  (\ref{f-a-inv-new}) and  (\ref{f-a-inv-1-new})  are satisfied. 
\end{cor}

\begin{proof}  
From relations (\ref{cond-long}), (\ref{cond-long-1}), 
(\ref{v-f-if}),  (\ref{aut2}),  (\ref{aut1}),  their conjugates,  
and  the decompositions (\ref{dec-direct}), (\ref{dec-direct-conj}), 
(\ref{desc-new}), 
it remains to prove two more statements:\

a)  $[X,Y]^{(F,a)}$ is a section of $E_{1} + E_{2}^{\prime}$, for any 
$X\in \Gamma (L_{1}\cap L_{2} \cap \mathcal S_{\mathbb{C}}^{\perp})$ and   $Y\in \Gamma (\bar{L}_{1} \cap L_{2} \cap \mathcal S_{\mathbb{C}}^{\perp})$;\

b)  $[X,Y]^{(F,a)}$ is a section of $E_{1} + \bar{E}_{2}^{\prime}$, for any 
$X\in \Gamma (L_{1}\cap \bar{L}_{2} \cap \mathcal S_{\mathbb{C}}^{\perp})$ and 
$Y\in \Gamma (\bar{L}_{1} \cap \bar{L}_{2} \cap \mathcal S_{\mathbb{C}}^{\perp})$.\

To prove these claims, we will show  that 
\begin{equation}\label{aaa}
[X,  Y]^{(F,a)} = \mathrm{pr}_{T}( [w_{1}, w_{2}]^{\perp} )
+\left( \frac{ G( [w_{1}, w_{2}],X_{0})}{G(X_{0}, X_{0})}+ \frac{ F(X,Y)}{a}\right)X_{0},
\end{equation}
for any  $w_{1} \in \Gamma (L_{1} \cap L_{2} \cap \mathcal S_{\mathbb{C}}^{\perp})$ and 
$w_{2} \in \Gamma ( \bar{L}_{1} \cap L_{2} \cap \mathcal S_{\mathbb{C}}^{\perp})$. Let 
$X:= \mathrm{pr}_{T}(w_{1})$ and  $Y:= \mathrm{pr}_{T} (w_{2}).$  
In order to prove relation (\ref{aaa}),   we remark that 
 $[w_{1}, w_{2} ] \in \Gamma ( L_{2} \cap \{ df^{H}\}^{\perp})$ (as $L_{2}$ and $\{ df^{H}\}^{\perp}$ are Courant integrable). 
This implies, using $df^{H} =\mathcal J_{2} X_{0}$,  
\begin{align}
\nonumber& G( [w_{1}, w_{2}], \mathcal J X_{0})  = G( [w_{1}, w_{2}], \mathcal J_{3} X_{0}) =0;\\
\label{scal1} & G( [w_{1}, w_{2}] , \mathcal J_{2} X_{0}) = - G( \mathcal J_{2} [w_{1}, w_{2} ], X_{0}) = - i G([w_{1}, w_{2}], X_{0}).
\end{align}
From (\ref{scal1}) we deduce that
\begin{equation}\label{j}
[w_{1}, w_{2}] = [w_{1}, w_{2}]^{\perp} +\frac{ G( [w_{1}, w_{2}],X_{0})}{G(X_{0}, X_{0})}
(X_{0}- i \mathcal J_{2} X_{0}),
\end{equation}
which  implies (\ref{aaa}).  
Note that $\mathrm{pr}_{T}(  [w_{1}, w_{2} ]^{\perp})\in  
\Gamma \mathrm{pr}_{T} ( L_{2} \cap \mathcal S^{\perp}_{\mathbb{C}})$.
From 
(\ref{X-vf})  and (\ref{aaa}) we obtain
\begin{equation}\label{j-1}
[X,Y]^{(F,a)}\in \Gamma (\mathrm{pr}_{T} (L_{2}\cap \mathcal S_{\mathbb{C}}^{\perp}) + E_{2} ^{\prime}\cap \bar{E}_{2}^{\prime}).
\end{equation}
As $\mathrm{pr}_{T}( L_{2} \cap \mathcal S^{\perp}_{\mathbb{C}}) \subset E_{2}^{\prime}$,  
the statements a) and b) follow from 
(\ref{j-1}) (and its conjugate).
\end{proof}

The proof of Proposition \ref{first-aim} is now completed.

\subsubsection{The differential equations on forms}\label{diff-section}

We  consider the setting from Theorem \ref{gen-KK} and we assume that
conditions i) and ii) from this theorem are satisfied. 
From the previous section, this means that condition
i)  from Proposition \ref{prop-twist}, applied to the generalized almost complex
structure $(G^{\prime}, \mathcal J)$, is  satisfied.
Let $\epsilon^{\prime}\in \Gamma (\mathrm{pr}_{T} ( L_{1} \cap L_{2}^{\prime})\wedge (E_{1}+ E_{2}^{\prime}))$ and 
$\tilde{\epsilon}^{\prime}\in \Gamma (\mathrm{pr}_{T} ( L_{1} \cap \bar{L}_{2}^{\prime})\wedge (E_{1}+ \bar{E}_{2}^{\prime}))$ 
be the two $2$-forms associated to  $(G^{\prime}, \mathcal J )$,
as  at the beginning of Section \ref{ah-section}.
Relations (\ref{f-a-herm-1}) from Theorem \ref{prop-twist}, 
 applied to the generalized almost Hermitian structure $(G^{\prime}, \mathcal J )$ and function
$h$, are
\begin{align}
\nonumber& d^{(F,a)} \epsilon_{1} =\frac{2}{h} \epsilon_{1} \wedge dh\vert_{E_{1}}\\
\nonumber& d^{(F,a)} \epsilon^{\prime} = \frac{2}{h} \epsilon^{\prime}\wedge dh\ \mathrm{on} \ \Lambda^{2} 
\mathrm{pr}_{T}(L_{1} \cap L_{2}^{\prime}) \wedge 
\mathrm{pr}_{T}(\bar{L}_{1} \cap L_{2}^{\prime})\\
\label{prime-appl}& d^{(F,a)}\tilde{ \epsilon}^{\prime} = \frac{2}{h} \epsilon^{\prime}\wedge dh\ \mathrm{on} \ 
\Lambda^{2} \mathrm{pr}_{T}(L_{1} \cap \bar{L}_{2}^{\prime}) \wedge 
\mathrm{pr}_{T}(\bar{L}_{1} \cap \bar{L}_{2}^{\prime}).  
\end{align}
In order to conclude the proof of Theorem \ref{gen-KK} it remains to show
(using the material from the previous section) that  
relations (\ref{prime-appl}) 
are equivalent to the remaining conditions iii), iv) and v) from Theorem \ref{gen-KK}. 
This will be done in this section.
Since $\mathcal J $ is integrable, 
$E_{1}$ is involutive and $d\epsilon_{1}=0$.
From relation  (\ref{alpha-p}) 
applied to $\alpha :=\epsilon_{1}$, we obtain that 
the first relation
(\ref{prime-appl}) is equivalent to the first relation 
(\ref{cond-e}).
 From now on we assume that these two equivalent relations 
hold.

\begin{prop}\label{final}
In this setting,  the second and third relation (\ref{prime-appl}) are equivalent to 
relation (\ref{log}), the second relation (\ref{cond-e}) and  relation (\ref{cond-de}).
\end{prop}

We divide the proof of the above proposition into several steps. We begin by 
computing 
$\epsilon^{\prime}$ and $\tilde{\epsilon}^{\prime}$. Recall the decompositions 
(\ref{dec-direct}) and  
(\ref{desc-new}).

\begin{lem}\label{e-eprime-def} i) The form 
${\epsilon}^{\prime}\in\Gamma ( (L_{1} \cap {L}_{2})^{*}\wedge (E_{1} + {E}_{2}^{\prime})^{*})$
is given by: 
\begin{align*}
&\epsilon^{\prime} (\mathrm{pr}_{T} (v_{f}), X) = \mathrm{pr}_{T^{*}} (v_{f}) (X),\ \forall X\in E_{1} + \mathrm{pr}_{T} (L_{2}
\cap \mathcal S_{\mathbb{C}}^{\perp}),\\
&\epsilon^{\prime} ( \mathrm{pr}_{T} (v_{f}), \mathrm{pr}_{T} (\bar{v}_{if})) =\frac{4}{f^{2}} G(X_{0}, X_{0}),\\
&\epsilon^{\prime} (X, Y) = \epsilon_{2}(X, Y),\ \forall X\in \mathrm{pr}_{T} (L_{1} \cap  L_{2} \cap \mathcal S^{\perp}_{\mathbb{C}}),\
Y\in \mathrm{pr}_{T} (L_{2} \cap \mathcal S^{\perp}_{\mathbb{C}}),\\
&\epsilon^{\prime} (X, Y) = \epsilon_{1} (X, Y),\ \forall X\in \mathrm{pr}_{T} (L_{1} \cap L_{2} \cap \mathcal S^{\perp}_{\mathbb{C}}),\
Y\in E_{1},\\
&\epsilon^{\prime} (X, \mathrm{pr}_{T}( \bar{v}_{if})) = - \mathrm{pr}_{T^{*}} (\bar{v}_{if})(X),\ \forall X\in
\mathrm{pr}_{T} (L_{1} \cap L_{2} \cap \mathcal S^{\perp}_{\mathbb{C}}).
\end{align*}

ii) The form $\tilde{\epsilon}^{\prime}\in\Gamma (  (L_{1} \cap \bar{L}^{\prime}_{2})^{*}\wedge (E_{1} + \bar{E}_{2}^{\prime})^{*})$
is given by 
\begin{align*}
&\tilde{\epsilon}^{\prime} (\mathrm{pr}_{T} (v_{if}), X) = \mathrm{pr}_{T^{*}} (v_{if}) (X),\ \forall X\in  E_{1}+\mathrm{pr}_{T} (\bar{L}_{2}
\cap \mathcal S_{\mathbb{C}}^{\perp}),\\
&\tilde{\epsilon}^{\prime} ( \mathrm{pr}_{T} (v_{if}), \mathrm{pr}_{T} (\bar{v}_{f})) =- \frac{4}{f^{2}} G(X_{0}, X_{0}),\\
&\tilde{\epsilon}^{\prime} (X, Y) = \bar{\epsilon}_{2}(X, Y),\ \forall X\in \mathrm{pr}_{T} (L_{1} \cap  \bar{L}_{2} \cap \mathcal S^{\perp}_{\mathbb{C}}),\
Y\in \mathrm{pr}_{T} (\bar{L}_{2} \cap W^{\perp}_{\mathbb{C}}),\\
&\tilde{\epsilon}^{\prime} (X, Y) = \epsilon_{1} (X, Y),\ \forall X\in \mathrm{pr}_{T} (L_{1} \cap\bar{ L}_{2} \cap \mathcal S^{\perp}_{\mathbb{C}}),\
Y\in E_{1},\\
&\tilde{\epsilon}^{\prime} (X, \mathrm{pr}_{T}( \bar{v}_{f})) =- \mathrm{pr}_{T^{*}} (\bar{v}_{f})(X),\ \forall X\in
\mathrm{pr}_{T} (L_{1} \cap \bar{L}_{2} \cap \mathcal S^{\perp}_{\mathbb{C}}).
\end{align*}
\end{lem}

\begin{proof} 
The proof is straightforward from definitions. 
Let us compute  for example  $ \epsilon^{\prime} ( \mathrm{pr}_{T} (v_{f}), \mathrm{pr}_{T} (\bar{v}_{if})) $. 
Using  $v_{f}\in  L_{1} \cap L_{2}^{\prime}$,  $\bar{v}_{if} \in \bar{L}_{1} \cap L_{2}^{\prime}$
and the definition of $\epsilon^{\prime}$, we obtain
\begin{equation}\label{primitiva}
\epsilon^{\prime}( \mathrm{pr}_{T} (v_{f}), \mathrm{pr}_{T} (\bar{v}_{if})) = \epsilon_{2}^{\prime} 
( \mathrm{pr}_{T} (v_{f}), \mathrm{pr}_{T} (\bar{v}_{if})  = 2\langle v_{f}, \mathrm{pr}_{T} ( \bar{v}_{if})\rangle .
\end{equation}
On the other hand, for any $u, v\in \mathbb{T} M$ orthogonal with respect to $\langle \cdot , \cdot \rangle$, 
$\langle \mathrm{pr}_{T} (u), v \rangle + \langle u, \mathrm{pr}_{T} (v) \rangle =0.$ 
Replacing in (\ref{primitiva}) the definitions of $v_{f}$ and ${v}_{if}$ and using this remark
for $u=v=\mathcal J X_{0}$, $u=v=\mathcal J_{3} X_{0}$ and for $u= \mathcal J X_{0}, v= \mathcal J_{3} X_{0}$, 
 we obtain that
$\langle v_{f}, \mathrm{pr}_{T} (\bar{v}_{if} ) \rangle = \frac{2}{f^{2}} G(X_{0}, X_{0} ).$  The required expression   
for   $ \epsilon^{\prime} ( \mathrm{pr}_{T} (v_{f}), \mathrm{pr}_{T} (\bar{v}_{if})) $ follows from (\ref{primitiva}). 
\end{proof}

For simplicity of notation, let
$\mathcal F$ be the $2$-form on   $\mathrm{pr}_{T} (L_{2} \cap \mathcal S^{\perp}_{\mathbb{C}})$  defined by the left hand side of (\ref{cond-de}).

\begin{lem}\label{comput1} The $3$-form $d^{(F,a)} \epsilon^{\prime} $ on $\Lambda^{2} \mathrm{pr}_{T}(L_{1} \cap L_{2}^{\prime})
\wedge 
\mathrm{pr}_{T}(\bar{L}_{1} \cap L_{2}^{\prime})$ is given by:\\

i) for any $X,  Y\in  \mathrm{pr}_{T} (L_{1} \cap {L}_{2} 
\cap \mathcal S_{\mathbb{C}}^{\perp})$ and  $Z\in \mathrm{pr}_{T}(\bar{L}_{1}\cap {L}_{2}\cap \mathcal S_{\mathbb{C}}^{\perp})$,
\begin{equation}\label{d}
(d^{(F,a)} \epsilon^{\prime}) (X,Y,Z) =0;
\end{equation}
ii) for any $X\in \mathrm{pr}_{T} (L_{1} \cap {L}_{2} \cap \mathcal S_{\mathbb{C}}^{\perp})$, 
\begin{equation}\label{a}
(d^{(F,a)} {\epsilon}^{\prime}) (X,\mathrm{pr}_{T}({v}_{f}),\mathrm{pr}_{T} (\bar{v}_{if})) = - \frac{ 4  X(af^{2})}{af^{4}} G(X_{0}, 
X_{0}); 
\end{equation}
iii)  for any $X\in  \mathrm{pr}_{T} (L_{1} \cap {L}_{2} \cap \mathcal S_{\mathbb{C}}^{\perp})$ and $Y\in  \mathrm{pr}_{T}(\bar{L}_{1} 
\cap {L}_{2} \cap \mathcal S_{\mathbb{C}}^{\perp})$, 
\begin{equation}\label{b}
(d^{(F,a)}{\epsilon}^{\prime}) (X, Y, \mathrm{pr}_{T}(v_{f}) ) = G(X_{0}, X_{0})\mathcal F (X,Y);
\end{equation}
iv) for any  $X,Y\in \mathrm{pr}_{T} (L_{1} \cap {L}_{2} \cap \mathcal S_{\mathbb{C}}^{\perp})$, 
\begin{equation}\label{c}
(d^{(F,a)}{\epsilon}^{\prime}) (X, Y, \mathrm{pr}_{T}(\bar{v}_{if}) ) =-  G(X_{0}, X_{0}) \mathcal F (X,Y).
\end{equation}
\end{lem}

\begin{proof} For claim i), let 
 $X, Y\in\Gamma  \mathrm{pr}_{T} (L_{1} \cap {L}_{2} 
\cap \mathcal S_{\mathbb{C}}^{\perp})$ and  $Z\in \Gamma \mathrm{pr}_{T}(\bar{L}_{1}\cap {L}_{2}\cap \mathcal S_{\mathbb{C}}^{\perp})$.
Choose $w_{1}, w_{2} \in \Gamma (L_{1} \cap L_{2} \cap \mathcal S_{\mathbb{C}}^{\perp})$ which project to $X$ and $Y$ respectively,
and $w_{3} \in \Gamma (\bar{L}_{1}\cap L_{2} \cap \mathcal S_{\mathbb{C}}^{\perp})$ which projects to $Z$. 
Since $\mathrm{pr}_{T}(L_{1} \cap L_{2} \cap \mathcal S_{\mathbb{C}}^{\perp})$ is involutive (see the proof of Lemma \ref{lema1})
and $F(X, Y)=0$ (see relation (\ref{F})) we obtain that 
$[X,Y]^{(F,a)} = [X,Y]\in \Gamma 
\mathrm{pr}_{T}(L_{1} \cap L_{2} \cap \mathcal S_{\mathbb{C}}^{\perp})$.
From Lemma \ref{e-eprime-def}, 
\begin{equation}\label{lie-1}
\epsilon^{\prime}([X, Y]^{(F,a)}, Z) = \epsilon_{2} ( [X,Y], Z).
\end{equation}
We now compute $\epsilon^{\prime} ( [X,Z]^{(F,a)}, Y).$ From relation (\ref{j}), 
\begin{equation}\label{ec-w}
 [X,Z]^{(F,a)}= \mathrm{pr}_{T} ( [w_{1},w_{3}]^{\perp}) + \left( \frac{G(X_{0}, [w_{1}, w_{3}])}{ G(X_{0}, X_{0})}
 + \frac{ F(X,Z)}{a}\right) X_{0}.
 \end{equation}
 Since   $\mathrm{pr}_{T} ( [w_{1},w_{3}]^{\perp})\in \Gamma (L_{2} \cap \mathcal S_{\mathbb{C}}^{\perp})$ 
and $Y\in \Gamma \mathrm{pr}_{T} (L_{1} \cap L_{2} \cap \mathcal S^{\perp}_{\mathbb{C}})$, we obtain   
 \begin{align}
 \nonumber& \epsilon^{\prime} ([X,Z]^{(F,a)}, Y) = \epsilon_{2} ( \mathrm{pr}_{T}( [w_{1}, w_{3}]^{\perp}), Y) \\
\label{t1}& +\left( \frac{G([w_{1}, w_{3} ], X_{0})}{G(X_{0}, X_{0})} +\frac{F(X,Z)}{a} \right)\epsilon^{\prime}(X_{0}, Y). 
 \end{align}

From Lemma \ref{e-eprime-def} and $Y\in \Gamma \mathrm{pr}_{T}(L_{1} \cap L_{2} \cap  \mathcal S_{\mathbb{C}}^{\perp})$, 
\begin{equation}\label{eps-pr-X}
2\epsilon^{\prime} ( X_{0}, Y) =\epsilon^{\prime}( \mathrm{pr}_{T}( v_{f}), Y)
+  \epsilon^{\prime}( \mathrm{pr}_{T}(\bar{v}_{if}), Y)   =
2 \langle \mathrm{pr}_{T^{*}} ( v_{f} + \bar{v}_{if}), Y\rangle = 0,
\end{equation}
where in the last equality we used 
$\mathrm{pr}_{T^{*}} (v_{f} + \bar{v}_{if}) = - \frac{2i}{f^{2}} \mathcal J_{2} X_{0}$
and  $Y\in\Gamma  \mathrm{pr}_{T} (\mathcal S^{\perp}_{\mathbb{C}}).$ 
We proved that
\begin{equation}\label{lie-2}
\epsilon^{\prime}([X,Z]^{(F,a)}, Y)= \epsilon_{2} ( [X,Z], Y).
\end{equation}
Similarly, 
\begin{equation}\label{lie-3}
\epsilon^{\prime}([Y,Z]^{(F,a)}, X)= \epsilon_{2} ( [Y,Z], X).
\end{equation}
From (\ref{lie-1}), (\ref{lie-2}),  (\ref{lie-3}) we obtain 
$$
(d^{(F,a)} \epsilon^{\prime})(X,Y, Z) = d\epsilon_{2} (X, Y, Z) =0.
$$
Claim i) is proved.

We now prove claim  ii).  Let $X\in \Gamma \mathrm{pr}_{T} (L_{1} \cap L_{2} \cap \mathcal S_{\mathbb{C}}^{\perp})$. 
From Corollary \ref{vfif} i), 
$[\mathrm{pr}_{T} (v_{f}), \mathrm{pr}_{T} (\bar{v}_{if} ) ] =0$. From relation
(\ref{eps-pr-X}), $\epsilon^{\prime } (X_{0}, X) =0$.
We obtain
\begin{equation}\label{rd1}
\epsilon^{\prime}([\mathrm{pr}_{T} (v_{f}) , \mathrm{pr}_{T}(\bar{v}_{if})]^{(F,a)} , X) = 0.
\end{equation}
Next, we compute $\epsilon^{\prime}( [\mathrm{pr}_{T} (v_{f}), X]^{(F,a)}, \mathrm{pr}_{T} (\bar{v}_{if}))$.
In order to do this, 
we add and subtract to $[\mathrm{pr}_{T} (v_{f}), X]^{(F,a)}$ the term 
$ \frac{XG(X_{0}, X_{0})}{G(X_{0}, X_{0})} \mathrm{pr}_{T} (v_{f})$:
\begin{align*}
&\epsilon^{\prime}( [\mathrm{pr}_{T}( v_{f}), X]^{(F,a)}, \mathrm{pr}_{T} (\bar{v}_{if}))
=  -\frac {XG(X_{0}, X_{0})}{G(X_{0}, X_{0})} \epsilon^{\prime} 
(\mathrm{pr}_{T} (v_{f}),  \mathrm{pr}_{T} (\bar{v}_{if})) \\
&+ \epsilon^{\prime} ( [\mathrm{pr}_{T}( v_{f}), X]^{(F,a)}  + \frac{XG(X_{0}, X_{0})}{G(X_{0}, X_{0})} \mathrm{pr}_{T} (v_{f}), \mathrm{pr}_{T} (\bar{v}_{if})).
\end{align*}
From relation (\ref{primitiva}), 
$\epsilon^{\prime}( \mathrm{pr}_{T} ( v_{f}), \mathrm{pr}_{T}(\bar{v}_{if}))  
= 2 \langle v_{f}, \mathrm{pr}_{T} ( \bar{v}_{if})\rangle.$ 
From relation (\ref{cond-long}) and Lemma \ref{e-eprime-def},
we obtain 
\begin{align*}
 &\epsilon^{\prime} ( [\mathrm{pr}_{T}( v_{f}), X]^{(F,a)} + \frac{XG(X_{0}, X_{0})}{G(X_{0}, X_{0})} \mathrm{pr}_{T} (v_{f}), \mathrm{pr}_{T} (\bar{v}_{if})) \\
& = -2 \langle [\mathrm{pr}_{T} (v_{f}), X]^{(F,a)} + \frac{XG(X_{0}, X_{0})}{G(X_{0}, X_{0})} \mathrm{pr}_{T} (v_{f}),
\mathrm{pr}_{T^{*}} (\bar{v}_{if})\rangle . 
\end{align*}
Combining the above relations 
and using 
\begin{equation}
\langle v_{f}, \mathrm{pr}_{T} (\bar{v}_{if})\rangle + \langle \mathrm{pr}_{T} (v_{f}), \mathrm{pr}_{T^{*}} (\bar{v}_{if})\rangle =  \langle v_{f}, \bar{v}_{if}\rangle =0,
\end{equation}
($v_{f}, \bar{v}_{if} \in L_{2}^{\prime}$ which is isotropic),
we obtain 
\begin{align}
\nonumber& \epsilon^{\prime}( [\mathrm{pr}_{T}( v_{f}), X]^{(F,a)} , \mathrm{pr}_{T} (\bar{v}_{if}))\\
\label{rd2}&=- 2\langle [\mathrm{pr}_{T} (v_{f}), X],  \mathrm{pr}_{T^{*}} (\bar{v}_{if})\rangle
 +\frac{ 2 F(\mathrm{pr}_{T}(v_{f}), X)}{af^{2}} G(X_{0} , X_{0}).
\end{align}
A similar computation which uses relation (\ref{aut2})  shows that 
\begin{align}
\nonumber& \epsilon^{\prime}( [\mathrm{pr}_{T}( \bar{v}_{if}), X]^{(F,a)} , \mathrm{pr}_{T} ({v}_{f}))\\
\label{rd3}&=- 2\langle [\mathrm{pr}_{T} (\bar{v}_{if}), X],  \mathrm{pr}_{T^{*}} ({v}_{f})\rangle
 -\frac{ 2 F(\mathrm{pr}_{T}(\bar{v}_{if}), X)}{af^{2}} G(X_{0} , X_{0}).
\end{align}
Using Lemma \ref{e-eprime-def} for 
$\epsilon^{\prime} ( \mathrm{pr}_{T} ( {v}_{f}), X)$, 
$\epsilon^{\prime} ( \mathrm{pr}_{T} ( \bar{v}_{if}), X) $  and 
$\epsilon^{\prime} (\mathrm{pr}_{T}(v_{f}), \mathrm{pr}_{T}( \bar{v}_{if})) $, together with relations 
(\ref{rd1}), (\ref{rd2}) and  (\ref{rd3}), we obtain 
\begin{align}
\nonumber&(d^{(F,a)} \epsilon^{\prime})(X , \mathrm{pr}_{T} ( v_{f}), \mathrm{pr}_{T} ( \bar{v}_{if})) = - 2\langle  [\mathrm{pr}_{T} (\bar{v}_{if}), v_{f} ] -  [\mathrm{pr}_{T} (v_{f}), \bar{v}_{if} ], X\rangle\\
\nonumber &+\frac{4}{af^{2}} F(X_{0}, X) G(X_{0}, X_{0}).
\end{align}
From  Corollary \ref{vfif}  ii), $i_{X_{0}} F = -da$  
and $X\in \Gamma \mathrm{pr}_{T} (\mathcal S^{\perp}_{\mathbb{C}})$ we obtain claim ii).

Claims iii) and iv) can be proved in a similar way.  We
only sketch the proof of claim iii). Let $X\in \Gamma \mathrm{pr}_{T}(L_{1}\cap L_{2} \cap \mathcal S_{\mathbb{C}}^{\perp})$ and 
$Y\in \Gamma  \mathrm{pr}_{T}(\bar{L}_{1} \cap L_{2} \cap \mathcal S_{\mathbb{C}}^{\perp})$.
For computing $\epsilon^{\prime} ( [\mathrm{pr}_{T} ( v_{f}), Y]^{(F,a)}, X)$ we use 
relation 
(\ref{aut1}) (by adding and substracting the 
term $\frac{YG(X_{0}, X_{0})}{G(X_{0}, X_{0})} \mathrm{pr}_{T} (\bar{v}_{if})+ \frac{2F(X_{0}, Y)}{a} X_{0}$), 
for computing $\epsilon^{\prime} ( [\mathrm{pr}_{T} ( v_{f}), X]^{(F,a)}, Y)$ we use, as above,  relation (\ref{cond-long}) and for 
computing $\epsilon^{\prime} ( [X,Y]^{(F,a)}, \mathrm{pr}_{T} (v_{f}))$ we use
(\ref{ec-w}). 
Using Lemma \ref{e-eprime-def} for $\epsilon^{\prime} ( X, Y)$, $\epsilon^{\prime} ( X, \mathrm{pr}_{T} (v_{f}))$
and $\epsilon^{\prime} ( Y, \mathrm{pr}_{T}(v_{f}))$  
we  finally  obtain 
\begin{align}
\nonumber& (d^{(F,a)} \epsilon^{\prime}) (X, Y, \mathrm{pr}_{T}(v_{f})) = - d(\mathrm{pr}_{T^{*}} (v_{f}) )(X,Y) + ({\mathcal D}_{\mathrm{pr}_{T}( v_{f})} \epsilon_{2} ) (X,Y)\\
\label{caseiii}&- (\frac{dG(X_{0}, X_{0})}{G(X_{0}, X_{0})} \wedge \mathrm{pr}_{T^{*}} (\bar{v}_{if}) )(X,Y) + \frac{ 2F(X,Y)}{af^{2}} G(X_{0}, X_{0}),
\end{align}
where  ${\mathcal D}_{\mathrm{pr}_{T} (v_{f})} (\epsilon_{2} ) (X,Y)$ is defined by
\begin{align}
\nonumber &({\mathcal D}_{\mathrm{pr}_{T} (v_{f})} \epsilon_{2} ) (X,Y) := \mathrm{pr}_{T}(v_{f}) ( \epsilon_{2} (X,Y))\\
\nonumber&   - \epsilon_{2} ( [\mathrm{pr}_{T} (v_{f}), X] + \frac{XG(X_{0}, X_{0})}{G(X_{0}, X_{0})} \mathrm{pr}_{T} ({v}_{f}), Y)\\
 \label{d-i-f} & - \epsilon_{2} ( X, [\mathrm{pr}_{T} (v_{f}), Y] -  \frac{YG(X_{0}, X_{0})}{G(X_{0}, X_{0})} \mathrm{pr}_{T} (\bar{v}_{if})).
\end{align}
(Remark that   ${\mathcal D}_{\mathrm{pr}_{T} (v_{f})} \epsilon_{2} $ is well-defined, owing to  (\ref{cond-long}),  (\ref{aut1})
and $X_{0} \in \Gamma (E_{2})$, which ensure that the maps
$$
 \Gamma \mathrm{pr}_{T} (L_{1} \cap L_{2} \cap \mathcal S^{\perp}_{\mathbb{C}}) \ni X \rightarrow 
[\mathrm{pr}_{T} (v_{f}), X] +\frac{ XG(X_{0}, X_{0})}{G(X_{0}, X_{0})} \mathrm{pr}_{T} (v_{f})
$$
and 
$$
 \Gamma \mathrm{pr}_{T} (\bar{L}_{1} \cap L_{2} \cap \mathcal S^{\perp}_{\mathbb{C}}) \ni Y \rightarrow 
[\mathrm{pr}_{T} (v_{f}), Y]  - \frac{ YG(X_{0}, X_{0})}{G(X_{0}, X_{0})} \mathrm{pr}_{T} (\bar{v}_{if})
$$
take values in $\Gamma (E_{2})$).
In particular, taking in (\ref{caseiii})  $f:=1$ and $F:=0$, we obtain 
\begin{align}
\nonumber(d \epsilon_{2}) (X, Y, \mathrm{pr}_{T}(v_{1}) )& = - d(\mathrm{pr}_{T^{*}} (v_{1}) )(X,Y) + ({\mathcal D}_{\mathrm{pr}_{T} (v_{1})} \epsilon_{2} ) (X,Y)\\
\label{caseiii-1}&- (\frac{dG(X_{0}, X_{0})}{G(X_{0}, X_{0})} \wedge \mathrm{pr}_{T^{*}} (\bar{v}_{i}) )(X,Y) ,
\end{align}
where  $({\mathcal D}_{\mathrm{pr}_{T} (v_{1})} \epsilon_{2} ) (X,Y)$ is defined by (\ref{d-i-f}), by 
replacing $v_{f}$ with $v_{1}$ and $\bar{v}_{if}$ with $\bar{v}_{i}.$  
But since $d\epsilon_{2} =0$, the right hand side of (\ref{caseiii-1}) vanishes. Substracting the right hand side of (\ref{caseiii-1}) from (\ref{caseiii}) we obtain, after a straightforward computation,  claim iii).  
\end{proof}

By similar computations we obtain $d^{(F,a)}\tilde{\epsilon}^{\prime} $. 
Below 
$\bar{\mathcal F}$ is a $2$-form  on $\mathrm{pr}_{T} (\bar{L}_{2} \cap \mathcal S_{\mathbb{C}}^{\perp})$, 
defined by $\bar{\mathcal F}(X, Y) := \overline{\mathcal F (\bar{X}, \bar{Y})}$, for any
$X, Y\in\mathrm{pr}_{T} (L_{2} \cap \mathcal S^{\perp}_{\mathbb{C}})$.

\begin{lem}\label{comput2} The $3$-form $d^{(F,a)}\tilde{\epsilon}^{\prime} $ on 
$\Lambda^{2} \mathrm{pr}_{T}(L_{1} \cap \bar{L}_{2}^{\prime}) \wedge 
\mathrm{pr}_{T}(\bar{L}_{1} \cap \bar{L}_{2}^{\prime})$ is given by:

i) for any $X, Y\in\mathrm{pr}_{T} (L_{1} \cap \bar{L}_{2} 
\cap \mathcal S_{\mathbb{C}}^{\perp})$ and  $Z\in  \mathrm{pr}_{T}(\bar{L}_{1}\cap \bar{L}_{2}\cap \mathcal S_{\mathbb{C}}^{\perp})$,
\begin{equation}\label{d}
(d^{(F,a)} \tilde{\epsilon}^{\prime}) (X,Y,Z) =0;
\end{equation}
ii) for any $X\in \mathrm{pr}_{T} (L_{1} \cap \bar{L}_{2} \cap \mathcal S_{\mathbb{C}}^{\perp})$, 
\begin{equation}\label{a-prime}
(d^{(F,a)} \tilde{\epsilon}^{\prime}) (X,\mathrm{pr}_{T}({v}_{if}),\mathrm{pr}_{T} (\bar{v}_{f})) =  \frac{ 4 X(af^{2})}{af^{4}} G(X_{0}, 
X_{0});
\end{equation}
iii)  for any $X\in\mathrm{pr}_{T} (L_{1} \cap \bar{L}_{2} \cap \mathcal S_{\mathbb{C}}^{\perp})$ and $Y\in \mathrm{pr}_{T}(\bar{L}_{1} 
\cap \bar{L}_{2} \cap \mathcal S_{\mathbb{C}}^{\perp})$, 
\begin{equation}
(d^{(F,a)}\tilde{\epsilon}^{\prime}) (X, Y, \mathrm{pr}_{T}(v_{if}) ) = - G(X_{0}, X_{0})\bar{\mathcal F} (X,Y);
\end{equation}
iv) for any  $X,Y\in \mathrm{pr}_{T} (L_{1} \cap \bar{L}_{2} \cap \mathcal S_{\mathbb{C}}^{\perp})$, 
\begin{equation}
(d^{(F,a)}\tilde{\epsilon}^{\prime}) (X, Y, \mathrm{pr}_{T}(\bar{v}_{f}) ) =  G(X_{0}, X_{0})\bar{\mathcal F} (X,Y);
\end{equation}
\end{lem}

The next two lemmas collect the exterior products  $dh\wedge {\epsilon}^{\prime}$ and $dh\wedge\tilde{\epsilon}^{\prime}$.
The proofs are  straighforward computations.

\begin{lem}\label{exterior1} The $3$-form $dh\wedge{\epsilon}^{\prime}$ on
$\Lambda^{2} \mathrm{pr}_{T} (L_{1} \cap L_{2}^{\prime})\wedge \mathrm{pr}_{T} ( \bar{L}_{1} \wedge L_{2}^{\prime})$
 is given by:

i) for any $X\in \mathrm{pr}_{T}( L_{1} \cap L_{2} \cap \mathcal S_{\mathbb{C}}^{\perp})$,  
\begin{equation}
(dh\wedge{\epsilon}^{\prime})(X, \mathrm{pr}_{T}(v_{f}), \mathrm{pr}_{T} (\bar{v}_{if})) = \frac{4G(X_{0}, X_{0})}{f^{2}} X(h); 
\end{equation}
ii) for any $X\in \mathrm{pr}_{T}(L_{1} \cap L_{2} \cap \mathcal S_{\mathbb{C}}^{\perp})$ and $Y\in \mathrm{pr}_{T}(\bar{L}_{1}\cap L_{2} 
\cap \mathcal S^{\perp}_{\mathbb{C}})$, 
\begin{equation}
 (dh\wedge{\epsilon}^{\prime})(X, Y, \mathrm{pr}_{T} ({v}_{f}))=\mathrm{pr}_{T}(v_{f}) (h) \epsilon_{2} (X,Y)
 +(\mathrm{pr}_{T^{*}} (v_{f})\wedge dh )(X,Y);
\end{equation}
iii)  for any $X,Y\in \mathrm{pr}_{T}(L_{1} \cap L_{2} \cap \mathcal S_{\mathbb{C}}^{\perp})$, 
\begin{equation}
 (dh\wedge{\epsilon}^{\prime})(X, Y, \mathrm{pr}_{T} (\bar{v}_{if}))=\mathrm{pr}_{T}(\bar{v}_{if}) (h) 
 \epsilon_{2} (X,Y) + (\mathrm{pr}_{T^{*}} (\bar{v}_{if})\wedge dh)(X,Y);
\end{equation}
iv) for any $X, Y\in \mathrm{pr}_{T}( L_{1} \cap L_{2} \cap \mathcal S_{\mathbb{C}}^{\perp})$ and 
$Z\in \mathrm{pr}_{T}(\bar{ L}_{1} \cap L_{2} \cap \mathcal S_{\mathbb{C}}^{\perp})$,
$$
(dh\wedge{\epsilon}^{\prime})(X, Y, Z) =(dh\wedge{\epsilon}_{2})(X, Y, Z). 
$$
\end{lem}

Similarly:

\begin{lem}\label{exterior2}
The $3$-form $dh\wedge\tilde{\epsilon}^{\prime}$ 
on $\Lambda^{2} \mathrm{pr}_{T} (L_{1} \cap\bar{L}_{2}^{\prime}) \wedge \mathrm{pr}_{T} (\bar{L}_{1}\cap
\bar{L}_{2}^{\prime})$ 
is given by:

i) for any $X\in \mathrm{pr}_{T}( L_{1} \cap \bar{L}_{2} \cap \mathcal S_{\mathbb{C}}^{\perp})$. 
\begin{equation}
(dh\wedge\tilde{\epsilon}^{\prime})(X, \mathrm{pr}_{T}(v_{if}), \mathrm{pr}_{T} (\bar{v}_{f})) 
= -\frac{4G(X_{0}, X_{0})}{f^{2}} X(h); 
\end{equation}
ii) for any $X\in \mathrm{pr}_{T}(L_{1} \cap \bar{L}_{2} \cap \mathcal S_{\mathbb{C}}^{\perp})$ and $Y\in \mathrm{pr}_{T}(\bar{L}_{1}\cap 
\bar{L}_{2} 
\cap \mathcal S_{\mathbb{C}}^{\perp})$, 
\begin{equation}
 (dh\wedge\tilde{\epsilon}^{\prime})(X, Y, \mathrm{pr}_{T} ({v}_{if}))=\mathrm{pr}_{T}(v_{if}) (h) 
 \bar{\epsilon}_{2} (X,Y)
 +(\mathrm{pr}_{T^{*}} (v_{if})\wedge dh )(X,Y);
\end{equation}
iii)  for any $X,Y\in \mathrm{pr}_{T}(L_{1} \cap \bar{L}_{2} \cap \mathcal S_{\mathbb{C}}^{\perp})$, 
\begin{equation}
 (dh\wedge\tilde{\epsilon}^{\prime})(X, Y, \mathrm{pr}_{T} (\bar{v}_{f}))=\mathrm{pr}_{T}(\bar{v}_{f}) (h) 
 \bar{\epsilon}_{2} (X,Y) +  (\mathrm{pr}_{T^{*}} (\bar{v}_{f})\wedge dh)(X,Y);
\end{equation}
iv) for any $X, Y\in \mathrm{pr}_{T}( L_{1} \cap\bar{L}_{2} \cap \mathcal S_{\mathbb{C}}^{\perp})$ and 
$Z\in \mathrm{pr}_{T}(\bar{ L}_{1} \cap \bar{L}_{2} \cap \mathcal S_{\mathbb{C}}^{\perp})$,
$$
(dh\wedge\tilde{\epsilon}^{\prime})(X, Y, Z) =(dh\wedge\bar{\epsilon}_{2})(X, Y, Z). 
$$
\end{lem}

Using the above lemmas, we now prove Proposition \ref{final} as follows.
Suppose that the second and third relation (\ref{prime-appl}) hold. We  apply these relations  to various types of arguments, according to the 
decomposition of 
$\mathrm{pr}_{T} (L_{1} \cap L_{2}^{\prime})$, $\mathrm{pr}_{T} (\bar{L}_{1} \cap L_{2}^{\prime} )$, 
$\mathrm{pr}_{T} (L_{1} \cap\bar{ L}_{2}^{\prime} )$
and $\mathrm{pr}_{T} (\bar{L}_{1} \cap\bar{ L}_{2}^{\prime} )$, 
and we  use Lemmas \ref{comput1},  \ref{comput2},
\ref{exterior1} and \ref{exterior2},
to obtain relation (\ref{log}), 
the second relation 
(\ref{cond-e}) and relation (\ref{cond-de}).

Let us explain how we obtain  relation (\ref{log}). 
From the second relation
(\ref{prime-appl}), 
Lemma \ref{comput1} ii) and Lemma \ref{exterior1}  i), 
we obtain 
\begin{equation}\label{log-1}
X(af^{2} h^{2}) =0,
\end{equation}
for any  $X\in \mathrm{pr}_{T}(L_{1} \cap L_{2} \cap \mathcal S^{\perp}_{\mathbb{C}})$.
From the third relation (\ref{cond-e}), Lemma  \ref{comput2} ii) and Lemma \ref{exterior2}  i), we obtain the same relation
(\ref{log-1}),  with  $X\in 
\mathrm{pr}_{T} (L_{1} \cap\bar{ L}_{2} \cap \mathcal S^{\perp}_{\mathbb{C}}).$ 
Therefore, (\ref{log-1}) holds on
$$
\mathrm{pr}_{T} (L_{1}\cap \mathcal S^{\perp}_{\mathbb{C}})=
\mathrm{pr}_{T} ( L_{1} \cap L_{2} \cap \mathcal S^{\perp}_{\mathbb{C}}) + 
\mathrm{pr}_{T} ( L_{1} \cap \bar{L}_{2}\cap \mathcal S^{\perp}_{\mathbb{C}}) .
$$
By conjugation, using that $L_{1} +\bar{L}_{1} = (TM)^{\mathbb{C}}$, we obtain (\ref{log-1}) for any
$X \in \mathrm{pr}(\mathcal S^{\perp})$. Relation (\ref{log}) follows.

We now prove the second relation  (\ref{cond-e}).
From the second and third relations  (\ref{prime-appl}),  Lemma \ref{comput1} i) together with
Lemma \ref{exterior1} iv), and, respectively, 
Lemma  \ref{comput2} i) together with Lemma \ref{exterior2} iv), 
we obtain
\begin{equation}\label{suma1}
dh\wedge \epsilon_{2} =0\ \mathrm{on}\ 
\Lambda^{2} \mathrm{pr}_{T} (L_{1} \cap L_{2} \cap \mathcal S^{\perp}_{\mathbb{C}})\wedge \mathrm{pr}_{T} ( \bar{L}_{1} \cap L_{2}
\cap \mathcal S^{\perp}_{\mathbb{C}})
\end{equation}
and, respectively,  
\begin{equation}\label{suma2}
 dh \wedge \bar{\epsilon}_{2} =0\ \mathrm{on}\  \Lambda^{2} \mathrm{pr}_{T} (L_{1} \cap \bar{L}_{2} \cap \mathcal S^{\perp}_{\mathbb{C}})\wedge 
\mathrm{pr}_{T} ( \bar{L}_{1} \cap \bar{L}_{2}
\cap \mathcal S^{\perp}_{\mathbb{C}}).
\end{equation}
Relation (\ref{suma1}) and the conjugate of (\ref{suma2})   give $dh\wedge \epsilon_{2} =0$ on  
$$
\mathrm{pr}_{T} ( L_{1} \cap L_{2} \cap \mathcal S^{\perp}_{\mathbb{C}})\wedge  \mathrm{pr}_{T} ( \bar{L}_{1} \cap L_{2} \cap \mathcal S^{\perp}_{\mathbb{C}})\wedge 
\mathrm{pr}_{T} (L_{2} \cap \mathcal S^{\perp}_{\mathbb{C}}).
$$
The second relation (\ref{cond-e}) follows. Finally,
one can check that the  second and third relation (\ref{prime-appl})
and the remaining statements from Lemmas  \ref{comput1},  \ref{comput2},  \ref{exterior1} and \ref{exterior2}  
give (\ref{cond-de}).
The proof of Proposition \ref{final} and  Theorem \ref{gen-KK} is completed.

\section{Applications, comments and  examples}\label{examples}

As a first application of Theorem \ref{gen-KK} , we  recover 
the KK correspondence 
 from the 
K\"{a}hler setting.

\begin{prop}\label{concise}    Let $(M, g, J)$ be a K\"{a}hler manifold with  a Hamiltonian Killing vector field
$X_{0}.$ Assume, for simplicity, that the Hamiltonian function $f^{H}$ of $X_{0}$ is positive. 
Let $X_{0}^{\flat}= g(X_{0}, \cdot )$ be the $1$-form dual to $X_{0}.$ 
Then the data formed by 
$F:= \omega -\frac{1}{2} dX_{0}^{\flat}$ together with
\begin{equation}
a := -  \left( f^{H} + \frac{g(X_{0}, X_{0})}{2}\right),\ 
f^{2}:= \frac{2f^{H}}{ 2 f^{H} + g(X_{0}, X_{0})},\ h^{2} := f^{H}
\end{equation}
satisfies all the conditions from Theorem  \ref{gen-KK} and gives rise to a K\"{a}hler 
manifold
$[\tau_{h} (g^{\prime}, J)]_{W}$.  (The assignment $(M, g, J, X_{0}) \rightarrow [\tau_{h} (g^{\prime}, J)]_{W}$ is known 
as
the {\cmssl KK correspondence} {\rm  \cite{ACM,ACDM}.)}
\end{prop}
 
\begin{proof} 
We consider $(M, g, J)$ as a generalized K\"{a}hler manifold and we apply Theorem \ref{gen-KK}. 
As $\mathcal J_{3} X_{0} = - X_{0}^{\flat}$  projects trivially on $TM$
and $F =0$ on $\Lambda^{2} T^{1,0}M$
(because $X_{0}$ is Hamiltonian Killing),   the first two conditions of
Theorem \ref{gen-KK} are  satisfied.  Since $af^{2} h^{2} = - (f^{H})^{2}$, its differential annihilates
$\mathrm{pr}_{T}(\mathcal S^{\perp})=\mathrm{Ker}\{ { df}^{H}, df^{H}\circ J \}$ and  condition iii) from Theorem \ref{gen-KK} is satisfied. 
The first relation (\ref{cond-e}) is trivially satisfied ($\epsilon_{1}=0$), and the second relation (\ref{cond-e}) is satisfied as well,
because $h$ depends only on $f^{H}$ and therefore $dh$ annihilates $\mathrm{pr}_{T} (\mathcal S^{\perp}).$  
Relation (\ref{cond-de})   follows from the definition of $F$. 
\end{proof}

In order to  apply Theorem \ref{gen-KK} in the generalized (non-K\"{a}hler) setting,   it is natural to look,
in view of condition i) from this theorem,  for examples with $ F =0$ on $\Lambda^{2}E_{1}.$  
But when $F=0$ on $\Lambda^{2}E_{1}$, the first relation (\ref{cond-e}) from Theorem \ref{gen-KK} becomes 
$\epsilon_{1} \wedge dh =0$ on $\Lambda^{3} E_{1}.$  
When the function $h$ is not constant, this imposes strong restrictions on the generalized K\"{a}hler manifold, as the next lemma shows.

\begin{lem}\label{supl} Let  $(M, G, \mathcal J)$ be a generalized K\"{a}hler manifold and $L(E_{1}, \epsilon_{1})$ the 
$(1,0)$-bundle of $\mathcal J .$ If there is a form $\beta \in \Omega^{1}(M)$ (non-trivial at any point), such that
$\epsilon_{1} \wedge \beta =0$ on $\Lambda^{3}E_{1}$, then either $(G, \mathcal J)$ is the $B$-field transformation of a K\"{a}hler
structure or $\mathrm{rank} (\Delta_{1} ) = 2$, where   $\Delta_1$ is the set of real points of  
$(\Delta_{1})^{\mathbb{C}}:= E_{1} \cap \bar{E}_{1}$.
\end{lem}

\begin{proof} 
The proof is  simple and we skip the details.  Assume that $\mathrm{rank} (\Delta_{1})\neq 2$. 
The condition $\epsilon_{1} \wedge \beta =0$ implies that 
$\Delta_{1} =0$, i.e.\ $\mathcal J$  is the $B$-field transformation of a complex structure.  From Remark
\ref{rem-elem} ii),   $(G, \mathcal J )$  is the $B$-field transformation of a 
K\"{a}hler structure.  
\end{proof}

In  Proposition \ref{gen-applied-1} below we apply  Theorem \ref{gen-KK}, with $h:=1$,  to produce new
generalized K\"{a}hler manifolds from given ones (not necessarily K\"{a}hler).
We need the following  lemma. 

\begin{lem}\label{J-invariance} Let $(M, G, \mathcal J )$ be a generalized K\"{a}hler manifold,  with a Hamiltonian Killing vector field $X_{0}$, such that $\mathcal J_{3} X_{0}\in \Omega^{1}(M).$  
Then 
$d(\mathcal J_{3} X_{0}) =0 $ on $\Lambda^{2} E_{1}$, where $L(E_{1}, \epsilon_{1})$ is the $(1,0)$-bundle of
$\mathcal J .$ 
\end{lem}

\begin{proof}
From the Cartan formula for $d(\mathcal J_{3} X_{0})$, we  obtain (using  $\mathcal J_{3} X_{0} \in \Omega^{1}(M)$),
for any $u, v\in \Gamma (\mathbb{T}M)$, 
\begin{align}
\nonumber d(\mathcal J_{3} X_{0})(\mathrm{pr}_{T}(u), \mathrm{pr}_{T}(v))
&=  -2 \left( \mathrm{pr}_{T} ( u) \langle df^{H}, \mathcal J v\rangle - \mathrm{pr}_{T}(  v)
\langle df^{H}, \mathcal J  u\rangle\right) \\
\label{cartan}&+
2 G(X_{0}, [ u ,  v ] ).
\end{align}
From (\ref{uvw}),  with $v$ replaced by $df^{H}$ and $w$ replaced by $\mathcal J v$, we obtain
$$
\mathrm{pr}_{T} (  u) \langle df^{H}, \mathcal J v\rangle = \langle [ u , df^{H} ], \mathcal J v\rangle +
\langle df^{H}, [ u, \mathcal J v] \rangle + \langle d \langle u, df^{H}\rangle  , \mathcal J v\rangle .
$$
On the other hand, from the definition of the Courant bracket, 
$[u, df^{H} ] = \frac{1}{2} d( \mathrm{pr}_{T} (u) ( f^{H}))$ and 
we deduce that 
$$
\mathrm{pr}_{T} (u) \langle df^{H}, \mathcal J v\rangle = \frac{1}{2} \mathrm{pr}_{T}(u) \mathrm{pr}_{T} (\mathcal J v)(f^{H}).
$$
Combining this relation with (\ref{cartan}) we obtain
\begin{align}
\nonumber d(\mathcal J_{3} X_{0}) (\mathrm{pr}_{T} ( u), \mathrm{pr}_{T} ( v)) & = - \mathrm{pr}_{T} (u)\mathrm{pr}_{T} (\mathcal J v)
(f^{H}) + \mathrm{pr}_{T} (v)\mathrm{pr}_{T} (\mathcal J u)
(f^{H})\\
\label{dj3}&  + 2\langle df^{H}, \mathcal J [u,v]\rangle .
\end{align}
Replacing in (\ref{dj3}) $u$, $v$ by $\mathcal J u$,  $\mathcal J v$ and using 
$N_{\mathcal J}(u, v) =0$ we obtain 
$$
d(\mathcal J_{3} X_{0}) (\mathrm{pr}_{T} (\mathcal J u), \mathrm{pr}_{T} (\mathcal J v))
= d(\mathcal J_{3} X_{0}) (\mathrm{pr}_{T} ( u), \mathrm{pr}_{T} ( v)),
$$ 
which implies our claim.
\end{proof}

\begin{prop}\label{gen-applied-1}  
Let $(M, G, \mathcal J) $ be a generalized K\"{a}hler manifold with a  Hamiltonian Killing vector field $X_{0}$, 
such that    
$\mathcal J_{3} X_{0} \in \Omega^{1}(M)$.    Consider the elementary deformation $(G^{\prime},\mathcal J )$ of
$(G, \mathcal J)$ by the vector field $X_{0}$ and function 
$$
f := \frac{1}{( 1 + K (f^{H})  G(X_{0}, X_{0} ))^{1/2}},
$$
where  $K= K(f^{H})$ is any smooth, non-negative function, depending only on the Hamiltonian function $f^{H}$ 
of $X_{0}.$  Consider  the twist data 
with curvature  $F$ and function $a$ given by 
\begin{equation}\label{definition-f}
F:= - \frac{1}{2} d ( K(f^{H}) \mathcal J_{3} X_{0} ),\  a := 1+ K (f^{H})  G(X_{0}, X_{0}).
\end{equation}
Then the twist $[( \mathcal J , G^{\prime})]_{W}$ is generalized K\"{a}hler.
\end{prop}

\begin{proof}   
Since $\mathcal J_{3} X_{0} \in \Omega^{1}(M)$, for any $u, v\in \mathbb{T}M$, 
$$
(df^{H} \wedge \mathcal J_{3}X_{0})(\mathrm{pr}_{T} (u), \mathrm{pr}_{T}(v))
= 4 \left( \langle df^{H}, u\rangle \langle \mathcal J_{3} X_{0} , v\rangle -
\langle df^{H}, v\rangle \langle \mathcal J_{3} X_{0},u\rangle\right)  
$$
and  
$$
(df^{H}\wedge \mathcal J_{3} X_{0} )(\mathrm{pr}_{T} (\mathcal J u ), \mathrm{pr}_{T} (\mathcal J v))
= (df^{H} \wedge \mathcal J_{3})(\mathrm{pr}_{T} (u), \mathrm{pr}_{T}(v)),
$$
i.e. $df^{H}\wedge \mathcal J_{3} X_{0} =0$ on $\Lambda^{2}E_{1}.$ From Lemma 
\ref{J-invariance}, also   $d (\mathcal J_{3} X_{0})=0$  on  $\Lambda^{2} E_{1}.$ 
We deduce that $F=0$ on $\Lambda^{2} E_{1}.$ The $1$-form $\alpha$
defined by (\ref{alpha})  is given by 
$\alpha = -\frac{ K^{\prime} (f^{H})}{K(f^{H})}df^{H}$, it  annihilates $\mathrm{pr}_{T} (\mathcal S^{\perp})$,  
and $af^{2} =1.$ The conditions from Theorem \ref{gen-KK} can be checked easily.
(We remark that relation (\ref{cond-de}) holds on the entire $\Lambda^{2} (TM)^{\mathbb{C}}$, not only on
$\Lambda^{2}\mathrm{pr}_{T} (L_{2} \cap \mathcal S^{\perp}_{\mathbb{C}})$). 
\end{proof}

Examples of generalized K\"{a}hler manifolds $(M , G, \mathcal J )$ with a Hamiltonian Killing vector field 
$X_{0}$ satisfying $\mathcal J_{3} X_{0} \in \Omega^{1}(M)$ can be found in the toric setting, which will be treated in the next section (see Example \ref{pr-t-ex}).

\subsection{Examples: toric generalized K\"{a}hler manifolds}\label{toric-section}

Following \cite{boulanger}, we briefly recall the local description of diagonal toric generalized K\"{a}hler
structures, in terms of a strictly convex function $\tau$ (the symplectic potential) and a  skew-symmetric matrix
$C$. This is a  generalization of the local description of toric K\"{a}hler structures, which can be recovered
as a particular case, when $C=0.$

 Let $(M, \omega ,\mathbb{T}^{n} )$ be a $2n$-dimensional toric symplectic manifold, that is,  
 a symplectic manifold $M$ of dimension $2n$ with a Hamiltonian action of the torus $\mathbb{T}^{n}$.
 Recall from Section
\ref{complex-geometry} that  generalized complex structures 
$\mathcal J$, with the property that  $(\mathcal  J , \mathcal J_{\omega })$ is a generalized K\"{a}hler structure, are in one to one correspondence with complex structures $J_{+}$ which tame $\omega$,  and whose  $\omega$-adjoint $J^{*\omega}_+$  is integrable. We assume that $\mathcal J$ (equivalently, $J_{+}$)  is $\mathbb{T}^{n}$-invariant
and that  $J_{+}\mathcal K = J_{+}^{*\omega}\mathcal K$, where $\mathcal K$ is the distribution tangent to 
the orbits of the $\mathbb{T}^{n}$-action. 
Such toric generalized K\"{a}hler structures are called  
{\cmssl diagonal}.   We restrict to the open subset (also denoted by $M$) of $M$ where $\mathcal K$ has constant rank $n$. 
Choose a basis of the Lie algebra of $\mathbb{T}^{n}$
and let $K_{i}\in {\mathfrak X}(M)$ be the associated fundamental vector fields generated by the $\mathbb{T}^{n}$-action.   
As $ \{ K_{i},  J_{+} K_{j} \}$ commute, we may choose local coordinates   
$(t^{i}, u^{i})$, such that   
$$
 K_{i} = -\frac{\partial}{\partial t^{i}},\  J_+K_{i} = \frac{\partial}{\partial u^{i}},\quad 1\leq i\leq n. 
$$
For any $1\leq i \leq n$, let $\mu^{i}$ be the moment map of $K_{i}$:  $ i_{K_{i}} \omega  = d\mu^{i}$. 
As $\mathrm{span}_{\mathbb{R}} \{ d\mu^{i}\} = \mathrm{span}_{\mathbb{R}} \{ du^{i} \}$, 
$du^{i} =  \sum_{j=1}^{n}\Psi_{ij} d\mu^{j}$, for some functions   
$\Psi_{ij}$ which depend only on the moment coordinates $\{ \mu^{i}\}$. 
In the coordinate system $(t^{i}, \mu^{i})$,
$\omega$,  
$J_{+}$ and $J_{-} = - J_{+}^{*\omega}$ are given by:
\begin{equation}\label{o}
\omega = \sum_{i=1}^{n} d\mu^{i}\wedge dt^{i},
\end{equation}
and   
\begin{align}
\nonumber& J_{+} = \sum_{i,j=1}^{n} \Psi_{ij} \frac{\partial}{\partial t^{i}} \otimes d\mu^{j} - \sum_{i,j=1}^{n} 
 \Psi^{ij} \frac{\partial}{\partial \mu^{i}} \otimes dt^{j}\\
\label{j-pm}&  J_{-} =  \sum_{i,j=1}^{n} \Psi_{ji} \frac{\partial}{\partial t^{i}} \otimes d\mu^{j} - \sum_{i,j=1}^{n} 
 \Psi^{ji} \frac{\partial}{\partial \mu^{i}} \otimes dt^{j},
\end{align} 
where $(\Psi^{ij}):= (\Psi^{-1})_{ij}$. Since  $J_{\pm}$  are integrable, $\Psi = ( \Psi_{ij})$ is of the form
\begin{equation}\label{psi-c}
\Psi_{ij} = \frac{\partial^{2} \tau}{\partial \mu^{i}\partial \mu^{j}} + C_{ij},
\end{equation}
where $\tau = \tau (\mu^{i})$ is strictly convex (i.e.\ has positive definite Hessian) and $C:= (C_{ij})$ is a (constant)  skew-symmetric matrix
(see Theorem 6 of \cite{boulanger}). 
Conversely, any strictly convex function $\tau = \tau( \mu^{i})$ together with a skew-symmetric matrix $C$, define, via
(\ref{o}),   (\ref{j-pm}) and  (\ref{psi-c}), a diagonal generalized K\"{a}hler structure.  
The function $\tau$ is referred to  as  the {\cmssl symplectic potential}.

\begin{notation}{\rm i) All toric generalized K\"{a}hler manifolds we are concerned with are diagonal. 
To simplify  terminology, from now on
we omit the word `diagonal' when  referring to them.\

ii) The superscripts ``$s$" and ``$a$"
used below mean the symmetric, respectively, the skew-symmetric parts of a matrix. }
 \end{notation}

\begin{lem}\label{cond-toric}
The Hamiltonian Killing vector field $X_{0}:= K_{1}= -\frac{\partial}{\partial t^{1}}$
on the the toric generalized K\"{a}hler manifold  associated to
$(M, J_{+}, \omega )$ satisfies  
\begin{align}
\nonumber& \mathrm{pr}_{T} (\mathcal J_{3} X_{0} ) = \sum_{r=1}^{n} \left( (\Psi^{-1})^{a} [ ( \Psi^{-1})^{s}]^{-1} \right)_{1r}
\frac{\partial}{\partial t^{r}}\\
\nonumber&  \mathrm{pr}_{T^{*}} (\mathcal J_{3} X_{0} ) = \sum_{r=1}^{n} \left(
(\Psi^{-1})^{s} -  (\Psi^{-1})^{a} [ ( \Psi^{-1})^{s}]^{-1}(\Psi^{-1})^{a}\right)_{r1}dt^{r}\\
\label{expr-j3}& G(X_{0}, X_{0})= \frac{1}{2} \left( \Psi^{-1} - (\Psi^{-1})^{a} [(\Psi^{-1})^{s}]^{-1} (\Psi^{-1})^{a}
\right)_{11} .  
\end{align}
\end{lem}

\begin{proof}
 Recall that $\mathcal J_{3} = - G^{\mathrm{end}}$
and that $G^{\mathrm{end}}$ is given by  (\ref{gend}), in terms of the Riemannian metric $g$ and $2$-form $b$
of the generalized  K\"{a}hler structure.
The first two relations (\ref{expr-j3}) follow from the expressions of $g$ and $b$ given by (\ref{hermitian-p}),
combined with  (\ref{o}) and (\ref{j-pm}).  
The expression of $G(X_{0}, X_{0})$ is computed from $\mathrm{pr}_{T^{*}} (\mathcal J_{3} X_{0})$,
using $G(X_{0}, X_{0}) =\frac{1}{2} \mathrm{pr}_{T^{*}} (\mathcal J_{3} X_{0}) (\frac{\partial}{\partial t^{1}})$.    
\end{proof}

Using Lemma \ref{cond-toric}, we construct examples of toric generalized K\"{a}hler (non-K\"{a}hler) manifolds, 
with a Hamiltonian Killing vector field $X_{0}$, for which  $\mathcal J_{3} X_{0}$  is a $1$-form, as required by Proposition \ref{gen-applied-1}.  We remark that such examples do not exist in four dimensions:
with $X_{0} = -\frac{\partial}{\partial t^{1}}$, we obtain, 
from the first relation (\ref{expr-j3}), 
\begin{equation}\label{pr-j3}
\mathrm{pr}_{T} (\mathcal J_{3} X_{0}) = -\frac{C_{12}}{\mathrm{det}( \Psi )- C_{12}^{2}} \left(  (\Psi_{12})^s \frac{\partial}{\partial t^{1}} +\Psi_{22}
\frac{\partial}{\partial t^{2}}\right), 
\end{equation}
where $(\Psi_{jk})^s := (\Psi^s)_{jk}$. We deduce    that $\mathrm{pr}_{T} (\mathcal J_{3} X_{0}) =0$ if and only if  $C_{12}=0$, i.e.  the generalized K\"{a}hler manifold is K\"{a}hler. The four dimensional case will be treated separately in the next section.

\begin{exa}\label{pr-t-ex}{\rm i) Let
$\tau : (\mathbb{R}^{>0})^{n} \rightarrow \mathbb{R}$, defined by
$\tau = \sum_{j=1}^{n}\mu^{j}\mathrm{log} (\mu^{j})$, 
be  the symplectic potential of the standard K\"{a}hler metric on $\mathbb{C}^{n}$ ($n\geq 3$). Let   $C:= (C_{ij})\in M_{n} (\mathbb{R})$  be
any skew-symmetric  matrix, 
with $C_{1j}= C_{j1} =0$, for any $1\leq j\leq n$. With these choices, the matrix function $\Psi$ defined by
(\ref{psi-c}) satisfies 
 $(\Psi^{-1})_{1i} = (\Psi^{-1})_{i1}=0$, for any $2\leq i\leq n$,  and the Hamiltonian Killing 
vector field $X_{0}=-\frac{\partial}{\partial t^{1}}$ on the  toric  generalized K\"{a}hler manifold defined by $\tau$ and $C$
satisfies $\mathrm{pr}_{T} (\mathcal J_{3} X_{0})=0.$\

ii) For a six-dimensional toric   generalized K\"{a}hler manifold, with  symplectic potential $\tau$ and skew-symmetric
matrix $C$, the vector field $X_{0}   = -\frac{\partial}{\partial t^{1}}$
satisfies $\mathrm{pr}_{T} (\mathcal J_{3} X_{0})=0$  if and only  if
\begin{equation}\label{cond-psi-0}
C_{12} \frac{\partial \tau}{\partial \mu^{3}} - C_{13} \frac{\partial \tau}{\partial \mu^{2}} + C_{23}
\frac{\partial \tau}{\partial \mu^{1}}
\end{equation} 
depends only on $\mu^{1}$. This is the case if and only if 
\begin{itemize}
\item[a)] 
$C=0$ or
\item[b)] $\tau$ admits a separation of variables $\tau = \tau_1(\mu^1) + \tau_2(\mu^2,\mu^3)$ and $C_{12}=C_{13}=0$. 
\end{itemize}
For example, the symplectic potential
$$
\tau :=\sum_{i=1}^{3} (\mu^{i} + c) \mathrm{log} (\mu^{i} + c) + (c- \mu^{2} - \mu^{3})
\mathrm{log} (c - \mu^{2} - \mu^{3})
$$
defined on 
$$
\Delta := \{ (\mu^{1}, \mu^{2},\mu^{3})\in \mathbb{R}^{3}: \mu^{1}+ c>0,\ \mu^{2} + c>0,\ \mu^{3} +c>0,\ 
\mu^{2} + \mu^{3} <c\}  
$$
(where $c>0$), together  with any skew-symmetric matrix $C$ for which $C_{12} = C_{13}=0$, 
satisfy  b) and hence (\ref{cond-psi-0}).}
\end{exa}

\subsection{The four-dimensional case}\label{4dim-sect}

In this final section we consider 
the setting formed by a four-dimensional toric generalized K\"{a}hler (non-K\"{a}hler) manifold $(M, G, \mathcal J)$, 
defined 
by a symplectic potential $\tau = \tau (\mu^{1}, \mu^{2})$,  skew-symmetric matrix $C = (C_{ij})$ 
and matrix valued function $\Psi = \mathrm{Hess} (\tau ) + C$, as in the previous section.
As $(M, G, \mathcal J)$ is non-K\"{a}hler, $C_{12} \neq 0.$
Let $f$ and $h$ be two  non-vanishing functions on $M$, independent of $t^{1}.$  
Our aim is to describe all such data,
together with the twist data $(X_{0}, F, a)$ with $X_{0} :=  -\frac{\partial}{\partial t^{1}}$,
for which the conditions from Theorem \ref{gen-KK} are satisfied. Remark that $(G, \mathcal J )$, $f$ and $h$ are
$X_{0}$-invariant. In addition, we assume that $h$ depends only on $\mu^{1}$  (the moment map of $X_{0}$). This hypothesis  
is natural, in view of the second relation (\ref{cond-e}).

\begin{prop}\label{dim2} In the above setting,  the conditions from Theorem \ref{gen-KK} are satisfied if and only if:\

i)  the function $f$ depends only on $\mu^{1}$;\

ii) the function $a$ and the curvature form $F$ of the twist data are given by 
$a= k_{0} h^{-2}$ (where $k_{0}\in \mathbb{R}\setminus \{ 0\}$) and  
\begin{equation}\label{curvature}
F = \frac{2ah^{\prime}}{h} d\mu^{1} \wedge dt^{1} + \left( \frac{(\Psi_{12})^{s}}{\Psi_{22}}(\lambda - \frac{2ah^{\prime}}{h})
d\mu^{1} + \lambda d\mu^{2} \right) \wedge dt^{2} ,
\end{equation}
where $\lambda\in C^{\infty}(M)$  is  independent of $t^{1}$;\

iii) the following relations hold: 
\begin{align}
\nonumber& \frac{\partial}{\partial \mu^{2} }\left( \frac{ (\Psi_{12})^{s}}{\Psi_{22}}\right) = -\frac{1}{(f^{2}-1)a}\left( \lambda
+\frac{2 ah^{\prime}f^{2}}{h}\right)\\
\label{k12}& \frac{( \Psi_{12})^{s}}{\Psi_{22}}\frac{\partial \lambda }{\partial \mu^{2}} -\frac{\partial \lambda}{\partial \mu^{1}} 
=\frac{1}{(f^{2}-1)a} \left( \lambda +\frac{2ah^{\prime}f^{2}}{h}\right) \left( \lambda - \frac{2ah^{\prime}}{h}\right) .
\end{align}
\end{prop}

\begin{proof}
We divide the proof of Proposition \ref{dim2} in several steps.
Let $L = L(E_{1}, \epsilon_{1})$ be the $(1,0)$-bundle of $\mathcal J$
and $(J_{\pm},\omega ,  g, b )$  the complex structures, Riemannian metric and $2$-form on $M$ associated
to $(G, \mathcal J )$, as in Section \ref{complex-geometry}. 
As $E_{1} = (T^{1,0}M)^{J_+} + (T^{1,0}M)^{J_{-}}$, cf.\ Section \ref{complex-geometry},
\begin{align}
\nonumber & (T^{1,0}M)^{J_{+}} =\mathrm{span}_{\mathbb{C}} \left\{ v_{j}^{+}:= \frac{\partial}{\partial \mu^{j}} - i \Psi_{kj}
\frac{\partial}{\partial t^{k}},\ j=1,2\right\} \\
\label{def-vj}& (T^{1,0}M)^{J_{-}} =\mathrm{span}_{\mathbb{C}} \left\{ v_{j}^{-}:= \frac{\partial}{\partial \mu^{j}} - i \Psi_{jk}
\frac{\partial}{\partial t^{k}},\ j=1,2\right\} 
\end{align}
and $C_{12}\neq 0$,  we obtain that  $E_{1} = (TM)^{\mathbb{C}}.$    (Note that here and in the following
we are using Einstein's summation convention.)
More precisely, 
\begin{align}
\nonumber& \frac{\partial}{\partial t^{1}} = \frac{i}{2C_{12}} (v_{2}^{+}- v_{2}^{-}),\    
\frac{\partial}{\partial t^{2}} = - \frac{i}{2C_{12}} (v_{1}^{+}- v_{1}^{-}); \\
\nonumber & \frac{\partial}{\partial \mu^{1}} =
(1+ \frac{\Psi_{21}}{2 C_{12}}) v_{1}^{+} -\frac{\Psi_{21}}{2 C_{12}} v_{1}^{-} -\frac{\Psi_{11}}{ 2 C_{12}} v_{2}^{+} 
+\frac{\Psi_{11}}{ 2 C_{12}} v_{2}^{-}; \\
\label{generates}& \frac{\partial}{\partial \mu^{2}}=  \frac{\Psi_{22}}{2 C_{12}} v_{1}^{+} -\frac{\Psi_{22}}{2 C_{12}} v_{1}^{-} + 
( 1- \frac{\Psi_{12}}{ 2 C_{12}} ) v_{2}^{+}  +\frac{\Psi_{12}}{ 2 C_{12}} v_{2}^{-}.
\end{align}

\begin{lem}\label{det-F}i)  The form $\epsilon_{1}$ is given by
\begin{align}
\nonumber  \epsilon_{1} & =  \frac{1}{C_{12}} \left( dt^{1}\wedge dt^{2} - \mathrm{det}(\Psi) d\mu^{1}\wedge d\mu^{2}\right)\\
\nonumber & + i \left(  (1+\frac{\Psi_{21}}{C_{12}}) dt^{1}\wedge d\mu^{1} +\frac{\Psi_{22}}{C_{12}} dt^{1}\wedge d\mu^{2} -
\frac{\Psi_{11}}{C_{12}} dt^{2}\wedge d\mu^{1} \right)  \\
\label{epsilon1} & +  i(1-\frac{\Psi_{12}}{C_{12}}) dt^{2}\wedge d\mu^{2}.
\end{align}
ii) A real $2$-form $F$ satisfies the first relation 
(\ref{cond-e}), with $\epsilon_{1}$ given by 
(\ref{epsilon1}) and $h, a\in C^{\infty}(M)$, if and only if $F$ is of the form
(\ref{curvature}), for a function $\lambda \in C^{\infty}(M).$\

 iii) The relation $ i_{X_{0}} F = -da $ holds if and only if $a=k_{0} h^{-2}$, where $k_{0}\in \mathbb{R}\setminus \{ 0\}$ is arbitrary.
\end{lem}

\begin{proof}
From the definition of $J_\pm$ in Section \ref{complex-geometry} it follows that
\begin{equation}\label{gEq}
L_{1} \cap C_{\pm} = \{ X + (b\pm g)(X),\ X \in (T^{1,0}M)^{J_{\pm}}\},
\end{equation}
 which together with $v_{j}^{\pm} \in (T^{1,0}M)^{J_{\pm}}$ implies that  
$i_{v_{j}^{\pm}} \epsilon_{1} = (b\pm g) (v_{j}^{\pm}).$
From the definition of $v_{j}^{\pm}$ and relations 
(\ref{hermitian-p}),  (\ref{o}) and (\ref{j-pm}),
\begin{align*}
b(v_{j}^{+}) =(\Psi_{kj})^{a} d\mu^{k} - i \Psi_{kj} (\Psi^{rk})^{a} dt^{r},\ 
g( v_{j}^{+}) = (\Psi_{kj})^{s} d\mu^{k} - i \Psi_{kj} (\Psi^{rk})^{s} dt^{r}\\
b(v_{j}^{-}) =(\Psi_{kj})^{a} d\mu^{k} - i \Psi_{jk} (\Psi^{rk})^{a} dt^{r},\ 
g( v_{j}^{-}) = (\Psi_{kj})^{s} d\mu^{k} - i \Psi_{jk} (\Psi^{rk})^{s} dt^{r},
\end{align*}
from where we deduce that
\begin{equation}\label{i-e}
i_{v_{j}^{+}} \epsilon_{1} = \Psi_{kj}d\mu^{k} - i dt^{j},\ i_{v_{j}^{-}} \epsilon_{1} = 
- \Psi_{jk} d\mu^{k} + i dt^{j}.
\end{equation}
Combining  (\ref{generates}) with  (\ref{i-e})
we obtain (\ref{epsilon1}).  Claim ii) follows from (\ref{epsilon1}), by
computing
$$
i_{X_{0}} \epsilon_{1} = -\frac{1}{C_{12}} ( i (\Psi_{k2})^{s} d\mu^{k} + dt^{2})
$$
and  identifying the 
real and imaginary parts in the first relation (\ref{cond-e}).  Claim iii) can be checked  using  the expression
(\ref{curvature}) of $F$ and that $h = h(\mu^{1}).$ Moreover, $k_{0}\neq 0$  as the function $a$ is non-vanishing. 
\end{proof}

With the above preliminary lemma, we now proceed to the proof of  Proposition  \ref{dim2}. 
The first relation (\ref{cond-e}) from the statement of Theorem  \ref{gen-KK} 
is satisfied if and only if $F$ is of the form (\ref{curvature}). From now on, we assume that $F$ is of this form.
Then the function $a$ is given as in claim iii) of Lemma
\ref{det-F}.
Since 
\begin{align*}
\mathrm{pr}_{T} (L_{1} \cap L_{2} \cap \mathcal S^{\perp}_{\mathbb{C}}) = (T^{1,0}M)^{J_{+}} \cap \mathrm{Ker} \{ d\mu^{1}\} \\
\mathrm{pr}_{T} (L_{1} \cap\bar{L}_{2} \cap \mathcal S^{\perp}_{\mathbb{C}}) = (T^{1,0}M)^{J_{-}} \cap \mathrm{Ker}
\{ d\mu^{1}\} 
\end{align*}
are of rank one, generated by $v_{2}^{+}$ and $v_{2}^{-}$ respectively, and 
$E_{1} = (TM)^{\mathbb{C}}$, condition i) from Theorem \ref{gen-KK} is satisfied. 
 We now consider condition ii) from this theorem.

\begin{lem}\label{cond-iv} Condition ii) from Theorem \ref{gen-KK} holds, with
$F$ given by (\ref{curvature}),   if and only if
$f$ depends only on $\mu^{1}$ and the first relation (\ref{k12})  is satisfied. 
\end{lem}

\begin{proof}
Condition ii) from Theorem \ref{gen-KK} holds if and only if
\begin{equation}\label{expr-dim}
\mathcal E_{1}:= [\mathrm{pr}_{T} (\mathcal J_{3} X_{0}), v_{2}^{+} ] + \alpha ( v_{2}^{+}) \mathrm{pr}_{T} ( \mathcal J_{3} X_{0}) 
+\frac{ f^{2} F(\mathrm{pr}_{T} (v_{f}), v_{2}^{+})}{a (f^{2} -1)} X_{0} 
\end{equation}
is  a multiple of $v_{2}^{+}$, and 
\begin{equation}\label{expr-dim}
\mathcal E_{2} := [\mathrm{pr}_{T} (\mathcal J_{3} X_{0}), v_{2}^{-} ] + \alpha ( v_{2}^{-}) \mathrm{pr}_{T} ( \mathcal J_{3} X_{0}) 
- \frac{ f^{2} F(\mathrm{pr}_{T} (v_{if}), v_{2}^{-})}{a (f^{2} -1)} X_{0} 
\end{equation}
is  a multiple of $v_{2}^{-}.$  
From (\ref{pr-j3}) and the definitions of $v_{2}^{\pm}$,  these conditions are equivalent to
$\mathcal E_{1} = \mathcal E_{2} =0$ (as $\mathcal E_{i}$ are linear combinations of $\frac{\partial }{\partial t^{1}}$ and
$\frac{\partial}{\partial t^{2}}$, while $v_{2}^{\pm}$ involve also $\frac{\partial}{\partial \mu^{2}}$) 
or to
\begin{align}
\nonumber& \frac{\partial E_{1}}{ \mu^{2}} - \alpha ( v_{2}^{+}) E_{1} +\frac{ f^{2} F(\mathrm{pr}_{T} (v_{f}), v_{2}^{+})}{
a(f^{2} -1) } =0\\
\nonumber & \frac{\partial E_{1}}{ \mu^{2}} - \alpha ( v_{2}^{-}) E_{1} - \frac{ f^{2} F(\mathrm{pr}_{T} (v_{if}), v_{2}^{-})}{
a(f^{2} -1) } =0\\
\label{123} & \frac{\partial E_{2}}{\partial \mu^{2}} - \alpha ( v_{2}^{+}) E_{2} = \frac{\partial E_{2}}{\partial \mu^{2}} - \alpha ( v_{2}^{-}) E_{2}=0,
\end{align} 
where 
\begin{equation}\label{j-3-dim}
E_{1} := -\frac{C_{12} (\Psi_{12})^{s}}{\mathrm{det}( \Psi )- C_{12}^{2}} ;\ E_{2}:=  -\frac{C_{12}\Psi_{22}}{\mathrm{det}( \Psi )- C_{12}^{2}}
\end{equation}
are the coordinates of $\mathrm{pr}_{T} (\mathcal J_{3} X_{0}).$ 
The third relations (\ref{123}) are equivalent to
\begin{equation}\label{det-f}
\alpha ( \frac{\partial}{\partial t^{1}}) =\alpha (\frac{\partial}{\partial t^{2}} )= 0,\ 
\alpha (\frac{\partial}{\partial \mu^{2}}) =\frac{1}{E_{2}} \frac{\partial E_{2}}{\partial \mu^{2}}.
\end{equation}
Let $H:= \frac{f^{2} -1}{ f^{2} G(X_{0}, X_{0})}$, so that $\alpha := - \frac{dH}{H}.$ 
The first two relations (\ref{det-f}) are equivalent to $f$ - independent of $t^{i}$ 
(recall that $G(X_{0}, X_{0})$ is independent of $t^{i}$) 
and the third relation (\ref{det-f}) is equivalent to $HE_{2}$ - independent of $\mu^{2}$. 
From Lemma \ref{cond-toric}, 
\begin{align}
\nonumber& \mathrm{pr}_{T^{*}} (\mathcal J_{3} X_{0}) = \frac{1}{ \mathrm{det} (\Psi ) - C_{12}^{2}}
(\Psi_{22} dt^{1} - (\Psi _{12} )^{s}dt^2 )\\
\label{G-dim2}& G(X_{0}, X_{0}) = \frac{ \Psi_{22}}{2( \mathrm{det} (\Psi ) - C_{12}^{2})}.
\end{align}
and from the second relation (\ref{G-dim2}) we obtain that $HE_{2}= -\frac{2 C_{12}(f^{2}-1)}{f^{2}}$.
We proved that relations (\ref{det-f}) hold if and only if   $f$ depends only on $\mu^{1}$, as required.\

We now study the first two relations (\ref{123}).   
From (\ref{det-f}), $\alpha (v_{2}^{+}) = \alpha (v_{2}^{-}) = \frac{1}{E_{2}} \frac{\partial E_{2}}{\partial
\mu^{2}}$ and the first two relations (\ref{123})  are equivalent to
\begin{align}
\nonumber& \mathrm{Im} F (\mathrm{pr}_{T} ( v_{f}), v_{2}^{+}) = 0,\  
\mathrm{Im} F (\mathrm{pr}_{T} ( v_{if}), v_{2}^{-}) = 0;\\
\nonumber & F(\mathrm{pr}_{T} (v_{f}), v_{2}^{+}) = - F(\mathrm{pr}_{T} (v_{if}), v_{2}^{-});\\
\label{123-f} & \frac{\partial}{\partial \mu^{2}} \left(\frac{E_{1}}{E_{2}}\right) 
= -\frac{f^{2} \mathrm{Re} F(\mathrm{pr}_{T} (v_{f}) , v_{2}^{+})}{ a(f^{2} -1)E_{2}}.  
\end{align}
To study these relations, we need to find $\mathrm{pr}_{T} (v_{f})$ and $\mathrm{pr}_{T}(v_{if})$. 
Recall that $\mathrm{pr}_{T} (\mathcal J_{3} X_{0})$ was computed in 
(\ref{pr-j3}) and $\mathrm{pr}_{T} (\mathcal J_{2} X_{0}) =0.$ 
We now compute  
$\mathrm{pr}_{T} (\mathcal J X_{0})$. Since $X_{0} + G^{\mathrm{end}} (X_{0}) \in C_{+}$, from the definition of $J_{+}$ we obtain
$$
J_{+} (X_{0} + \mathrm{pr}_{T} G^{\mathrm{end}} (X_{0})) = \mathrm{pr}_{T} (\mathcal J X_{0} + \mathcal J_{2} X_{0}) 
= \mathrm{pr}_{T} (\mathcal J X_{0}),
$$
i.e. 
$\mathrm{pr}_{T} (\mathcal J X_{0}) = J_{+} ( X_{0}-\mathrm{pr}_{T} (\mathcal J_{3} X_{0})).$
From this relation and (\ref{pr-j3}),  we obtain

\begin{equation}
\mathrm{pr}_{T} (\mathcal J X_{0})= \frac{1}{\mathrm{det} (\Psi ) -C_{12}^{2}} \left( \Psi_{22} 
\frac{\partial}{\partial \mu^{1}} - ( \Psi_{12})^{s}\frac{\partial}{\partial \mu^{2}} \right) .
\end{equation}
The first two relations (\ref{123-f}) are equivalent to 
\begin{align*}
F(\mathrm{pr}_{T}(\mathcal J X_{0}), \frac{\partial}{\partial \mu^{2}}) + \Psi_{k2}F(X_{0} -\frac{1}{f^{2}}
\mathrm{pr}_{T} (\mathcal J_{3} X_{0}) , \frac{\partial}{\partial t^{k}}) =0\\
F(\mathrm{pr}_{T}(\mathcal J X_{0}), \frac{\partial}{\partial \mu^{2}}) + \Psi_{2k}F(X_{0} + \frac{1}{f^{2}}
\mathrm{pr}_{T} (\mathcal J_{3} X_{0}) , \frac{\partial}{\partial t^{k}}) =0
\end{align*}
and are satisfied (from the expressions of $F$, $\mathrm{pr}_{T} (\mathcal J X_{0})$ and
$\mathrm{pr}_{T} (\mathcal J_{3} X_{0})$),  and the third relation (\ref{123-f}) is equivalent to
$$
2\Psi_{22} F (\mathrm{pr}_{T} (\mathcal J X_{0}) , \frac{\partial}{\partial t^{2}}) + (\Psi_{12}+ \Psi_{21}) 
F( \mathrm{pr}_{T} (\mathcal J X_{0}),\frac{\partial}{\partial t^{1}}) =0
$$
and is satisfied as well (again, from the expressions of $F$ and $\mathrm{pr}_{T} (\mathcal J X_{0})$). The fourth relation 
(\ref{123-f}) becomes the first relation (\ref{k12}).  
\end{proof}

According to Lemma \ref{cond-iv}, we assume that $f$ depends only on $\mu^{1}$ and that the first  relation
(\ref{k12}) is satisfied. Since  $f$, $h$  depend only on $\mu^{1}$
and $a = k_{0} h^{-2}$,   the function  $af^{2} h^{2}$ depends only on $\mu^{1}$ as well and
$X( af^{2} h^{2} ) =0$, for any $X\in \mathrm{pr}_{T} (\mathcal S^{\perp}).$ Therefore, condition iii) from
Theorem \ref{gen-KK} is satisfied.  Condition iv) from this theorem is also satisfied: the first relation
(\ref{cond-e}) follows  from Lemma \ref{det-F} ii) and the second relation (\ref{cond-e}) is a consequence of
the fact that $h$ depends only on $\mu^{1}$.  It remains to study condition iv) of Theorem \ref{gen-KK}.
Since 
$$
\mathrm{pr}_{T}( L_{2} \cap \mathcal S^{\perp}_{\mathbb{C}})=\mathrm{pr}_{T} (L_{1} \cap L_{2} \cap \mathcal S^{\perp}_{\mathbb{C}}) + \mathrm{pr}_{T} (\bar{L}_{1} \cap L_{2} \cap \mathcal S^{\perp}_{\mathbb{C}})
$$ 
is generated by $v_{2}^{+}$ and 
$\overline{v_{2}^{-}}$, relation (\ref{cond-de}) holds if and only if it holds when applied to the pair 
$(v_{2}^{+}, \overline{v_{2}^{-}}).$ From 
the first relation (\ref{G-dim2}), 
$$
d\left(  \mathrm{pr}_{T^{*}}(\mathcal J_{3} X_{0}) \right) (v_{2}^{+}, \overline{ v_{2}^{-}})
= -2i \frac{\partial}{\partial \mu^{2}} \left(\frac{(\Psi_{12})^{s}}{\Psi_{22}}\right) \frac{ (\Psi_{22})^{2}}{ \mathrm{det}(\Psi ) - C_{12}^{2}}.
$$
Computations similar to those from the proof of Lemma \ref{cond-iv} show that relation (\ref{cond-de}) is equivalent to the first relation (\ref{k12}). 
Moreover, a  straightforward computation shows that $F$, defined by
(\ref{curvature}), is closed, if and only if 
the second relation (\ref{k12}) holds and $\lambda$ is independent of $t^{1}.$
This concludes the proof of Proposition 
\ref{dim2}.
\end{proof}

We now consider the setting of Proposition \ref{dim2} with $\lambda:=\lambda_{0}$ 
constant. Under this assumption, relations (\ref{k12}) become
\begin{align}
\nonumber& (1- f^{2})a  \frac{\partial}{\partial \mu^{2}} \left( \frac{(\Psi_{12})^{s}}{\Psi_{22}}\right)
= \lambda_{0} +\frac{2ah^{\prime} f^{2}}{h}\\
\label{k12-prime}& (\lambda_{0} +\frac{2ah^{\prime}f^{2}}{h}) ( \lambda_{0}-\frac{2ah^{\prime}}{h}) =0.
\end{align}
The functions $h$ and  $a=k_{0}h^{-2}$  depend only on $\mu^{1}.$ By Proposition \ref{dim2}, $f$ has to depend
only on $\mu^{1}$. From our hypothesis of the generalized KK correspondence, $f^{2}-1$ is non-vanishing.
From the first relation (\ref{k12-prime}), we obtain that  $\frac{\partial}{\partial \mu^{2}} \left( \frac{(\Psi_{12})^{s}}{\Psi_{22}}\right)$ has to depend only on $\mu^{1}$.
On the other hand, from the second relation (\ref{k12-prime})  we distinguish two cases, according to the vanishing of the two factors from its
left hand side.  Corollary \ref{corolar-1} and \ref{corolar-2} below correspond to   $\lambda_{0} +\frac{2ah^{\prime}f^{2}}{h}=0$, respectively to $ \lambda_{0}=\frac{2ah^{\prime}}{h}$. Their proofs 
are straightforward, from Proposition \ref{dim2}, and will be omitted.

\begin{cor}\label{corolar-1}
Consider a $4$-dimensional toric generalized K\"{a}hler structure, determined by a matrix valued function $\Psi$, such that
$\frac{\partial}{\partial \mu^{2}} \left( \frac{(\Psi_{12})^{s}}{\Psi_{22}}\right) =0$. Let $h = h(\mu^{1})$ be any smooth 
non-vanishing function with $h^{\prime}$ non-vanishing
and $\lambda_{0}, k_{0}\in \mathbb{R}\setminus \{ 0\}$,  such that $\frac{\lambda_{0}h^{3}}{ k_{0} h^{\prime}}$ is
(at any point) negative and different from $-1$. Choose
\begin{equation}\label{a-f-4}
 X_{0}:= -\frac{\partial}{\partial t^{1}},\  a:= k_{0} h^{-2},\ f^{2}:=  - \frac{\lambda_{0}h^{3}}{ 2k_{0} h^{\prime}}
\end{equation}
and $F$ of the form (\ref{curvature}), with $\lambda := \lambda_{0}$. Then all conditions from
Theorem \ref{gen-KK} are satisfied.
\end{cor}

\begin{cor}\label{corolar-2} 
Consider a $4$-dimensional toric generalized K\"{a}hler structure, determined by a matrix valued function $\Psi$, such that 
$\frac{\partial}{\partial \mu^{2}} \left( \frac{(\Psi_{12})^{s}}{\Psi_{22}}\right)$ depends only on $\mu^{1}.$  
Choose $X_{0}:= -\frac{\partial}{\partial t^{1}}$, 
\begin{equation}\label{a-f-f}
a:= k_{0}k_{1} -\lambda_{0} \mu^{1},\ h^{2}: = \frac{ - k_{0}}{\lambda_{0} \mu^{1} - k_{0} k_{1}},\ 
 f^{2}: =\frac{(\mu^{1} - \frac{k_{0}k_{1}}{\lambda_{0}}) \frac{\partial}{\partial \mu^{2}} \left( \frac{(\Psi_{12})^{s}}{\Psi_{22}}\right)+1}{ (\mu^{1} - \frac{k_{0}k_{1}}{\lambda_{0}})\frac{\partial}{\partial \mu^{2}} \left( \frac{(\Psi_{12})^{s}}{\Psi_{22}}\right)
-1},
\end{equation}
and $F$ given by (\ref{curvature}) with $\lambda := \lambda_{0}$. Then all conditions from 
Theorem \ref{gen-KK} are satisfied. Above $k_{0},\lambda_{0}\in \mathbb{R}\setminus \{ 0\}$, 
$k_{1}\in \mathbb{R}$, and 
we assume that 
the defining expression for 
$a$ in (\ref{a-f-f})  is non-vanishing, while the defining expressions  for  
$h^{2}$ and $f^{2}$ are positive. 
 
\end{cor}

The conditions from Proposition \ref{dim2} (and  Corollaries \ref{corolar-1} and \ref{corolar-2}),
on the toric generalized K\"{a}hler structure, involve only the 
symmetric part of $\Psi$, i.e. the symplectic potential $\tau .$
There are various interesting symplectic potentials 
 for which   $\frac{\partial}{\partial \mu^{2}} ( \frac{\tau_{12}}{\tau_{22}})$ depends only on
$\mu^{1}$, as required by Corollaries \ref{corolar-1} and \ref{corolar-2}  (here $\tau_{ij}:= \frac{\partial^{2} \tau}{\partial \tau^{i}\tau^{j}}$):

\begin{exa}\label{exx-last}{\rm  
Let  $c\in \mathbb{R}$ and  $k\in \mathbb{R}^{>0}$. The symplectic potential 
 $\tau (\mu^{1}, \mu^{2}) := e^{\mu^{1}} ( e^{\mu^{2} + c} + k)$ defined on $\mathbb{R}^{2}$
satisfies  $\frac{\partial}{\partial \mu^{2}} \left( \frac{(\Psi_{12})^{s}}{\Psi_{22}}\right)=0$.}
\end{exa}

\begin{exa}\label{cp}{\rm   {\bf (Symplectic potential of the Fubini-Study metric).} Consider the symplectic potential of the Fubini Study metric on $\mathbb{C}P^{2}$, of constant holomorphic sectional curvature equal to $2$:
\begin{equation}
\tau (\mu^{1}, \mu^{2}):=\sum_{i=1}^{2}  (\mu^{i} + \frac{1}{3}) \mathrm{log} ( \mu^{i} +\frac{1}{3}) 
 + (\frac{1}{3} -\sum_{i=1}^{2} \mu^{i})
\mathrm{log} (\frac{1}{3} - \sum_{i=1}^{2} \mu^{i} ) 
\end{equation}
defined on $\Delta := \{ (\mu^{1}, \mu^{2}) \in \mathbb{R}^{2}:\ \mu^{1} + 1/3>0,\ \mu^{2} + 1/3 >0,\ 
\mu^{1} +\mu^{2} < 1/3\}. $  Then 
$\frac{\partial}{\partial \mu^{2}} ( \frac{\tau_{12}}{\tau_{22}})=\frac{1}{ 2/3 - \mu^{1}}$
depends only on $\mu^{1}$.  
In the notation of Corollary \ref{corolar-2},  choose
$k_{0}, k_{1} \in \mathbb{R}$ such that
$k_{0} <0$ and $k_{0} k_{1} < -\frac{4}{3}$, and let $\lambda_{0} :=1.$ 
Then 
\begin{equation}\label{h-f}
h^{2} = -\frac{k_{0}}{\mu^{1} - k_{0}k_{1}},\ f^{2} = \frac{2/3 - k_{0}k_{1}}{ 2\mu^{1} - ( 2/3 + k_{0}k_{1})}.
\end{equation}
The choice of  $k_{0}, k_{1}$ ensure that
the defining expressions for  $h^{2}$ and $f^{2}$ in (\ref{h-f}) are positive on $\Delta$. Also $a= -\mu^{1} +k_{0}k_{1}$ is negative
on $\Delta$ 
(in particular, non-vanishing).}\end{exa}

\begin{exa}\label{hz}{\rm {\bf (Symplectic potential of admissible metrics on Hirzebruch surfaces).} The $k$-Hirzebruch surface $\mathbb{F}_{k}:= 
\mathbb{P} ( 1 \oplus \mathcal O (-k))$ (where $k$ is a positive integer) admits a class of  toric  K\"{a}hler metrics (called admissible), 
which generalize Calabi's extremal K\"{a}hler metrics \cite{calabi}. 
A detailed description of these metrics can be found in \cite{apostolov, paul, paul-book}.  They are defined in terms of a smooth  function (the momentum profile) 
$\Theta : [a,b] \rightarrow \mathbb{R}$, positive on $(a,b)$, which satisfies the boundary conditions
$$
\Theta (a) =0,\  \Theta (b) = 0,\ \Theta^{\prime} (a) =2,\ \Theta^{\prime} (b) =-2
$$
(where $b>a>0$). 
The symplectic potential 
$$
\tau :  \{ (\mu^{1}, \mu^{2})\in \mathbb{R}^{2}:\ 
\mu^{2} > 0,\ k \mu^{1} - \mu^{2} >0, - \mu^{1} + b>0,\ \mu^{1} - a>0\}\rightarrow \mathbb{R}
$$
of the admissible K\"{a}hler metric with momentum profile $\Theta$ satisfies
$$
\tau_{11}= \frac{1}{\Theta (\mu^{1})} +\frac{ k \mu^{2}}{2\mu^{1} ( k \mu^{1} -\mu^{2})},
\tau_{12}= \frac{- k}{2( k\mu^{1} - \mu^{2})},\ \tau_{22} = \frac{k \mu^{1}}{2\mu^{2} ( k\mu^{1} - \mu^{2})}.
$$
We obtain that $\frac{\partial}{\partial \mu^{2}} (\frac{\tau_{12}}{\tau_{22}})= - \frac{1}{\mu^{1}}$
depends only on $\mu^{1}$. In the notation of Corollary \ref{corolar-2}, let
$\lambda_{0} :=1$. Then $f$ and $h$
are given by
\begin{equation}\label{f-g-2}
f^{2} = \frac{-k_{0} k_{1}}{2\mu^{1}- k_{0} k_{1}},\ h^{2} = \frac{- k_{0}}{\mu^{1} - k_{0}k_{1}}.
\end{equation}
Choosing  $k_{0}, k_{1}\in \mathbb{R}^{>0}$ 
with $k_{0}k_{1} >2b$, we obtain that the defining expressions for $f^{2}$,  $h^{2}$ 
in (\ref{f-g-2}) and $a= k_{0}k_{1} - \mu^{1}$ are positive.}
\end{exa}

\begin{rem}\label{future}{\rm 
It would be interesting  to study the properties of the  generalized K\"{a}hler structures $[\tau_{h} (G^{\prime}, \mathcal J )]_{W}$ produced by 
Proposition \ref{dim2}.  Here we only remark that 
they are not of symplectic type. 
The manifold $W$ inherits the vector fields $(\frac{\partial}{\partial t^{1}})_{W}$ and  $(\frac{\partial}{\partial t^{2}})_{W}$, which commute   (from (\ref{s2}), as 
$\frac{\partial }{\partial t^{1}}$ and $\frac{\partial }{\partial t^{2}}$ commute and $F(\frac{\partial}{\partial t^{1}},
\frac{\partial}{\partial t^{2}} )=0$). However, the abelian Lie algebra generated by   $(\frac{\partial}{\partial t^{1}})_{W}$ and  $(\frac{\partial}{\partial t^{2}})_{W}$  does not necessarily preserve the generalized K\"{a}hler structure  $[\tau_{h} (G^{\prime}, \mathcal J )]_{W}$.  Further investigations in this direction are needed.} 
\end{rem}

\noindent 
{\bf Acknowledgements}: We warmly thank Paul Gauduchon for sharing 
with us his unpublished manuscripts \cite{paul,paul-1,paul-book} and for drawing 
our attention to reference \cite{boulanger}. V.C. was
partly supported by the German Science Foundation (DFG) under the Research Training 
Group 1670 ``Mathematics inspired by String Theory". Part of this work (Sections 1-3) was 
done during L.D.'s  Humboldt Research Fellowship at 
the University of Hamburg. The remaining sections were developed after the end of the 
Humboldt fellowship, when she was supported by a grant of the Ministery of Research and Innovation,
CNCS-UEFISCDI, project no. PN-III-P4-ID-PCE-2016-0019, within PNCDI III.
She thanks the Humboldt Foundation and CNCS-UEFISCDI for financial support and 
University of Hamburg for hospitality.

 V.\ Cort\'es: vicente.cortes@uni-hamburg.de\

Department of Mathematics 
and Center for Mathematical Physics, University of Hamburg,
Bundesstra{\ss}e 55, D-20146 Hamburg, Germany.

L. David: liana.david@imar.ro\

``Simion Stoilow" Institute of Mathematics of the Romanian Academy, 
Calea Grivitei 21, Sector 1, Bucharest, Romania

\end{document}